\documentclass[preprint, 12pt, sort&compress]{elsarticle}

\usepackage{latexsym}
\usepackage{amsfonts}
\usepackage{graphicx,amssymb,natbib}
\usepackage{amsmath}
\usepackage{amssymb}
\usepackage[mathscr]{eucal}
\usepackage{enumerate}
\usepackage{amsthm}

\usepackage{bm}

\usepackage{color, soul}

\begin{document}

\newtheorem{tm}{Theorem}[section]
\newtheorem{pp}[tm]{Proposition}
\newtheorem{lm}[tm]{Lemma}
\newtheorem{df}[tm]{Definition}
\newtheorem{tl}[tm]{Corollary}
\newtheorem{re}[tm]{Remark}
\newtheorem{eap}[tm]{Example}

\newcommand{\pof}{\noindent {\bf Proof} }
\newcommand{\ep}{$\quad \Box$}

\newcommand{\al}{\alpha}
\newcommand{\be}{\beta}
\newcommand{\var}{\varepsilon}
\newcommand{\la}{\lambda}
\newcommand{\de}{\delta}
\newcommand{\str}{\stackrel}
\newcommand{\rmn}{\romannumeral}

\renewcommand{\proofname}{\bf Proof}

\allowdisplaybreaks

\begin{frontmatter}

\title{Properties of several metric spaces of fuzzy sets
\tnoteref{usc}
 }
\tnotetext[usc]{Project supported by
 Natural Science Foundation of Fujian Province of China(No. 2020J01706)}
\author{Huan Huang}
 \ead{hhuangjy@126.com }
\address{Department of Mathematics, Jimei
University, Xiamen 361021, China}

\date{}

\begin{abstract}
This paper discusses the properties the spaces of fuzzy sets in a metric space
equipped with
 the endograph metric and the sendograph metric, respectively.
We first give some relations among
  the
endograph metric, the sendograph metric and the $\Gamma$-convergence, and then investigate the
level characterizations of the endograph metric and the
 $\Gamma$-convergence.
By using the above results, we give some relations among the endograph metric, the sendograph metric, the supremum metric and the $d_p^*$ metric, $p\geq 1$.
On the basis of the above results,
we
present the characterizations of total boundedness, relative compactness and compactness
in the space
of fuzzy sets whose $\al$-cuts are
compact when $\al>0$ equipped with the endograph metric,
and
in
the space
of
 compact support fuzzy sets equipped with the sendograph metric, respectively.
 Furthermore, we
give
completions of
these metric spaces, respectively.

\end{abstract}


\end{frontmatter}

\date{}

\section{Introduction}

Fuzzy set is a fundamental tool to investigate fuzzy phenomenon
\cite{da, du, wu, garcia, wa, wang2, gong, grzegorzewski2}.
A fuzzy set can be identified with its endograph. Also, a fuzzy set can be identified with its sendograph.
So convergence structures
on fuzzy sets can be defined on their endographs or sendographs.
The the endograph metric $H_{\rm end}$ convergence, the sendograh metric $H_{\rm send}$ convergence and the
$\Gamma$-convergence are this kind of convergence structures.
These three convergence structures
are related to each other
\cite{rojas, huang}.

The endograph metric on fuzzy sets
is shown to
have significant advantages
 \cite{kloeden, kupka}.
The sendograph metric has attracted deserving attentions \cite{fan,greco}.
Compactness is one of the central concepts in topology and analysis and is useful in applications (see \cite{kelley, wa}).
There is a lot of work devoted to characterizations
of
compactness in various fuzzy set spaces endowed with different topologies \cite{fan,greco,greco3,huang,huang9,roman,trutschnig,wu2}.
 It is natural
to consider what the
 completion of a metric space is. The recent results on completions of fuzzy set spaces include \cite{huang, huang9}.

In \cite{huang},
we presented the relations and level characterizations of the endograph metric $H_{\rm end}$ and the $\Gamma$-convergence.
Based on this, we
have given the characterizations of total boundedness, relative compactness and compactness
of
fuzzy set spaces equipped with the endograph metric $H_{\rm end}$.
We also gave the completions of fuzzy set spaces
according to the endograph metric $H_{\rm end}$.

The common fuzzy sets used
in theoretical research and practical applications
are fuzzy sets in a metric space whose
$\al$-cuts are nonempty compact sets when $\al>0$.
Common compact fuzzy sets are common fuzzy sets whose support sets are compact.
Throughout this paper, we suppose that $X$ \emph{is a nonempty set and} $d$ \emph{is the metric on} $X$.
For simplicity,
we also use $X$
to denote the \emph{metric space} $(X, d)$.
The symbols $F^1_{USCG}(X)$ and $F^1_{USCB} (X)$ are used to denote
the family of common fuzzy sets in $X$ and
the family of common compact fuzzy sets in $X$, respectively.
We
use $F^1_{USC}(X)$ to denote the family of normal and upper semi-continuous fuzzy sets in $X$.
$F^1_{USCB} (X)$ is a subset of $F^1_{USCG}(X)$.
$F^1_{USCG}(X)$ is a
subset of $F^1_{USC}(X)$.

The results in \cite{huang} are obtained
 on the realm of fuzzy sets in the $m$-dimensional Euclidean space $\mathbb{R}^m$ ($\mathbb{R}^1$ is also written as $\mathbb{R}$).
$\mathbb{R}^m$ is a special type
of metric space.
Of course, it is worth to study the
fuzzy sets in a metric space
\cite{jarn, greco, greco3}.
In this paper, the results are obtained on the realm of fuzzy sets in a general metric space $X$.
We mainly discuss
 $H_{\rm end}$ metric and $H_{\rm send}$ metric on $F^1_{USC}(X)$
including
the relations among $H_{\rm end}$ metric, $H_{\rm send}$ metric and other convergence structures,
and properties of $H_{\rm end}$ metric and $H_{\rm send}$ metric.

The first part of this paper is devoted to the relations among the $H_{\rm end}$ metric, the $H_{\rm send}$ metric and the
$\Gamma$-convergence, the level characterizations of the endograph metric $H_{\rm end}$ and the $\Gamma$-convergence, and
the relations among the supremum metric $d_\infty$, the $H_{\rm end}$ metric, the $H_{\rm send}$ metric and
the $d_p^*$ metric.
The $d_p^*$ metric is an expansion of the $L_p$-type $d_p$ metric on $F^1_{USC}(X)$.

To aid discussion, we introduce the sets $P^1_{USC}(X)$ and $P^1_{USCB}(X)$.
An element of $P^1_{USC} (X)$ (respectively, $P^1_{USCB} (X)$) is a subset of
$X\times [0,1]$, which
is equal to or in a specific way slightly larger than
  the sendograph of a certain fuzzy set in $F^1_{USC} (X)$ (respectively, $F^1_{USCB} (X)$).
$F^1_{USC}(X)$ and $F^1_{USCB}(X)$ can be viewed as the subsets of $P^1_{USC}(X)$ and $P^1_{USCB}(X)$, respectively.
$P^1_{USCB}(X)$ is a subset of $P^1_{USC}(X)$.

We
define the
$H_{\rm send}$ distance and the $H_{\rm end}$ distance on $P^1_{USC}(X)$,
and give the relations among
the $H_{\rm send}$ distance, the $H_{\rm end}$ distance and the Kuratowski convergence on $P^1_{USC}(X)$.
Then, as corollaries, we obtain the
 relations among
the $H_{\rm send}$ metric, the $H_{\rm end}$ metric and the $\Gamma$-convergence on $F^1_{USC}(X)$.

We discuss the level characterizations
of the $\Gamma$-convergence and the
endograph metric $H_{\rm end}$ on
fuzzy sets in $F^1_{USC}(X)$.
It is shown that
under some conditions,
the $\Gamma$-convergence of fuzzy sets can be decomposed to
the Kuratowski convergence of certain $\al$-cuts,
and
the $H_{\rm end}$ metric convergence of fuzzy sets
can be decomposed to
the Hausdorff metric convergence of certain $\al$-cuts.

The understanding of the relations among the $H_{\rm end}$ metric, the $H_{\rm send}$ metric and the $\Gamma$-convergence
 is beneficial for the understanding of themselves.
The
 level characterizations
 help to study these three convergence structures on fuzzy sets by using
the
properties of the corresponding $\al$-cuts.

A
 $H_{\rm send}$ metric convergent sequence
must be
a
 $H_{\rm end}$ metric convergent sequence.
A
 $H_{\rm end}$ metric convergent sequence must be a $\Gamma$-convergent sequence.
So the
knowledge of the $\Gamma$-convergent sequences can help us to analyse the properties of the $H_{\rm end}$ convergent sequences and the $H_{\rm send}$ convergent sequences.
For this reason,
we
 give the level characterizations of a $\Gamma$-convergent sequence in this paper.

Based on the results in the first part, we give the other results of this paper.
 The second part of this paper
is devoted to the characterizations of total boundedness, relative compactness and compactness
in $(F^1_{USCG} (X), H_{\rm end})$ and $(F^1_{USCB} (X), H_{\rm end})$, respectively.
Here we mention
that
the characterization of
relative compactness
in
$(F^1_{USCB} (X), H_{\rm send})$ has already been given by Greco \cite{greco}.

The total boundedness is the key property of compactness in metric space.
We show that a set $U$ in $(F^1_{USCG} (X), H_{\rm end})$ is totally bounded (respectively, relatively compact)
if and only if
for each $\al\in (0,1]$,
the union of all the $\al$-cuts of $u\in U$ is totally bounded (respectively, relatively compact) in $(X,d)$.
We also show that
 a set $U$ in $(F^1_{USCB} (X), H_{\rm send})$ is totally bounded
if and only if
the union of all the $0$-cuts of $u\in U$ is totally bounded in $(X, d)$ (see Section 7).
These results indicate that
for a set $U$ in $(F^1_{USCG} (X), H_{\rm end})$ or $(F^1_{USCB} (X), H_{\rm send})$,
the total boundedness, relative compactness and compactness
of $U$
are closely related to
 the total boundedness, relative compactness and compactness    of the union of all the $\al$-cuts of $u\in U$ in $(X,d)$, respectively.

We point out that some part of the proof of the characterizations in this paper is similar to the corresponding part in \cite{huang}.
But in general, since a set in $X$ need not have the properties of the set in $\mathbb{R}^m$,
 the proof of the conclusions in this paper requires deep understandings of the problem.

The third part is devoted to the completions of several common fuzzy set spaces under the $H_{\rm end}$ metric and the $H_{\rm send}$ metric, respectively.

Let
 $\widetilde{X}$ denote the completion of $X$.
We show that the
space $(P^1_{USCB} (\widetilde{X}), H_{\rm send})$ is a completion of the fuzzy set space $(F^1_{USCB} (X), H_{\rm send})$.
Then
we show that $(F^1_{USCG} (\widetilde{X}), H_{\rm end})$
is a completion of $(F^1_{USCB} (X), H_{\rm end})$.
So, of course,
 $(F^1_{USCG} (\widetilde{X}), H_{\rm end})$ is also a completion of $(F^1_{USCG} (X), H_{\rm end})$.

These conclusions indicate that in the case of the $H_{\rm end}$ metric, a completion of
 the space of common compact fuzzy set in $X$ is the space of common fuzzy set in $\widetilde{X}$;
in the case of the $H_{\rm send}$ metric, a completion of
 the space of common compact fuzzy set in $X$
is a metric space in which the space of common compact fuzzy set in $\widetilde{X}$
can be isometrically embedded,
and
each element of which
is a nonempty compact set in $\widetilde{X}\times[0,1]$.

The conclusions for the completions of the spaces of fuzzy set in $X$
given in this paper apply to not only the cases that $X$ is a complete metric space
but also the cases that $X$ is an incomplete metric space.

The remainder of this paper is organized as follows.
In Section \ref{bas}, we recall and give some basic notions and fundamental results related to fuzzy sets
and
convergence structures on them.
 In Section \ref{fsm}, we discuss
the properties and relations of $H_{\rm send}$, $H_{\rm end}$ and Kuratowski convergence on
 $P^1_{USC} (X)$.
Based on this, we give some
 relations among
$H_{\rm end}$, $H_{\rm send}$
and $\Gamma$-convergence on $F^1_{USC} (X)$.
In Sections \ref{lcg} and \ref{lce}, we investigate the level characterizations
of
the $\Gamma$-convergence and the $H_{\rm end}$ convergence, respectively.
By using the above results, Section \ref{rem} discusses some relations
among the $d_\infty$ metric, the $d_p^*$ metric, the $H_{\rm end}$ metric
and the $H_{\rm send}$ metric.
In Section \ref{cmfuzzy},
on the basis of the conclusions in previous sections, we give characterizations of total boundedness, relative compactness and compactness
in $(F^1_{USCG} (X), H_{\rm end})$ and $(F^1_{USCB} (X), H_{\rm send})$, respectively.
In
Section 8, we give
completions of
$(F^1_{USCG} (X), H_{\rm end})$ and $(F^1_{USCB} (X), H_{\rm send})$, respectively.
At last, we draw the conclusions in Section 9.

\section{Fuzzy sets and convergence structures on them} \label{bas}

In this section, we recall and give some basic notions and fundamental results related to fuzzy sets
and
convergence structures on them.
Readers
can refer to \cite{wu, du,kle} for related contents.

A \emph{fuzzy set} $u$ in $X$ can be seen as a function $u:X \to [0,1]$.
A
subset $S$ of $X$ can be seen as a fuzzy set in $X$. If there is no confusion,
 the fuzzy set corresponding to $S$ is often denoted by $\chi_{S}$; that is,
\[ \chi_{S} (x) = \left\{
                    \begin{array}{ll}
                      1, & x\in S, \\
                      0, & x\in X \setminus S.
                    \end{array}
                  \right.
\]
For simplicity,
for
$x\in X$, we will use $\widehat{x}$ to denote the fuzzy set  $\chi_{\{x\}}$ in $X$.
In this paper, if we want to emphasize a specific metric space $X$, we will write the fuzzy set corresponding to $S$ in $X$ as
$S_{F(X)}$, and the fuzzy set corresponding to $\{x\}$ in $X$ as $\widehat{x}_{F(X)}$.

The symbol $F(X)$ is used
to
denote the set of
all fuzzy sets in $X$.
For
$u\in F(X)$ and $\al\in [0,1]$, let $\{u>\al \} $ denote the set $\{x\in X: u(x)>\al \}$, and let $[u]_{\al}$ denote the \emph{$\al$-cut} of
$u$, i.e.
\[
[u]_{\al}=\begin{cases}
\{x\in X : u(x)\geq \al \}, & \ \al\in(0,1],
\\
{\rm supp}\, u=\overline{    \{ u > 0 \}    }, & \ \al=0,
\end{cases}
\]
where $\overline{S}$
denotes
the topological closure of $S$ in $(X,d)$.

For
$u\in F(X)$,
define
\begin{gather*}
{\rm end}\, u:= \{ (x, t)\in  X \times [0,1]: u(x) \geq t\},
\\
{\rm send}\, u:= \{ (x, t)\in  X \times [0,1]: u(x) \geq t\} \cap  ([u]_0 \times [0,1]).
\end{gather*}
 $
{\rm end}\, u$ and ${\rm send}\, u$
 are called the \emph{endograph} and the \emph{sendograph} of $u$, respectively.

The symbol $K(X)$ and
 $C(X)$ are used
to
 denote the set of all nonempty compact subsets of $X$ and the set of all nonempty closed subsets of $X$, respectively.
The symbol $\mathbb{N}$ is used to denote the set of all positive integers.

Let
$F^1_{USC}(X)$
denote
the set of all normal and upper semi-continuous fuzzy sets $u:X \to [0,1]$,
i.e.,
$$F^1_{USC}(X) :=\{ u\in F(X) : [u]_\al \in  C(X)  \  \mbox{for all} \   \al \in [0,1]   \}.  $$

We introduce some subclasses of $F^1_{USC}(X)$, which will be discussed in this paper.
Define
\begin{gather*}
F^1_{USCB}(X):=\{ u\in  F^1_{USC}(X): [u]_0 \in K(X) \},
\\
F^1_{USCG}(X):=\{ u\in  F^1_{USC}(X): [u]_\al \in K(X) \ \mbox{for all} \   \al\in (0,1] \}.
 \end{gather*}
Clearly,
 $$F^1_{USCB}(X) \subseteq  F^1_{USCG}(X)   \subseteq  F^1_{USC}(X).$$

The following representation theorems for
$F^1_{USCB}(X)$, $F^1_{USCG}(X)$, $F^1_{USC}(X)$ are similar to Theorem 3.3 in \cite{huang}.

\begin{pp}\label{repm}
Let $u\in F^1_{USCB} (X)$ (respectively, $u\in F^1_{USCG} (X)$, $u\in F^1_{USC} (X)$). Then
(\romannumeral1) $[u]_\al\in K(X)$ for all $\al\in [0,1]$ (respectively, $[u]_\al\in K(X)$ for all $\al\in (0,1]$, $[u]_\al\in C(X)$ for all $\al\in [0,1]$);
(\romannumeral2) $[u]_\al= \bigcap_{\beta<\al} [u]_\beta$ for all $\al\in (0,1]$;
 (\rmn3) $[u]_0 =  \overline{\bigcup_{\al>0} [u]_\al}$.

Moreover, if the family of sets $\{v_\al: \al\in [0,1]\}$ satisfies (\romannumeral1) through (\romannumeral3), then
there exists a unique $u\in F^1_{USCB} (X)$ (respectively, $u\in F^1_{USCG} (X)$, $u\in F^1_{USC} (X)$) such that $[u]_\al= v_\al$ for each
$\al\in [0,1]$.
\end{pp}

Using Proposition \ref{repm}, we can define a certain type fuzzy set
by giving the family of its $\al$-cuts.
For brevity,
in the sequel,
we often directly point out that what we defined is a certain type fuzzy set
 without mentioning the use of Proposition \ref{repm} since it is easy to see.

Let
$(X,d)$ be a metric space.
 We
use $\bm{H}$ to denote the \emph{\textbf{Hausdorff distance}}
on
 $C(X)$ induced by $d$, i.e.,
$$
\bm{H(U,V)}  =   \max\{H^{*}(U,V),\ H^{*}(V,U)\}
$$
for arbitrary $U,V\in C(X)$,
where
  $$
H^{*}(U,V)=\sup\limits_{u\in U}\,d\, (u,V) =\sup\limits_{u\in U}\inf\limits_{v\in
V}d\, (u,v).
$$

The metric $\overline{d}$ on $X \times [0,1]$ is defined
as follows: for $(x,\al), (y, \beta) \in X \times [0,1]$,
$$  \overline{d } ((x,\al), (y, \beta)) = d(x,y) + |\al-\beta| .$$

Throughout this paper, we suppose that \emph{the metric on} $X\times[0,1]$ \emph{is} $\overline{d}$.
For simplicity,
we also use $X \times [0,1]$
to denote the metric space $(X \times [0,1], \overline{d})$.

If
there is no confusion, we also use $H$ to denote the Hausdorff distance on $C(X\times [0,1])$ induced by $\overline{d}$.

\begin{re}
{\rm

$\rho$ is said to be a \emph{metric} on $Y$ if $\rho$ is a function from $Y\times Y$ into $\mathbb{R}$
satisfying
positivity, symmetry and triangle inequality. At this time, $(Y, \rho)$ is said to be a metric space.

  $\rho$ is said to be an \emph{extended metric} on $Y$ if $\rho$ is a function from $Y\times Y$ into $\mathbb{R} \cup \{+\infty\} $
satisfying
positivity, symmetry and triangle inequality. At this time, $(Y, \rho)$ is said to be an extended metric space.

We can see that for arbitrary metric space $(X,d)$, the Hausdorff distance $H$ on $K(X)$ induced by $d$ is a metric.
So
the Hausdorff distance $H$ on $K(X\times [0,1])$ induced by $\overline{d}$ on $X\times [0,1]$
is a metric.

The Hausdorff distance $H$ on $C(X)$ induced by $d$ on $X$
is an extended metric, but may not be a metric,
because
$H(A,B)$ could be equal to $+\infty$ for certain metric space $X$ and $A, B \in C(X)$.
Clearly, if $H$ on $C(X)$ induced by $d$
is not a metric, then $H$ on $C(X\times [0,1])$ induced by $\overline{d}$
is also not a metric.
So
the Hausdorff distance $H$ on $C(X\times [0,1])$ induced by $\overline{d}$ on $X\times [0,1]$
 is an extended metric but may not be a metric.
We can see that $H$ on $C(\mathbb{R}^m)$ is an extended metric but not a metric,
and then the same is $H$ on $C(\mathbb{R}^m\times [0,1])$.

We call the Hausdorff distance $H$ the Hausdorff metric (respectively, the Hausdorff extended metric) if
 $H$ is a metric (respectively, an extended metric).
In this paper, for simplicity,
 we refer to both the Hausdorff extended metric and the Hausdorff metric as the Hausdorff metric.

}
\end{re}

The Hausdorff metric has the following important properties.

\begin{tm} \cite{roman,kle}\label{bfc} Let $(X,d)$ be a metric space and let $H$ be the Hausdorff metric induced by $d$.
Then
\\
 ({\romannumeral1})\ $ (X,d)$ is complete $\Longleftrightarrow$  $(K(X), H)$ is complete;
 \\
  ({\romannumeral2})\ $ (X,d)$ is separable $\Longleftrightarrow$  $(K(X), H)$ is separable;
  \\
  ({\romannumeral3})\ $ (X,d)$ is compact $\Longleftrightarrow$  $(K(X), H)$ is compact.
\end{tm}

Rojas-Medar and Rom\'{a}n-Flores \cite{rojas} introduced
the $\Gamma$-convergence of a sequence
of fuzzy sets in a metric space based on the Kuratowski convergence
of a sequence of sets in a certain metric space.

Let $(X,d)$ be a metric space.
Let $C$ be a set in $X$ and
$\{C_n\}$ a sequence of sets in $X$.
 $\{C_n\}$ is said to \emph{\textbf{Kuratowski converge}} to
$C$, if
$$
C
=
\liminf_{n\rightarrow \infty} C_{n}
=
\limsup_{n\rightarrow \infty} C_{n},
$$
where
\begin{gather*}
\liminf_{n\rightarrow \infty} C_{n}
 =
 \{x\in X: \  x=\lim\limits_{n\rightarrow \infty}x_{n},    x_{n}\in C_{n}\},
\\
\limsup_{n\rightarrow \infty} C_{n}
=
\{
 x\in X : \
 x=\lim\limits_{j\rightarrow \infty}x_{n_{j}},x_{n_{j}}\in C_{n_j}
\}
 =
 \bigcap\limits_{n=1}^{\infty}   \overline{   \bigcup\limits_{m\geq n}C_{m}    }.
\end{gather*}
In this case, we will write
\bm{  $C=\lim^{(K)}_{n\to\infty}C_n    $ }.

\begin{re} \label{ksc}
{\rm
Definition 3.1.4 in \cite{kle} gives the definitions of
$\liminf C_{n}$, $\limsup C_{n}$
 and
$\lim C_{n}$
for
 a net of subsets $\{C_n, n\in D\}$ in a topological space.
When $\{C_n, n=1,2,\ldots\}$ is
 a sequence of subsets of a metric space,
$\liminf C_{n}$, $\limsup C_{n}$
 and
$\lim C_{n}$
according to Definition 3.1.4 in \cite{kle}
are
$\liminf_{n\rightarrow \infty} C_{n}$, $\limsup_{n\rightarrow \infty} C_{n}$
and $\lim^{(K)}_{n\to\infty}C_n  $
according to
 the above definitions, respectively.
So Corollary 3.2.13 in \cite{kle} implies
Theorem \ref{infc}.

  }
\end{re}

  Let $u$, $u_n$, $n=1,2,\ldots$, be fuzzy sets in $F(X)$.
   $\{u_n\}$ is said to \bm{$\Gamma$}-\emph{\textbf{converge}}
  to
  $u$, denoted by \bm{$u = \lim_{n\to \infty}^{(\Gamma)}  u_n$},
  if
  ${\rm end}\, u= \lim_{n\to \infty}^{(K)}  {\rm end}\, u_n$.
Here $\lim_{n\to \infty}^{(K)} {\rm end}\, u_n$ is considered
with respect to
$(X \times [0,1], \overline{d})$.

Let $(X,d)$ be a metric space and let $u \in  F(X)$. Then from basic analysis, the following three properties
are
equivalent:
(\romannumeral1) $u$ is upper semi-continuous;
(\romannumeral2) ${\rm end}\, u$ is closed in $(X\times [0,1],   \overline{d})$;
(\romannumeral3) ${\rm send}\, u$ is closed in $(X\times [0,1],  \overline{d})$.

Kloeden \cite{kloeden2} introduced the endograph metric $H_{\rm end}$.
The endograph metric $H_{\rm end} $ and the sendograph metric $H_{\rm send} $
on $F^1_{USC}(X)$ are defined as follows.
For $u,v \in F^1_{USC}(X)$,
\begin{gather*}
\bm{  H_{\rm end}(u,v)    }: =  H({\rm end}\, u,  {\rm end}\, v ),
\\
\bm{  H_{\rm send}(u,v)    }: =  H({\rm send}\, u,  {\rm send}\, v ),
  \end{gather*}
where
 $H$
is
the Hausdorff
metric on $C(X \times [0,1])$ induced by $\overline{d}$ on $X \times [0,1]$.

The $d_\infty$ metric
  on $ F^1_{USC} (X)$
is defined as follows: for $u,v\in F^1_{USC} (X)$,
$$\bm{d_\infty (u,v)} :=   \sup\{ H([u]_\al, [v]_\al) :\al\in [0,1]  \}.$$
Below property \eqref{smr} may be known. We gave this kind of conclusion in \cite{huang91u}.
\begin{equation}\label{smr}
 \mbox{For each } u,v \in F^1_{USC} (X), \ d_\infty(u,v) \geq H_{\rm send}(u,v) \geq H_{\rm end}(u,v).
\end{equation}

\begin{re}
  {\rm

We can see that $H_{\rm end}$ is a metric on $F^1_{USC}(X)$ with $H_{\rm end}(u,v) \leq 1$ for all $u,v \in F^1_{USC}(X)$.
 Both $d_\infty$
and
$H_{\rm send}$ are metrics on $F^1_{USCB}(X)$.
However,
each one of
$d_\infty$ and $H_{\rm send}$ on $F^1_{USC}(X)$ is an extended metric but may  not be a metric. See also
Remark 3.3 in \cite{huang17}.
We can see that both $d_\infty$ and $H_{\rm send}$ on $F^1_{USCG}(\mathbb{R}^m)$ are not metrics, they are extended metrics.

For simplicity,
in this paper, we call
$H_{\rm send}$ on $F^1_{USC}(X)$ the $H_{\rm send}$ metric or the sendograph metric $H_{\rm send}$.
We call
$d_\infty$ on $F^1_{USC}(X)$ the $d_\infty$ metric or the supremum metric $d_\infty$.

}
\end{re}

\section{$H_{\rm end}$, $H_{\rm send}$ and Kuratowski convergence on $P^1_{USC} (X)$ } \label{fsm}

In this section, we introduce $P^1_{USC} (X)$ and its subset $P^1_{USCB} (X)$,
and
 define the $H_{\rm send}$ distance and the $H_{\rm end}$ distance on $P^1_{USC} (X)$.
We discuss
the properties and relations of $H_{\rm send}$, $H_{\rm end}$ and Kuratowski convergence on
 $P^1_{USC} (X)$.
Based on this, we give some
 relations among
$H_{\rm end}$, $H_{\rm send}$
and $\Gamma$-convergence on $F^1_{USC} (X)$.

For $u \subseteq  X \times [0,1]$ and $\al\in [0,1]$,
define
 $\bm{\langle u \rangle_\al} := \{ x:   (x,\al) \in u\}$.
 $P^1_{USC} (X)$ and $P^1_{USCB} (X)$
are
subsets of the power set of $X\times [0,1]$ defined by
\begin{gather*}
\begin{split}
\bm{P^1_{USC} (X)}: = \{u  \subseteq   X\times [0,1]:    \langle u \rangle_\al = &  \bigcap_{\beta <\al} \langle u \rangle_\beta    \  \mbox{for all} \  \al\in (0,1];
\\
&   \langle u \rangle_\al \in C(X) \ \mbox{for all} \  \al\in [0,1] \},
\end{split}
\\
\begin{split}
\bm{P^1_{USCB} (X)}: = \{u  \in  P^1_{USC} (X): \langle u\rangle_\al \in K(X) \ \mbox{for all} \  \al\in [0,1] \}.
\end{split}
\end{gather*}
Clearly $P^1_{USCB} (X) \subseteq P^1_{USC} (X)$.

By Proposition \ref{repm},
we can see $F^1_{USC} (X)$ (respectively, $F^1_{USCB} (X)$) as a subset of $P^1_{USC} (X)$ (respectively, $P^1_{USCB} (X)$) by identifying a fuzzy set with its
sendograph.
So the conclusions on $P^1_{USC} (X)$ and $P^1_{USCB} (X)$
are useful
for the discussions of fuzzy sets in this paper.

From the basic analysis, we can obtain that
\\
(\romannumeral1) if $u\in  P^1_{USC} (X)$,
then $u\in C(X\times [0,1])$;
\\
(\romannumeral2)
if $u\in C(X\times [0,1])$, then $\langle u \rangle_\al \in C(X) \cup \{\emptyset\}$ for all $\al\in [0,1]$;
\\
(\romannumeral3)
if $u\in  P^1_{USCB} (X)$, then $u\in K(X\times [0,1])$;
\\
(\romannumeral4)
if $u\in K(X\times [0,1])$, then $\langle u \rangle_\al \in K(X) \cup \{\emptyset\}$ for all $\al\in [0,1]$.

So we have
\begin{gather*}
\begin{split}
P^1_{USC} (X) = \{u  \in  C(X\times [0,1]):   \langle u \rangle_1 \not=\emptyset,  \langle u \rangle_\al =   \bigcap_{\beta <\al} \langle u \rangle_\beta    \  \mbox{for all} \  \al\in (0,1] \},
\end{split}
\\
\begin{split}
P^1_{USCB} (X) = \{u  \in  P^1_{USC} (X):  u  \in  K(X\times [0,1])\}.
\end{split}
\end{gather*}

 We define the $H_{\rm send}$ distance and the $H_{\rm end}$ distance on $P^1_{USC} (X)$: for each $u,v\in P^1_{USC} (X)$,
\begin{gather*}
 H_{\rm send} (u,v)  :=  H(u,v),
  \
  H_{\rm end} (u,v)  :=  H(\underline{u}, \underline{v}),
\end{gather*}
where $H$ is the Hausdorff
metric on $C(X \times [0,1])$ induced by $\overline{d}$ on $X \times [0,1]$,
and
$\underline{u}  :=   u \cup  (  X \times \{0\}   )$.

As $H_{\rm send}$ on $P^1_{USC} (X)$ is
the restriction to $P^1_{USC} (X)\times P^1_{USC} (X)$
of the Hausdorff extended
metric $H$ on $C(X \times [0,1])$,
we have that
$H_{\rm send}$ is an extended metric on $P^1_{USC} (X)$.

Obviously
$H_{\rm end} \leq 1$ on $P^1_{USC} (X)$.
$H_{\rm end}$ is a pseudometric on $P^1_{USC} (X)$; that is,
for all $u,v\in P^1_{USC}(X)$,
$H_{\rm end}(u,v) \in \mathbb{R}$ and $H_{\rm end}(u,u)=0$,
and $H_{\rm end}$
satisfies symmetry and triangle inequality.
$H_{\rm end}$ is a metric on $P^1_{USC} (X)$
if and only if $X$ is a singleton.
The ``if'' part is obvious. To show the ``only if'' part,
assume that $X$ contains more than one point.
Let $x,y\in X$ with $x\not=y$.
Define $u\in P^1_{USCB} (X)$ by putting $\langle u\rangle_\al = \{x\}$ for all $\al\in [0,1]$,
and
$v\in P^1_{USCB} (X)$
 by putting
$\langle v\rangle_\al = \{x\}$ for $\al\in (0,1]$ and $\langle v\rangle_0 = \{x,y\}$.
Then $u \not= v$
and
$H_{\rm end}(u,v) = 0$. So $H_{\rm end}$ does not satisfy
the
positivity and hence is not a metric on $P^1_{USC} (X)$.

Define a function
$f: F^1_{USC} (X) \to P^1_{USC} (X)$
by
$f(u) = {\rm send}\, u$.
Then 
$f$ is an isometric embedding
of
$(F^1_{USC} (X), H_{\rm send})$ in $(P^1_{USC} (X), H_{\rm send})$.
So $(F^1_{USC} (X), H_{\rm send})$
can be embedded isometrically in $(P^1_{USC} (X), H_{\rm send})$.
Note that
 $f(F^1_{USCB}(X) ) \subseteq P^1_{USCB} (X)$.
Thus
$(F^1_{USCB} (X), H_{\rm send})$ can be embedded isometrically in $(P^1_{USCB} (X), H_{\rm send})$.

For $u\in  F^1_{USC} (X) $,
we define
$\bm{\overrightarrow{u}}:= f(u)= {\rm send}\, u$. Then $\overrightarrow{u} \in P^1_{USC} (X) $.

Let
$v \in P^1_{USC} (X)$. Define
$\bm{v'} \in f(F^1_{USC} (X)) \subseteq P^1_{USC} (X)$ by putting
\[
\langle v' \rangle_\al =
\left\{
  \begin{array}{ll}
\langle v \rangle_\al, & \al\in (0,1],
\\
\overline{\cup_{\al>0} \langle v \rangle_\al}, & \al=0.
  \end{array}
\right.
\]
Define $\bm{\overleftarrow{v}} \in F^1_{USC} (X)$
by putting
$[\overleftarrow{v}]_\al = \langle v' \rangle_\al$ for $\al\in[0,1]$.
Then
$\overrightarrow{\overleftarrow{v}}=f(\overleftarrow{v}) = v'$.
$\langle v \rangle_0 \supseteq \langle v' \rangle_0$ as $\langle v \rangle_0\in C(X)$ and $\langle v \rangle_0 \supseteq \cup_{\al>0} \langle v \rangle_\al$. Note that $\langle v \rangle_0 = \langle v' \rangle_0 \cup \overline{\langle v \rangle_0\setminus\langle v' \rangle_0} $.
Thus we have (a) $v=v' \cup (\overline{\langle v \rangle_0\setminus\langle v' \rangle_0} \times \{0\})$.
So $P^1_{USC}(X) = \{
\overrightarrow{u} \cup (A\times\{0\}): u\in F^1_{USC}(X), A \in C(X)\cup \{\emptyset\} \}$ (``$\subseteq$'' follows from (a), and ``$\supseteq$'' is obvious).
Clearly $\langle v \rangle_0 \in K(X) $ if and only if $\langle v' \rangle_0\in K(X)$ and $\overline{\langle v \rangle_0\setminus\langle v' \rangle_0} \in K(X)\cup \{\emptyset\}$.
So $P^1_{USCB}(X) = \{
\overrightarrow{u} \cup (A\times\{0\}): u\in F^1_{USCB}(X), A\in K(X)\cup \{\emptyset\} \}$ (``$\supseteq$'' is obvious).
We can see that the above two expressions of $P^1_{USC}(X)$ and $P^1_{USCB}(X)$ remain true
if ``$\cup \{\emptyset\} $'' is deleted in them.

For a subset $U$ of $F^1_{USC}(X)$, we use
$\bm{\overrightarrow{U}}$ to denote
the set
$\{\overrightarrow{u}: u\in U\}$.
For
 a subset $U$ of $P^1_{USC}(X)$, we use
$\bm{\overleftarrow{U}}$ to denote
the set
$\{\overleftarrow{u}: u\in U\}$.
Clearly $\overleftarrow{P^1_{USC}(X)}=F^1_{USC}(X)$
and
$\overleftarrow{P^1_{USCB}(X)}=F^1_{USCB}(X)$.

\begin{pp} \label{pfbe}
Let $v \in P^1_{USC} (X)$. Suppose the conditions:
(\romannumeral1)
$v\in \overrightarrow{F^1_{USC}(X)}$;
(\romannumeral2)
$\langle v \rangle_0 = \overline{\bigcup_{\delta>0} \langle v\rangle_\delta}$;
(\romannumeral3)
$v=v'$;
(\romannumeral4)
$\langle v \rangle_0 = [\overleftarrow{v}]_0 $;
(\romannumeral5)
$v\in \overrightarrow{F^1_{USCB}(X)}$;
(\romannumeral6)
$\lim_{\delta\to 0+} H(\langle v\rangle_\delta, \langle v \rangle_0)=0$.
\\
(a) (\romannumeral5)$\Rightarrow$(\romannumeral6)$\Rightarrow$(\romannumeral1)$\Leftrightarrow$(\romannumeral2)$\Leftrightarrow$(\romannumeral3)$\Leftrightarrow$(\romannumeral4).
\\
(b) If $v \in P^1_{USCB} (X)$, then
(\rmn1)$\Leftrightarrow$(\rmn2)$\Leftrightarrow$(\rmn3)$\Leftrightarrow$(\rmn4)$\Leftrightarrow$(\rmn5)$\Leftrightarrow$(\rmn6).
\end{pp}

\begin{proof}
Clearly (\rmn1)$\Leftrightarrow$(\rmn2)$\Leftrightarrow$(\rmn3)$\Leftrightarrow$(\rmn4).
By Lemma \ref{gnc}(\rmn5), (\romannumeral5)$\Rightarrow$(\romannumeral6).
From Remark \ref{crg}(\rmn1), (\romannumeral6)$\Rightarrow$(\romannumeral2).
So (a) is proved.

If $v \in P^1_{USCB} (X)$, then $\langle v\rangle_0\in K(X)$.
Assume that (\rmn1) is true; that is, there is a $u\in F^1_{USC}(X)$ with $v=\overrightarrow{u}$. Then $[u]_0=\langle v\rangle_0$, and thus
$u\in F^1_{USCB} (X)$.
Hence (\rmn1)$\Rightarrow$(\romannumeral5). So from (a), we have that (b) is true.
\end{proof}

\begin{tm} \label{pseu}
Let $(X,d)$ be a metric space.
For $u, v \in P^1_{USC} (X)$:
\\
 (\romannumeral1) \ $H_{\rm end}(u, v) \leq H_{\rm send}(u, v)$;
\\
(\romannumeral2) \ $H(\langle u \rangle_0, \langle v\rangle_0) \leq H_{\rm send}(u, v)$;
\\
(\romannumeral3) \ If $H_{\rm end} (u,v) < 1$, then
$
  H_{\rm send}(u,v) \leq H_{\rm end}(u,v) + H(\langle u\rangle_0, \langle v\rangle_0).
$
\\
For a sequence $\{u_n\}$ in $P^1_{USC} (X)$ and $u$ in $P^1_{USC} (X)$:
\\
(\romannumeral4) \
  $H_{\rm send}(u_n, u) \to 0$ if and only if $H_{\rm end}(u_n, u) \to 0$ and $H(\langle u_n\rangle_0, \langle u \rangle_0) \to 0$;
\\
(\romannumeral5) \ $\lim_{n\to \infty}^{(K)} u_n = u $ if and only if
$\lim_{n\to \infty}^{(K)}  \underline{u_n } = \underline{u} $
and
$ \lim_{n\to \infty}^{(K)} \langle u_n \rangle_0 = \langle u \rangle_0$.
\end{tm}

\begin{proof}
Clearly (\romannumeral1) and (\romannumeral2) are true.
To show (\romannumeral3),
let $(x,\alpha) \in  u$. It suffices to verify the following property
(a) if $H_{\rm end} (u,v) < 1$, then
\begin{equation}\label{gen}
 \overline{d}((x,\al),  v)\leq H_{\rm end}(u,v) + H(\langle u\rangle_0, \langle v\rangle_0).
\end{equation}

  We claim that
(b) $ \overline{d}((x,\al),  v)
\leq
\al + d(x, \langle v\rangle_0)\leq
\al + H(\langle u\rangle_0, \langle v\rangle_0)$, and
(c)
if $\overline{d}((x,\al),  \underline{v})<\al$,
then
$
\overline{d}((x,\al),  v)
=
\overline{d}((x,\al),  \underline{v})\leq H_{\rm end}(u,v).
$

Notice that $ \overline{d}((x,\al),  v)
\leq \overline{d}((x,\al), \langle v\rangle_0 \times \{0\})
\leq \overline{d}((x,\al), (x,0)) + \overline{d}((x,0), \langle v\rangle_0 \times \{0\})
= \al + d(x, \langle v\rangle_0)$.
Thus (b) is true.
Note that $\overline{d}((x,\al),  \underline{v}) = \min\{\overline{d}((x,\al),  v), \ \overline{d}((x,\al), X\times\{0\}) \} = \min\{\overline{d}((x,\al),  v), \ \alpha \}$.
So if $\overline{d}((x,\al),  \underline{v}) < \al $,
then
$
\overline{d}((x,\al),  v)
=
\overline{d}((x,\al),  \underline{v}).
$
Thus (c) is true.

If $\al\in [0, H_{\rm end}(u,v)]$, then by (b), \eqref{gen} holds.
If $H_{\rm end} (u,v) < 1$ and $\al\in (H_{\rm end}(u,v), 1]$,
then
$\overline{d}((x,\al),  \underline{v})
< \alpha$.
Hence by (c), $
     \overline{d}((x,\al),  v) \leq H_{\rm end}(u,v)$, and thus
    \eqref{gen} holds. So (a) is true and then
(\romannumeral3) is proved.

Let $u, v \in P^1_{USC} (X)$. Suppose the following conditions (d-1) $H_{\rm end}(u,v) =0$,
  (d-2) $H(\langle u\rangle_0, \langle v\rangle_0) =0 $, and
 (d-3) $H(\langle u\rangle_0, \langle v\rangle_0) = +\infty$.
If (d-1) holds, then by (\rmn2) and (\rmn3), $H_{\rm send}(u,v) =  H(\langle u\rangle_0, \langle v\rangle_0)$.
If (d-2) holds, then $\langle u\rangle_0 = \langle v\rangle_0$ and so $H_{\rm send}(u,v) = H_{\rm end}(u,v)$.
If (d-3) holds, then by (\rmn2), $H_{\rm send}(u,v)=+\infty$.
So if one of the  conditions (d-1),
  (d-2), and (d-3) holds,
then $
  H_{\rm send}(u,v) = H_{\rm end}(u,v) + H(\langle u\rangle_0, \langle v\rangle_0)
$.
However, the converse is false.

(\romannumeral4) follows immediately from (\romannumeral1), (\romannumeral2) and (\romannumeral3).
Below we verify (\romannumeral5).

Suppose that
$\lim_{n\to \infty}^{(K)} u_n = u $. To show
$\lim_{n\to \infty}^{(K)}  \underline{u_n } = \underline{u} $
and
$ \lim_{n\to \infty}^{(K)} \langle u_n \rangle_0 = \langle u \rangle_0$,
we only need to show
that
\begin{gather*}
 \underline{u} \subseteq \liminf_{n\to\infty} \underline{u_n}, \
\limsup_{n\to\infty} \underline{u_n} \subseteq \underline{u},\\
\langle u\rangle_0 \subseteq \liminf_{n\to\infty} \langle u_n \rangle_0,\
\limsup_{n\to\infty} \langle u_n \rangle_0  \subseteq   \langle u \rangle_0.
\end{gather*}

Let $(x,\al) \in \underline{u}$. If $\al=0$, then clearly  $(x,\al) \in \liminf_{n\to\infty} \underline{u_n}$.
If
 $\al>0$, then $(x,\al) \in u$, and thus $(x,\al) \in  \liminf_{n\to\infty} u_n \subseteq \liminf_{n\to\infty} \underline{u_n}$.
So $\underline{u} \subseteq \liminf_{n\to\infty} \underline{u_n}$.

Let $(x,\al) \in \limsup_{n\to\infty} \underline{u_n}$.
If $\al=0$, then clearly  $(x,\al) \in \underline{u}$.
If
 $\al>0$, then $(x,\al) \in \limsup_{n\to\infty} u_n = u \subseteq \underline{u}$.
So $\limsup_{n\to\infty} \underline{u_n} \subseteq \underline{u}$.

Let $x\in \langle u\rangle_0 $.
Then $(x,0) \in u = \liminf_{n\to\infty} u_n $. Thus there is a sequence $\{(x_n, \al_n)\}$
such that
$(x_n, \al_n)\in u_n$, $n=1,2,\ldots$ and $(x,0)=\lim_{n\to \infty} (x_n,\al_n)$.
Hence
$x_n\in \langle u_n \rangle_0$ and $x=\lim_{n\to \infty} x_n$.
So
 $\langle u\rangle_0 \subseteq \liminf_{n\to\infty} \langle u_n \rangle_0$.

Let $x\in \limsup_{n\to\infty} \langle u_n \rangle_0$.
Then
there is a sequence $\{x_{n_i}\}$
such that
$x_{n_i} \in \langle u_{n_i} \rangle_0$, $i=1,2,\ldots$
and
$x= \lim_{i\to \infty} x_{n_i}$.
Thus
$(x,0) = \lim_{i\to \infty} (x_{n_i}, 0) \in \limsup _{n\to\infty}  u_n =u$.
Hence
$x\in \langle u \rangle_0$.
So
$\limsup_{n\to\infty} \langle u_n \rangle_0  \subseteq   \langle u \rangle_0$.

Suppose that
$\lim_{n\to \infty}^{(K)}  \underline{u_n } = \underline{u} $
and
$ \lim_{n\to \infty}^{(K)} \langle u_n \rangle_0 = \langle u \rangle_0$.
To show
$\lim_{n\to \infty}^{(K)} u_n = u $,
we only need to show that
\begin{equation*}
  u \subseteq \liminf_{n\to\infty} u_n, \  \limsup_{n\to\infty} u_n \subseteq u.
\end{equation*}

Let $(x,\al)\in u$.
If $\al=0$, then $x\in \langle u \rangle_0 = \lim_{n\to \infty}^{(K)} \langle u_n \rangle_0$.
Thus there is a sequence $\{x_n\}$ such that $x_n \in \langle u_n \rangle_0$, $n=1,2,\ldots$
and
$x = \lim_{n\to \infty} x_n$.
Hence
$(x,\alpha) = (x,0) = \lim_{n\to \infty} (x_n, 0) \in \liminf_{n\to\infty} u_n$.

If
$\al>0$, then from
 $(x,\al) \in \underline{u}= \lim_{n\to \infty}^{(K)}  \underline{u_n }$,
 there is a sequence  $\{(x_n,\al_n)\}$ such that $(x,\al) = \lim_{n\to \infty} (x_n,\al_n)$,
and
$(x_n,\al_n) \in  \underline{u_n}$ and $\al_n>0$ for $n=1,2,\ldots$.
Thus
$(x_n,\al_n) \in   u_n$, $n=1,2,\ldots$ and hence
$(x,\al) \in  \liminf_{n\to\infty} u_n $.

Let $(x,\al) \in \limsup_{n\to\infty} u_n$.
Then
there is a sequence $\{(x_n,\al_n)\}$ such that for each $n\in \mathbb{N}$,
$(x_n,\al_n) \in u_n$,
and
$\overline{d}((x_n,\al_n), (x,\al)) \to 0$.
So for each $n\in \mathbb{N}$,
$x_n \in \langle u_n\rangle_0$,
and
 $d(x_n, x) \to 0$, and hence
$x \in \limsup_{n\to\infty} \langle u_n \rangle_0 = \langle u \rangle_0$.
Clearly $(x,\al) \in \limsup_{n\to\infty} \underline{u_n} = \underline{u}$.
Thus
$(x,\al) \in u$.
\end{proof}

For each $u,v\in P^1_{USC} (X)$, $\underline{u}={\rm end}\,\overleftarrow{u}$
and so
$H_{\rm end}(u,v) = H_{\rm end}(\overleftarrow{u}, \overleftarrow{v})$.

For each $u\in F^1_{USC} (X)$, $[u]_0=\langle \overrightarrow{u} \rangle_0$  and ${\rm end}\,u = \underline{(\overrightarrow{u})}$.

Let $u,v \in  F^1_{USC} (X) $. Then $\overrightarrow{u}, \overrightarrow{v} \in P^1_{USC} (X) $, and
$
 H_{\rm send} (u,v)  =  H_{\rm send}  (\overrightarrow{u}, \overrightarrow{v})$,
 $
  H_{\rm end} (u,v)  =  H_{\rm end} (\overrightarrow{u}, \overrightarrow{v})
$, and
$
H([u]_0, [v]_0) = H(\langle   \overrightarrow{u} \rangle_0, \langle   \overrightarrow{v} \rangle_0)$.
 Thus (\romannumeral1), (\romannumeral2) and (\romannumeral3) in Theorem \ref{pseu} imply that (the following \eqref{emsf} is part of \eqref{smr})
\begin{gather}
  H_{\rm end}(u, v) \leq H_{\rm send}(u, v),  \label{emsf}
\\
   H( [u]_0, [v]_0) \leq H_{\rm send}(u, v),  \label{umsf}
\\
\mbox{if } H_{\rm end} (u,v) < 1, \mbox{then }
  H_{\rm send}(u,v) \leq H_{\rm end}(u,v) + H([u]_0, [v]_0).\label{secf}
\end{gather}

\begin{pp} \label{sge}
Let $u$, $u_n$, $n=1,2,\ldots$, be fuzzy sets in $F^1_{USC} (X)$.
\\
(\romannumeral1)
  $H_{\rm send} (u_n, u) \to 0$ if and only if $  H_{\rm end} (u_n, u) \to 0 $ and  $ H([u_n]_0, [u]_0) \to 0$.
\\
(\romannumeral2)
 $   \lim_{n\to \infty}^{(K)} {\rm send}\, u_n= {\rm send}\,u $ if and only if
$\lim_{n\to \infty}^{(\Gamma)}  u_n = u $
and
$ \lim_{n\to \infty}^{(K)} [u_n]_0 = [u]_0$.

\end{pp}

\begin{proof}
  (\romannumeral1) follows immediately from
\eqref{emsf}, \eqref{umsf} and \eqref{secf}.
(\romannumeral1) can be seen as a special case of the clause (\romannumeral4) in Theorem \ref{pseu}.
 (\romannumeral2)
can be seen as
a special case of the clause (\romannumeral5) in Theorem \ref{pseu}.
\end{proof}

The following Theorem \ref{hkg}
is an already known conclusion\footnote{
Theorem \ref{hkg} may come from \cite{akin}.
However we cannot obtain \cite{akin}.
A reviewer kindly recommended \cite{akin} and the symbol $\chi_{S}$ in Paragraph 2 of Section 2.
}, which is useful in this paper.
  It can be proved in a similar fashion to Theorem 4.1 in \cite{huang}.
In Theorem 4.1 in \cite{huang}, we include
the case that $C=\emptyset$.
In Theorem \ref{hkg}, we exclude the case that $C=\emptyset$.

\begin{tm} \label{hkg}
Suppose that $C$, $C_n$ are sets in $C(X)$, $n=1,2,\ldots$. Then $H(C_n, C) \to 0$ implies that $\lim_{n\to \infty}^{(K)} C_n \, =C$.
\end{tm}

\begin{re} \label{hkr}
  {\rm
From Theorem \ref{hkg}, we obtain that
 for $u$, $u_n$, $n=1,2,\ldots$ in $P^1_{USC} (X)$:
(\romannumeral1) $H_{\rm end}(u_n,u) \to 0$
implies that
$\lim_{n\to \infty}^{(K)}  \underline{u_n} = \underline{u} $;
(\romannumeral2) $H_{\rm send}(u_n,u) \to 0$
implies that
$\lim_{n\to \infty}^{(K)}  u_n = u$.

So
for $u$, $u_n$, $n=1,2,\ldots$ in $F^1_{USC} (X)$:
(\romannumeral3) $H_{\rm end}(u_n,u) \to 0$
implies that
$\lim_{n\to \infty}^{(\Gamma)}  u_n = u $;
(\romannumeral4) $H_{\rm send}(u_n,u) \to 0$
implies that
$\lim_{n\to \infty}^{(K)} {\rm send}\, u_n = {\rm send}\, u$.

}
\end{re}

The converses of the implications in clauses (\romannumeral1), (\romannumeral2), (\romannumeral3) and (\romannumeral4) in Remark \ref{hkr}
are false.
Let $u=[0, +\infty)_{F(\mathbb{R})}$ and for $n=1,2,\ldots$,
let $u_n= [0,n]_{F(\mathbb{R})}$.
Then $\lim_{n\to \infty}^{(K)} {\rm send}\, u_n = {\rm send}\, u$,
but
$H_{\rm end}(u_n,u) = 1 \not\to 0$.
So combined with Proposition \ref{sge},
the converses of the implications in (\romannumeral3) and (\romannumeral4) are false,
and
thus
the converses of the implications in (\romannumeral1) and (\romannumeral2) are false.

\begin{re}  \label{crg}
{\rm
Let $U\in C(X)$ and
let $\{u(\al): \al\in [0,1]\}$ be a subset of $C(X)$ with $u(\al) \supseteq u(\beta)$ for all $0\leq \al \leq \beta \leq 1$.
\\
(\rmn1) \ If $\al\in [0,1)$ and
 $\lim_{\gamma\to\al+} H(u(\gamma), U) = 0$, then $U=\overline{\bigcup_{\gamma>\al}  u(\gamma)} $;
 \\
(\rmn2) \ If $\al\in (0,1]$ and
 $\lim_{\beta\to\al-} H(u(\beta), U) = 0$, then $U=\bigcap_{\beta<\al}  u(\beta) $.

  Assume that $\al\in [0,1)$ and
 $\lim_{\gamma\to\al+} H(u(\gamma), U) = 0$. Choose a decreasing sequence $\{\gamma_n\}$ in $(\al, 1]$ with $\gamma_n\to\al+$.
Then $\lim_{n\to\infty} H(u(\gamma_n), U) = 0$.
Thus by Theorem \ref{hkg}, $U = \lim_{n\to\infty}^{(K)} u(\gamma_n)= \overline{\bigcup_{n\in \mathbb{N}} u(\gamma_n)}= \overline{\bigcup_{\gamma>\al}  u(\gamma)}$, and hence (\rmn1)
is true.

Assume that $\al\in (0,1]$ and
 $\lim_{\beta\to\al-} H(u(\beta), U) = 0$. Choose an increasing sequence $\{\beta_n\}$ in $[0,\al)$ with $\beta_n\to\al-$.
Then $\lim_{n\to\infty} H(u(\beta_n), U) = 0$.
Thus by Theorem \ref{hkg}, $U = \lim_{n\to\infty}^{(K)} u(\beta_n)= \bigcap_{n\in \mathbb{N}} u(\beta_n) = \bigcap_{\beta<\al}  u(\beta)$, and hence (\rmn2) is true.

}
\end{re}

\section{ Level characterizations of $\Gamma$-convergence } \label{lcg}

In this section, we investigate the level characterizations
of
the $\Gamma$-convergence.
It is shown that
under some conditions,
the $\Gamma$-convergence of fuzzy sets can be decomposed to
the Kuratowski convergence of certain $\al$-cuts.

Throughout this paper, we suppose that the measure on
$\mathbb{R}$ is the
Lebesgue measure.

Rojas-Medar and Rom\'{a}n-Flores (\cite{rojas}) have presented
the following important and useful property
of the
$\Gamma$-convergence.

\begin{tm}
\cite{rojas} \label{Gcln}
   Suppose that $u$, $u_n$, $n=1,2,\ldots$, are fuzzy sets
  in $F^1_{USC} (X)$.
  Then
$\lim_{n\to \infty}^{(\Gamma)}  u_n = u $ if and only if for
all $\al\in (0,1],$
\begin{equation} \label{gcln}
   \{u>\alpha\}
   \subseteq
\liminf_{n\to \infty}[u_n]_{\alpha}
\subseteq
\limsup_{n\to \infty}[u_n]_\alpha
\subseteq
 [u]_\alpha.
\end{equation}

\end{tm}

Theorem \ref{infc} is Theorem 2.1 in \cite{huang}. Of course, the conclusion that
$\limsup_{n\rightarrow \infty} C_{n}$ are closed sets in $(X, d)$ in Theorem \ref{infc}
 can also be deduced from
the fact that $\limsup_{n\rightarrow \infty} C_{n}= \bigcap\limits_{n=1}^{\infty}   \overline{   \bigcup\limits_{m\geq n}C_{m}    }
$.

\begin{tm} \cite{kle} \label{infc}
Let $(X,d)$ be a metric space
and
let
$\{C_{n}\}$ be a sequence of sets in $X$.
  Then
    $\liminf_{n\rightarrow \infty} C_{n}$ and $\limsup_{n\rightarrow \infty} C_{n}$ are closed sets in $(X, d)$.
\end{tm}

\begin{tm} \label{Gclnre}
   Suppose that $u$, $u_n$, $n=1,2,\ldots$, are fuzzy sets
  in $F^1_{USC} ( X )$.
  Then
 $\lim_{n\to \infty}^{(\Gamma)}  u_n = u $ if and only if for
all $\al\in (0,1],$
\begin{equation*}
  \overline{     \{u>\alpha\}     }
   \subseteq
\liminf_{n\to \infty}[u_n]_{\alpha}
\subseteq
 \limsup_{n\to \infty}[u_n]_\alpha
\subseteq
 [u]_\alpha.
\end{equation*}
\end{tm}

\begin{proof}
By Theorem \ref{infc}, for each $\al\in [0,1]$, $\liminf_{n\to \infty}[u_n]_{\alpha}$ is a closed set in $(X,d)$.
  So the desired result follows from Theorem \ref{Gcln}.
\end{proof}

\begin{re}
\label{sur}
 {\rm Suppose that $u$, $u_n$, $n=1,2,\ldots$, are fuzzy sets
  in $F^1_{USC} ( X )$.
 If $\lim_{n\to \infty}^{(\Gamma)}  u_n = u $,
then (\romannumeral1)  $[u]_0 =
\overline{\cup_{\al>0}\liminf_{n\to \infty}[u_n]_{\al}}$,
  (\romannumeral2) $[u]_0 \subseteq
\liminf_{n\to \infty}[u_n]_{0}$, and
(\romannumeral3) $[u]_0
\subsetneqq
\liminf_{n\to \infty}[u_n]_0
$
may happen.

By Theorem \ref{Gcln},
 $[u]_0 =  \overline{\cup_{\al>0}\{ u>\al \}}
   \subseteq\overline{\cup_{\al>0}\liminf_{n\to \infty}[u_n]_{\al}}\subseteq \overline{ \cup_{\al>0}[u]_\al}= [u]_0$.
So (\romannumeral1) is true.
 By Theorem \ref{infc}, $\liminf_{n\to \infty}[u_n]_{0}$ is closed in $(X,d)$,
 so
$\overline{\cup_{\al>0}\liminf_{n\to \infty}[u_n]_{\al}} \subseteq
\liminf_{n\to \infty}[u_n]_{0}$, and then
 (\romannumeral2) follows from (\romannumeral1).
   We show (\romannumeral3) by the following example.
   Let $u= \widehat{0}_{F(\mathbb{R})}$
and for $n=1,2,\ldots$,
define $u_n \in F^1_{USCB}(\mathbb{R})$ as
\[
 u_n(x) = \left\{
            \begin{array}{ll}
              1, & x=0,\\
     1/n, & x\in (0,1],\\
0,& \mbox{otherwise}.
            \end{array}
          \right.
          \]
Then $H_{\rm end}(u_n, u) \to 0$, and therefore from Remark \ref{hkr} $\lim_{n\to \infty}^{(\Gamma)}  u_n = u $.
And
$[u]_0 = \{0\}
\subsetneqq
[0,1]=
\liminf_{n\to \infty}[u_n]_0
$.

Since $\lim_{n\to \infty}^{(\Gamma)}  u_n = u $ implies that $[u]_0 \subseteq
\liminf_{n\to \infty}[u_n]_{0}$,
then
by Proposition \ref{sge}, $   \lim^{(K)}_{n\to \infty} {\rm send}\, u_n= {\rm send}\,u $ if and only if
$\lim_{n\to \infty}^{(\Gamma)}  u_n = u $ and
$[u]_0
\supseteq
\limsup_{n\to \infty}[u_n]_0
$.
}
\end{re}

Let $u\in F(X)$. Denote
\begin{gather*}
  D(u) :=         \{ \al\in (0,1):  [u]_\al \nsubseteq \overline{\{u>\al\}} \},  \\
   P(u) :=         \{ \al\in (0,1):  \overline{\{u>\al\}} \subsetneqq  [u]_\al \}.
\end{gather*}
A
 number $\al$ in $P(u)$
is
called a platform point of $u$.
Clearly $P(u) \subseteq D(u)$.
$P(u)\subsetneqq D(u) $ could happen.
See Example \ref{pdr}.

\begin{eap}\label{pdr}
{\rm
  let $u\in F(\mathbb{R})$ be defined by
\[u(x)
=\left\{
   \begin{array}{ll}
    1,       & x\in (0,1),\\
     0.6,     & x\in [1,3],  \\
     0,        &  x\in \mathbb{R}\setminus (0,3].
   \end{array}
 \right.
\]
Then $P(u)=\emptyset$ and $D(u)=\{0.6 \}$.
So $P(u) \subsetneqq D(u)$.
}
\end{eap}

The symbol $l^2$ denotes the
 \textbf{Hilbert space} $\bm{l^2}:= \{(x_i)_{i=1}^{+\infty}: \sum_{i=1}^{+\infty} x_i^2 < +\infty \}$.
 $\|\cdot\|$ and $\langle\cdot,\cdot\rangle$ denote the norm and inner product on $l^2$, respectively.
For each $x = (x_i)_{i=1}^{+\infty}$ and $y = (y_i)_{i=1}^{+\infty}$ in $l^2$,
 $
 \langle x,y\rangle = \sum_{i=1}^{+\infty} x_i \cdot y_i$ and
 $
 \|x\|=\sqrt{\langle x,x\rangle}$.
 $\mathbb{R}^m$ can be seen as a subspace of $l^2$.
 $\|\cdot\|$ and $\langle\cdot,\cdot\rangle$ are also used to denote the Euclidean norm and Euclidean inner product on $\mathbb{R}^m$, respectively.

 A set $S$ is said to be countable if it is finite or countably infinite.
 In this sense, ``countable'' means the same as ``at most countable''.
 Theorem 5.1 in \cite{huang} says that
$D(u)$
is countable when
$u\in F (\mathbb{R}^m)$.

\begin{tm} \label{ldc}
   Let $u\in F(l^2)$. Then the set $D(u)$
   is countable.
\end{tm}

\begin{proof}
  The proof is similar to
that of Theorem 5.1 in \cite{huang}.
A sketch of the proof is given as follows.
We use $\mathbf{S}^1$ to denote the set $\{ e\in l^2: \|e\|=1\}$.

For each $t\in l^2$ and $r>0$,
 define
$S_{u, t, r}(\cdot, \cdot) :  \mathbf{S}^1 \times [0,1] \to \{-\infty \} \cup \mathbb{R} $
by
\[
S_{u, t, r} (e,\al)
=\left\{
   \begin{array}{ll}
 -\infty,  & \mbox{if} \  [u]_\al \cap \overline{B(t,r)} = \emptyset,
\\
\sup \{     \langle e, x-t \rangle :  x\in  [u]_\al \cap \overline{B(t,r)}  \}, & \mbox{if} \  [u]_\al \cap \overline{B(t,r)} \not= \emptyset,
   \end{array}
 \right.
\]
where
 $B(t,r):= \{x\in l^2: \|x-t\| < r \}$ and $\overline{B(t,r)}:= \{x\in l^2: \|x-t\| \leq r \}$.

Let $t\in l^2$, $r>0$ and $e\in \mathbf{S}^1$. Define
a set $D(u,t,r,e)$ as follows:
$\al \in D(u,t,r,e)$ if and only if (\rmn1) $\al\in(0,1)$ and $S_{u, t, r}(e,\al)\in \mathbb{R}$,
(\rmn2) $S_{u, t, r}(e,\beta)=-\infty$ for all $\beta > \al$ or
$-\infty < \lim_{\beta\to\al+}S_{u, t, r}(e,\beta)< \lim_{\beta\to\al-}S_{u, t, r}(e,\beta)$.
From the monotonicity of $S_{u, t, r}(e,\cdot)$, $D(u,t,r,e)$ is countable.
Then proceed similarly to the proof of Lemma A.1 in \cite{huang} (Delete the ``,'' and add ``and all $e,f\in \mathbf{S}^1$,'' at the end of line 9 of page 82 in \cite{huang} and the narrative of this proof will become clearer),
in more detail, replacing $\mathbb{S}^{m-1}$ by $\mathbf{S}^1$ in this proof, we obtain
that $D(u,t,r) = \bigcup_{e\in \mathbf{S}^1} D(u,t,r,e)$ is countable.

Here we mention that
 the narrative of the proof of formula (A.6) in Theorem 5.1 in \cite{huang} can be slightly simplified.
The detailed operations are performed as follows:
replace
Lines 1 and 2 from the bottom in Page 82 and Lines 1 and 2 in Page 83 in \cite{huang}
by
\begin{itemize}
  \item
Since $2\langle a, b \rangle =  \| a \|^2 + \|b\|^2 - \|a-b\|^2$ for each $a,b \in \mathbb{R}^m$, then
$$
\langle e, x-q\rangle = \frac{    \langle   y-q,    x-q   \rangle    }    {  \|y-q\|    }  =  \frac{ \|x-q\|^2 + \|y-q\|^2 -\|x-y\|^2 } {2 \|y-q\| }
$$
\end{itemize}

Note that
$2\langle a, b \rangle =  \| a \|^2 + \|b\|^2 - \|a-b\|^2$ for each $a,b \in l^2$.
So proceed similarly to this slightly simplified proof of Theorem 5.1 in \cite{huang},
in more detail, replacing $\mathbb{R}^m$ by $l^2$, the definition of $\mathbb{Q}^m$ in Page 82 by $l^2_{\mathbb{Q}}:=\{(z_i)_{i=1}^{+\infty} \in l^2 : z_i\in \mathbb{Q} \mbox{ for each } i\in \mathbb{N}\}$, and $\mathbb{Q}^m$ by $l^2_{\mathbb{Q}}$ in this proof, we obtain that $D(u)$
is countable. Here we mention that $l^2_{\mathbb{Q}}$ is countable.
Of course, this fact can also be shown by proceeding similarly to the original proof of Theorem 5.1 in \cite{huang}.
\end{proof}

The following Theorem \ref{sdc} is
a generalization of
 Theorem 5.1 in \cite{huang} and Theorem \ref{ldc}.
Its proof is based on the well-known conclusion:
\begin{itemize}
  \item each separable metric space is homeomorphic to a subspace
of the Hilbert space $l^2$.
\end{itemize}

\begin{tm}\label{sdc}
 Let $(X,d)$ be a metric space and $u \in F(X)$. If $([u]_0, d)$ is separable, then the set
 $D(u)$
 is countable.
\end{tm}

\begin{proof}
Let $f$ be a homeomorphism
from
$([u]_0, d)$ to a subspace
of
$l^2$.
Consider $u_f \in F(l^2)$ defined by
\[
u_f(t)
=\left\{
   \begin{array}{ll}
     u(f^{-1} (t)), & t\in  f([u]_0), \\
          0, & t\in l^2 \setminus   f([u]_0).
   \end{array}
 \right.
\]
Then by Theorem \ref{ldc}, $D(u_f)$ is countable.
To show
that
$D(u)$
 is countable, it suffices to show that
$ D(u)= D(u_f)$.

For $S \subseteq [u]_0$,
let
$ \overline{S}^{[u]_0}$ denote the topological closure of $S$ in $([u]_0, d)$.
For
 $W \subseteq f([u]_0)$,
let
 $\overline{W}^{l^2} $ denote
the topological closure of $W$ in $l^2$,
and let
$\overline{W}^{f([u]_0)} $ denote the topological closure of $W$ in $f([u]_0)$, here we see $f([u]_0)$ as a subspace of $l^2$.
As described previously, for $S \subseteq X$, we use $\overline{S}$
to denote
the topological closure of $S$ in $X$.

We claim that
\\
(\romannumeral1) \ $D(u)=\{\al \in (0,1):  \mbox{ there exists an } x\in \{u>0\} \mbox{ such that } x \in [u]_\al \mbox{ and } x \notin \overline{\{u>\al\}}     \}$;
\\
(\romannumeral2) \ $ D(u_f)=\{\al\in (0,1): \mbox{ there exists a } y \in \{u_f>0\}\mbox{ such that } y\in [u_f]_\al \mbox{ and } y \notin \overline{\{u_f>\al\}}^{l^2}   \}$;
\\
(\romannumeral3) \
$ D(u_f)=\{\al\in (0,1): \mbox{ there exists an } x\in  \{u>0\}  \mbox{ such that } f(x) \in [u_f]_\al \mbox{ and } f(x) \notin \overline{\{u_f>\al\}}^{l^2}   \}$;
\\
(\romannumeral4) \
Let $\al\in [0,1]$ and $x\in [u]_0$. Then
\\
(\romannumeral4-1)  $x\in [u]_\al    \Leftrightarrow   f(x) \in [u_f]_\al, $
and
\\
 (\romannumeral4-2)
$x \notin \overline{\{u>\al\}} \Leftrightarrow  x \notin \overline{\{u>\al\}}^{[u]_0}
  \Leftrightarrow    f(x) \notin \overline{\{u_f>\al\}}^{f([u]_0)}
\Leftrightarrow
f(x) \notin \overline{\{u_f>\al\}}^{l^2} $.

Clearly (\romannumeral1) and (\romannumeral2) are true.
Let $y \in l^2$. Then $y\in \{u_f > 0 \}$ if and only if there exists a unique $x\in  \{u>0\}$ such that $y = f(x)$.
So by (\romannumeral2), (\romannumeral3) is true.

Note that
for each $x\in [u]_0$, $u(x)=u_f(   f(x))$.
So (\romannumeral4-1) is true.

Now we show (\romannumeral4-2).
Since $x\in [u]_0$, the first ``$\Leftrightarrow$'' holds. Since $f(x)\in f([u]_0)$, the third ``$\Leftrightarrow$'' holds.  Note that $f(\{u>\al\}) = \{u_f >\al\}$. As
$f$ is a homeomorphism
from
$([u]_0, d)$ to $f([u]_0)$, it follows that $f(\overline{\{u>\al\}}^{[u]_0})= \overline{\{u_f>\al\}}^{f([u]_0)} $. Thus the second ``$\Leftrightarrow$'' holds.
So (\romannumeral4-2) is true.

From (\romannumeral1), (\romannumeral3) and (\romannumeral4),
$D(u) = D(u_f)$.
\end{proof}

\begin{tl}\label{smc}
   Let $(X,d)$ be a separable metric space and $u \in F(X)$. Then the set
 $D(u)$
 is countable.
\end{tl}

\begin{proof}
  As $([u]_0,d)$ is a subspace of the separable metric space $(X,d)$, it follows that $([u]_0,d)$ is separable. Thus,
by Theorem \ref{sdc}, $D(u)$ is countable.

\end{proof}

\begin{re}{\rm
It is well-known that
  both $\mathbb{R}^m$ and $l^2$ are separable metric spaces.
Thus
both Theorem 5.1 in \cite{huang}
and
Theorem \ref{ldc}
are special cases of Corollary \ref{smc}, which is a corollary of Theorem \ref{sdc}.
So
Theorem \ref{sdc} is a generalization
of Theorem 5.1 in \cite{huang}
and
Theorem \ref{ldc}.

}
\end{re}

\begin{re} \label{spc}
{\rm
 Theorems \ref{ldc}, \ref{sdc} and Corollary \ref{smc}
remain true if $D(u)$ is replaced by $P(u)$, since $P(u) \subseteq D(u)$ for each $u\in F(X)$.

Let $u\in F(X)$. Define $F(u) := \{ \al\in (0,1): \overline{\{u>\al\}} \not= [u]_\al\}$. Clearly $(0,1) \setminus F(u)   \subseteq (0,1) \setminus D(u)$; that is, $D(u) \subseteq F(u)$.
$D(u) \subsetneqq F(u)$ may happen.
For instance, let $u$ be the fuzzy set defined in Example \ref{pdr}. Then $F(u) = (0,1)\supsetneqq \{0.6\}=D(u)$.

For each $u\in F(X)$, $P(u) \subseteq D(u)\subseteq F(u)$.
We claim that if $u\in F^1_{USC} (X)$, then $P(u)= D(u)=F(u)$.
Let
$u\in F^1_{USC} (X)$.
Then for each $\al\in [0,1]$, $\overline{\{u>\al\}} \subseteq [u]_\al$.
Thus
$F(u)\subseteq P(u)$, and so $P(u)= D(u)=F(u)$.

}
\end{re}

\begin{tm} \label{gusc}
    Suppose that $u$, $u_n$, $n=1,2,\ldots$, are fuzzy sets
  in $F^1_{USC} ( X )$. Then the following statements are true.
\\
(\romannumeral1) If there is a dense set $P$ in $(0,1)$ such that $[u]_\al=\lim^{(K)}_{n\to \infty} [u_n]_\al $ for $\al\in P$, then $u = \lim_{n\to \infty}^{(\Gamma)}  u_n$.
\\
(\romannumeral2)    If $u = \lim_{n\to \infty}^{(\Gamma)}  u_n $, then $[u]_\al=\lim^{(K)}_{n\to \infty} [u_n]_\al $ for all $\al\in (0,1) \setminus P(u)$.

\end{tm}

\begin{proof}

 The proof of (\romannumeral1)
is similar to
  ``(\romannumeral2) $\Rightarrow$ (\romannumeral1)''
in the proof of
   Theorem 6.2 in \cite{huang}. Here we mention the following basic facts.
 Let $\al\in(0,1)$. Then there exist two sequences $\{\al_n\}$ and $\{\beta_n\}$ in $P$ with $\al_n\to\al+$ and
   $\beta_n\to\al-$. Indeed, we can get two such sequences by
  choosing $\al_n$ in $P\cap (\al, \al+(1-\al)/n)$ and $\beta_n$ in $P\cap ((1-1/n)\al,\al)$ for each $n\in \mathbb{N}$.
 Let $\al=1$. Then there exists a sequence $\{\gamma_n\}$ in $P$ with $\gamma_n\to 1-$. Indeed, we can get a such sequence by
 choosing $\gamma_n$ in $P\cap (1-1/n, 1)$ for each $n\in \mathbb{N}$.

 (\romannumeral2) follows immediately from Theorem \ref{Gclnre}.
\end{proof}

The following theorem gives some conditions under which
the $\Gamma$-convergence of fuzzy sets can be
decomposed to
the Kuratowski convergence of certain $\al$-cuts.
For simplicity, we use
``a.e.'' to denote ``almost everywhere''.

\begin{tm} \label{gdm}
Suppose that $u$, $u_n$, $n=1,2,\ldots$, are fuzzy sets
  in $F^1_{USC} ( X )$. If $([u]_0, d)$ is separable, then the following are equivalent:
\\
(\romannumeral1) $\lim_{n\to \infty}^{(\Gamma)}  u_n = u $;
\\
(\romannumeral2) $\lim^{(K)}_{n\to \infty} [u_n]_\al = [u]_\al$ holds a.e. on $\al \in (0,1)$;
\\
(\romannumeral3) $\lim^{(K)}_{n\to \infty} [u_n]_\al = [u]_\al$ holds for all $\al\in (0,1) \setminus P(u)$;
\\
(\romannumeral4) There is a dense subset $P$ of $(0,1) \backslash P(u)$ such that $\lim^{(K)}_{n\to \infty} [u_n]_\alpha = [u]_\al$ holds for
$\al\in P$;
\\
(\romannumeral5) There is a countable dense subset $P$ of $(0,1) \backslash P(u)$ such that $\lim^{(K)}_{n\to \infty} [u_n]_\alpha = [u]_\al$ holds for $\al\in P$.

\end{tm}

\begin{proof}
Suppose that $u$, $u_n$, $n=1,2,\ldots$, are fuzzy sets
  in $F^1_{USC} ( X )$. We claim that
\\
(a) (\romannumeral2)$\Rightarrow$(\romannumeral1)$\Rightarrow$(\romannumeral3)$\Rightarrow$(\romannumeral4)$\Leftrightarrow$(\romannumeral5).
\\
(b) If $(0,1) \backslash P(u)$ is dense in $(0,1)$, then
(\romannumeral1)$\Leftrightarrow$(\romannumeral3)$\Leftrightarrow$(\romannumeral4)$\Leftrightarrow$(\romannumeral5).
\\
(c) If $P(u)$ is a set of measure zero, then
(\romannumeral1)$\Leftrightarrow$(\romannumeral2)$\Leftrightarrow$(\romannumeral3)$\Leftrightarrow$(\romannumeral4)$\Leftrightarrow$(\romannumeral5).

(\romannumeral2)$\Rightarrow$(\romannumeral1) follows from Theorem \ref{gusc} (\romannumeral1). (\romannumeral1)$\Rightarrow$(\romannumeral3) is Theorem \ref{gusc} (\romannumeral2).
Clearly
 (\romannumeral3)$\Rightarrow$(\romannumeral4)$\Leftrightarrow$(\romannumeral5).
So (a) is true.
If $(0,1) \backslash P(u)$ is dense in $(0,1)$, then a
dense subset
 $P$ of $(0,1) \backslash P(u)$ is also a dense subset of $(0,1)$.
 So
by
Theorem \ref{gusc} (\romannumeral1),
 (\romannumeral4)$\Rightarrow$(\romannumeral1). Combined with (a), we have that (b) is true.

 Assume that $P(u)$ is a set of measure zero.
Then (\romannumeral3)$\Rightarrow$(\romannumeral2).
At this time, $(0,1) \backslash P(u)$ is dense in $(0,1)$.
So from (a) and (b), we obtain that (c) is true.

By Theorem \ref{sdc} and Remark \ref{spc}, we know that if $([u]_0, d)$ is separable, then $P(u)$ is countable,
 and hence
$P(u)$ is a set of measure zero.
So by (c), we obtain the desired results.
\end{proof}

\begin{re} \label{pvr}
{\rm
By Lemma \ref{gnc}, if $u\in F^1_{USCG}(X)$ then $P(u)=P_0(u)$ is countable.
So by affirmation (c) in the proof of Theorem \ref{gdm},
 Theorem \ref{gdm}
remains true if  ``$([u]_0, d)$ is separable'' is replaced by ``$u\in F^1_{USCG}(X)$''.

Here we mention that the condition ``$u\in F^1_{USCG}(X)$'' implies the condition ``$([u]_0, d)$ is separable''.
Let
$u\in F^1_{USCG}(X)$. Then for each $\al\in (0,1]$, $([u]_\al, d)$ is separable, because each compact metric space is separable.
Since $[u]_0 = \overline{\cup_{n=1}^{+\infty} [u]_{1/n} }$,
it follows that
 $([u]_0, d)$ is separable
(For $n=1,2,\ldots$, let $A_n$ be a countable dense set of $([u]_{1/n},d)$.
Then
$\bigcup_{n=1}^{+\infty}A_n$ is a countable dense set of $([u]_0, d)$).

  }
\end{re}

\section{Level characterizations of endograph metric convergence}\label{lce}

In this section, we discuss the level characterizations of endograph metric convergence.
It is shown that under some condition,
the $H_{\rm end}$ metric convergence of fuzzy sets
can be decomposed to
the Hausdorff metric convergence of certain $\al$-cuts.

Let $u$ be a fuzzy set in $F^1_{USC} (X)$. Denote
$$P_0(u):   = (0,1) \setminus   \{\al\in (0,1) :   \lim_{\beta \to \al}  H([u]_\beta,  [u]_\al) = 0 \}.$$
If $\al\in P_0(u)$, then $\lim_{\beta \to \al}  H([u]_\beta,  [u]_\al)$ may not exist.
In
this paper the $\lim$ quantities are allowed to take the value $+\infty$.

\begin{pp}\label{prea}
 Let $u\in F^1_{USC} (X)$. Then $P(u) \subseteq   P_0(u)$.
\end{pp}

\begin{proof}
 If $\al\in P(u)$, then
 $\overline{\{u>\al\}} \subsetneqq  [u]_\al$, i.e. $\xi:= H(\overline{\{u>\al\}},   [u]_\al) >0$.
Thus $\lim_{\gamma \to \al+}  H([u]_\gamma,  [u]_\al)\geq \xi >0.$
Hence $\al \in P_0(u)$. So $P(u) \subseteq   P_0(u)$.
\end{proof}

$P(u) \subsetneqq   P_0(u)$ could happen.
See Example \ref{rse}.

\begin{eap} \label{rse}
{\rm
Let $u \in F^1_{USC} (\mathbb{R}^2)$ be defined by putting
\[
[u]_\al=\{0\} \cup  \{z: \arg z \in [\al,1]\} \ \mbox{for each} \ \al\in [0,1],
\]
here we write each $(x,y) \in \mathbb{R}^2$ as a complex number $z = x+iy$.
Then $P(u) = \emptyset$ and $P_0(u) = (0,1)$. So $P(u) \subsetneqq   P_0(u)$.

This example also shows that for $u\in F^1_{USC}(X)$,
$P_0(u)$ need not be countable even $X$ is a separable metric space.
  }
\end{eap}

Below Lemma \ref{c} and Lemma \ref{gnc}(\rmn1)(\rmn2)(\rmn5) may be known. Here we
give their
proofs for the completeness of this paper.

\begin{lm} \label{c}
For $n=1,2,\ldots$, let $U_n\in K(X)$ and $V_n \in K(X)$.
\\
(\romannumeral1) \ If $U_1\supseteq U_2 \supseteq \ldots \supseteq U_n\supseteq \ldots$, then
$U=\bigcap_{n=1}^{+\infty}  U_n    \in K(X)$ and $H(U_n, U) \to 0$ as $n\to +\infty$.
\\
(\romannumeral2) \ If $V_1\subseteq V_2 \subseteq \ldots \subseteq V_n\subseteq \ldots$
and
$V = \overline{\bigcup_{n=1}^{+\infty}  V_n }   \in K(X)$, then $H(V_n, V) \to 0$ as $n\to +\infty$.
\end{lm}

\begin{proof}
Here we give a proof using Theorem \ref{bfc}.
We only show (\romannumeral2). (\romannumeral1) can be shown in a similar way.
Clearly for $n=1,2,\ldots$, $V_n \in K(V)$. Then by Theorem \ref{bfc}, $\{V_n\}$ has a subsequence $\{V_{n_k}\}$ which
converges to
$C\in K(V)$.
Then $H(V_n, C) \to 0$ since for $V_{n_{k_1}}\subseteq V_n \subseteq V_{n_{k_2}}$, $H(V_n, C) \leq \max\{H^*(V_{n_{k_2}},C), H^*(C, V_{n_{k_1}})\} = \max\{H(V_{n_{k_2}},C), H(C, V_{n_{k_1}})\} $. Thus by Theorem \ref{hkg}, $C = \lim^{(K)}_{n\to\infty} V_n = V$.
So (\romannumeral2) is true.
\end{proof}

\begin{lm} \label{gnc}
Let $u\in F^1_{USCG}(X)$.
\\
(\romannumeral1) \ For each $\al\in (0,1]$, $\lim_{\beta\to \al-} H([u]_\beta, [u]_\al) =0$.
\\
(\romannumeral2) \ For each $\al\in (0,1)$, $\lim_{\gamma\to \al+} H([u]_\gamma, \overline{\{u>\al\}}) =0$.
\\
(\romannumeral3) \ $P(u) =   P_0(u)$.
\\
(\romannumeral4) \ $P_0 (u)$ is countable.
\\
(\romannumeral5) \ If $u\in F^1_{USCB}(X)$, then $\lim_{\gamma\to 0+} H([u]_\gamma, [u]_0) =0$.
\end{lm}

\begin{proof}
Assume that (\rmn1) is not true. Then there is an $\al\in (0,1]$, an $\varepsilon>0$
and
an increasing sequence $\{\beta_n\}$ in $[0,\al)$ with $\beta_n \to \al-$
such that
$H([u]_{\beta_n}, [u]_\al)>\varepsilon$.
By Lemma \ref{c} (\rmn1), $H([u]_{\beta_n}, [u]_\al)\to 0$. This is a contradiction.
So (\rmn1) is true. Similarly, by Lemma \ref{c} (\rmn2),
 (\romannumeral2) and  (\romannumeral5) are true.

Now we show (\romannumeral3).
Given $\al\in (0,1)\setminus P(u)$. Then $[u]_\al = \overline{\{u>\al\}}$ .
By
 (\romannumeral1) and (\romannumeral2), we have that $\lim_{\delta\to \al} H([u]_\delta, [u]_\al) =0$.
Thus
$\al\in (0,1)\setminus P_0(u)$.
So $ (0,1)\setminus P(u) \subseteq (0,1)\setminus P_0(u)$.
This means that $P_0(u) \subseteq P(u)$.
Combined with Proposition \ref{prea}, we obtain
that $P(u) = P_0(u)$.
So (\romannumeral3) is proved.

(\romannumeral4) is Lemma 6.12 in \cite{huang17}.
(\romannumeral4)
can also be shown in such a way:
let $u\in F^1_{USCG}(X)$, then by
Remark \ref{pvr}, $([u]_0, d)$ is separable, and thus by Theorem \ref{sdc} and Remark \ref{spc},
$P(u)$ is countable. So by (\romannumeral3),
 $P_0 (u)$ is countable.
\end{proof}

\begin{pp}\label{ecn}
  (\romannumeral1) \ If $u,v \in  F^1_{USCG}(X)$, then $H([u]_\al, [v]_\al) $ is
 left-continuous at each $\al\in (0,1]$.
  (\romannumeral2) \ If $u,v \in  F^1_{USCB}(X)$, then $H([u]_\al, [v]_\al) $
is
 right-continuous at $\al=0$.
\end{pp} 

 \begin{proof} Let $u,v \in  F^1_{USCG}(X)$.
 For each $\al\in (0,1]$,
by
 Lemma \ref{gnc}(\romannumeral1), $ \lim_{\beta\to \al- } ( H([u]_\al, [u]_\beta) + H([v]_\al,   [v]_\beta) ) = 0$.
 Note that $H([u]_\al, [v]_\al)$ is finite at each $\al \in (0,1]$
and
for each $\alpha, \beta \in (0,1]$,
$ |H([u]_\al, [v]_\al) - H([u]_\beta, [v]_\beta) |  \leq
   H([u]_\al, [u]_\beta) + H([v]_\al,   [v]_\beta).
$
So for each $\al\in (0,1]$,
  $\lim_{\beta \to \al-}  H([u]_\beta, [v]_\beta) = H([u]_\alpha, [v]_\alpha), $
  and (\rmn1) is true.
  Similarly, it can be shown that (\rmn2) is true by using Lemma \ref{gnc}(\romannumeral5).
   \end{proof}

\begin{pp}\label{aen}
 Let $u,v\in F^1_{USC}(X)$ and
$\varepsilon\in (0,1]$. Let $\alpha,\beta\in [0,1]$ with $\alpha-\beta\geq \varepsilon$.
If
$
H^*({\rm end}\, u, {\rm end}\, v)<\varepsilon$, then $H^*([u]_{\alpha}, [v]_\beta)\leq H^*({\rm end}\, u, {\rm end}\, v)<\varepsilon$.
\end{pp}

\begin{proof}
Set $a:= H^*([u]_{\al}, \overline{\{v>\beta\}})$
and $b:= \sup\{\overline{d}((y,\al), {\rm end}\, v): y\in [u]_\al\}$.
We claim the following property (\rmn1) if
$
b<\varepsilon$, then $a\leq b$.

Assume that $b<\varepsilon$.
Let $y\in [u]_\al$.
Notice that $\overline{d}((y,\al), {\rm end}\, v)=\inf\{\overline{d}((y,\al), (z,\gamma)): (z,\gamma)\in {\rm end}\, v \}\leq b <\varepsilon$, and that for $(z,\gamma) \in {\rm end}\, v$ with $\gamma \leq \beta$,
$\overline{d}((y,\al), (z,\gamma)) \geq |\al-\beta|\geq \varepsilon$. Thus
$\overline{d}((y,\al), {\rm end}\, v)=\inf\{\overline{d}((y,\al), (z,\gamma)): (z,\gamma)\in {\rm end}\, v \mbox{ with } \gamma> \beta\}$.
Since
for each
$(z,\gamma)\in {\rm end}\, v \mbox{ with } \gamma > \beta$,
$\overline{d}((y, \al), (z, \gamma)) \geq d(y, z) \geq d(y, \overline{\{v>\beta\}})$,
it follows that
$\overline{d}((y,\al), {\rm end}\, v)
\geq d(y, \overline{\{v>\beta\}}) $.
Since $y\in [u]_\al$ is arbitrary, we have that $a\leq b$.
So (\rmn1) is true.

If $
H^*({\rm end}\, u, {\rm end}\, v)<\varepsilon$, then
$b<\varepsilon$, and hence by (\rmn1), $H^*([u]_{\alpha}, [v]_\beta)\leq a\leq
b \leq H^*({\rm end}\, u, {\rm end}\, v)<\varepsilon$.
So the proof is complete.
\end{proof}

\begin{tm} \label{lres}
  Let $u$, $u_n$, $n=1,2,\ldots$, be fuzzy sets in $F^1_{USC} (X)$.

(\romannumeral1) \ The following are equivalent:
\\
(\romannumeral1-1) \
$\lim_{n\to \infty} H^*({\rm end}\, u, {\rm end}\, u_n) = 0$;
\\
(\romannumeral1-2) \
For each $\al\in [0,1)$ and $\xi\in (0, 1-\alpha]$,
$\lim_{n\to \infty} H^*([u]_{\al+\xi},    [u_n]_\al) = 0$;
\\
(\romannumeral1-3) \
There is a dense subset $P$ of $[0,1)$
such that
for each $\al\in P$ and $\xi\in (0, 1-\alpha]$,
$\lim_{n\to \infty} H^*([u]_{\al+\xi},    [u_n]_\al) = 0$.

(\romannumeral2) \  The following are equivalent:
\\
(\romannumeral2-1) \
$\lim_{n\to\infty} H^*{(\rm end}\, u_n, {\rm end}\, u) = 0$;
\\
(\romannumeral2-2) \
for each $\al\in (0,1]$ and $\zeta \in (0, \al]$,
$\lim_{n\to \infty} H^*([u_n]_\al, [u]_{\al-\zeta}) = 0$;
\\
(\romannumeral2-3) \
There is a dense subset $P$ of $(0,1]$
such that for each $\al\in P$ and $\zeta \in (0, \al]$,
$\lim_{n\to \infty} H^*([u_n]_\al, [u]_{\al-\zeta}) = 0$.

\end{tm}

\begin{proof} \ We only prove (\romannumeral1). (\romannumeral2) can be proved similarly.

To show (\romannumeral1-1)$\Rightarrow$(\romannumeral1-2),
assume that $\lim_{n\to \infty} H^*({\rm end}\, u, {\rm end}\, u_n) = 0$. Let $\al\in [0,1)$ and $\xi\in (0, 1-\alpha]$.
Then for each $\varepsilon \in (0, \xi)$, there
exists an $N(\varepsilon)$ such that for all $n\geq N$,
$
H^*({\rm end}\, u, {\rm end}\, u_n)  < \varepsilon
$, and hence by Proposition \ref{aen}, $H^*( [u]_{\al+\xi},  [u_n]_\al) < \varepsilon$.
Since $\varepsilon \in (0, \xi)$ is arbitrary,
it follows that
$\lim_{n\to \infty} H^*([u]_{\al+\xi},    [u_n]_\al) = 0$.
So
 (\romannumeral1-1)$\Rightarrow$(\romannumeral1-2).

Assume that
(\romannumeral1-2)
is true. To show that (\romannumeral1-1) is true,
let $\varepsilon>0$. Select a $k \in \mathbb{N}$ with $1/k\leq \varepsilon$.
From (\romannumeral1-2), we have that for $l=1,\ldots, k$,
$\lim_{n\to\infty} H^*(  [u]_{l/k},  [u_n]_{(l-1)/k}    ) = 0$.
So
there is an $N(\varepsilon)$ such that for all $n\geq N$ and $l=1,\ldots, k$,
\begin{equation}\label{lse}
  H^*(  [u]_{l/k},  [u_n]_{(l-1)/k}    )  \leq \varepsilon.
\end{equation}

Let $(x,\al) \in {\rm end}\, u$ and $n\in \mathbb{N}$. If $\al\leq \varepsilon$,
then
$\overline{d}((x,\al), {\rm end}\, u_n)\leq \overline{d}((x,\al), (x,0))\leq \varepsilon$.
If $\al > \varepsilon$,
then we can
choose an $l \in \{1,\ldots, k-1\}$ such that $l/k < \al  \leq (l+1)/k$,
and thus
\begin{align*}
&
\overline{d}((x,\al),  {\rm end}\, u_n ) \leq \overline{d}((x,\al),\  [u_n]_{(l-1)/k} \times \{(l-1)/k\}   )\\
&= \inf\{d(x,y) + |\al-\frac{l-1}{k}|: y \in  [u_n]_{(l-1)/k} \}=d(x, [u_n]_{(l-1)/k} ) + |\al-\frac{l-1}{k}|.
\end{align*}
So by \eqref{lse}, for $n\geq N$,
$$
H^* ( {\rm end}\, u, {\rm end}\, u_n )\leq \varepsilon \vee
 (\max\{H^*([u]_{l/k},  [u_n]_{(l-1)/k}   ), \ l=1,\ldots,k-1\} + 2/k) \leq 3\varepsilon.
$$
Since
 $\varepsilon>0$ is arbitrary, we have that
$\lim_{n\to\infty} H^* ( {\rm end}\, u, {\rm end}\, u_n ) = 0$.
So (\romannumeral1-2)$\Rightarrow$(\romannumeral1-1).

Let
  $\al\in [0,1)$ and $\xi\in (0, 1-\alpha]$.
Choose $\beta \in P \cap [\al, \al+\xi)$.
Then
\begin{equation*}
  H^*([u]_{\al+\xi},    [u_n]_\al) \leq H^*([u]_{\al+\xi},    [u_n]_\beta).
\end{equation*}
Using this fact, we obtain that
(\romannumeral1-3)$\Rightarrow$(\romannumeral1-2).

(\romannumeral1-2)$\Rightarrow$(\romannumeral1-3) is obvious.
Thus (\romannumeral1) is proved.
\end{proof}

\begin{re} \label{acme} {\rm
 Let $u$, $u_n$, $n=1,2,\ldots$, be fuzzy sets in $F^1_{USC} (X)$.
Clearly, (\romannumeral1-2) in Theorem \ref{lres} is equivalent to the following (\romannumeral1-2)$'$,
and
(\romannumeral2-2) in Theorem \ref{lres} is equivalent to the following (\romannumeral2-2)$'$:
\\
(\romannumeral1-2)$'$ \
For each $\al\in [0,1)$ there is a sequence $\{\xi_m\}$ in $(0, 1-\alpha]$ with $\xi_m \to 0+ $ satisfying
that
$\lim_{n\to \infty} H^*([u]_{\al+\xi_m},    [u_n]_\al) = 0$;
\\
(\romannumeral2-2)$'$ \
For each $\al\in (0,1]$ there is a sequence $\{\zeta_m\}$ in $(0, \al]$ with $\zeta_m \to 0+ $ satisfying
that
$\lim_{n\to \infty} H^*([u_n]_\al, [u]_{\al-\zeta_m}) = 0$.

Similarly, we can give
(\romannumeral1-3)$'$ and (\romannumeral2-3)$'$ which are equivalent to (\romannumeral1-3) and (\romannumeral2-3) in Theorem \ref{lres},
respectively.

}
\end{re}

\begin{tl} \label{lrec}
  Let $u$, $u_n$, $n=1,2,\ldots$, be fuzzy sets in $F^1_{USC} (X)$.
Then
the following are equivalent:
\\
(\romannumeral1) \
$\lim_{n\to\infty} H_{\rm end} (u_n, u) = 0$;
\\
(\romannumeral2) \ For each $\al\in (0,1)$,
$\lim_{n\to\infty} H^*([u]_{\al+\xi},    [u_n]_\al) = 0$ when $\xi\in (0, 1-\alpha]$,
and
$\lim_{n\to\infty} H^*([u_n]_\al, [u]_{\al-\zeta}) = 0$ when $\zeta \in (0, \al]$;
\\
(\romannumeral3) \ There is a dense subset $P$ of $(0,1)$
such that
for each $\al\in P$,
$\lim_{n\to\infty} H^*([u]_{\al+\xi},    [u_n]_\al) = 0$ when $\xi\in (0, 1-\alpha]$,
and
$\lim_{n\to\infty} H^*([u_n]_\al, [u]_{\al-\zeta}) = 0$ when $\zeta \in (0, \al]$.
\end{tl}

\begin{proof}
Note that $\lim_{n\to\infty} H_{\rm end} (u_n, u) = 0$ if and only if
$\lim_{n\to \infty} H^*({\rm end}\, u, {\rm end}\, u_n) = 0$
and
$\lim_{n\to\infty} H^*{(\rm end}\, u_n, {\rm end}\, u) = 0$.
So
 Theorem \ref{lres} imply that (\romannumeral1)$\Rightarrow$(\romannumeral2)
and
 (\romannumeral3)$\Rightarrow$(\romannumeral1).
Clearly (\romannumeral2)$\Rightarrow$(\romannumeral3).
So
(\romannumeral1)$\Leftrightarrow$(\romannumeral2)$\Leftrightarrow$(\romannumeral3).
\end{proof}

\begin{lm} \label{emr}
 Let $u$, $u_n$, $n=1,2,\ldots$, be fuzzy sets in $F^1_{USC} (X)$, and
 let $P$ be a dense subset of
$(0,1)$.
\\
(\romannumeral1) \
  If for each $\al\in P$,
$\lim_{n\to\infty} H^*(\overline{\{u>\al\}},    [u_n]_\al) = 0$,
then
$\lim_{n\to \infty} H^*({\rm end}\, u, {\rm end}\, u_n) = 0$;
\\
(\romannumeral2) \
  If for each $\al\in P$,
$\lim_{n\to\infty} H^*([u_n]_\al, [u]_{\al}) = 0$,
then
$\lim_{n\to\infty} H^*{(\rm end}\, u_n, {\rm end}\, u) = 0$;
\\
(\romannumeral3) \
  If for each $\al\in P$,
$\lim_{n\to\infty} H^*(\overline{\{u>\al\}},    [u_n]_\al) = 0$
and
$\lim_{n\to\infty} H^*([u_n]_\al, [u]_{\al}) = 0$,
then
$\lim_{n\to\infty} H_{\rm end} (u_n, u) = 0$.
\end{lm}

\begin{proof}  The desired results follow immediately from Theorem \ref{lres}.
\end{proof}

\begin{lm} \label{acm}
   Let $u$, $u_n$, $n=1,2,\ldots$, be fuzzy sets in $F^1_{USC} (X)$.
\\
(\romannumeral1) \ Let $\al\in [0,1)$. If $\lim_{n\to\infty} H^*({\rm end}\, u, {\rm end}\, u_n) = 0$, and
$\lim_{\gamma \to \alpha+} H( [u]_{\gamma}, \overline{\{u>\al\}}) = 0$,
then
 $\lim_{n\to\infty} H^*(\overline{\{u>\al\}},    [u_n]_\al) = 0$.
\\
(\romannumeral2) \ Let $\al\in (0,1]$. If $\lim_{n\to\infty} H^*{(\rm end}\, u_n, {\rm end}\, u) = 0$,
and
$\lim_{\beta \to \alpha-} H( [u]_\al, [u]_{\beta}) = 0$,
then
 $\lim_{n\to\infty} H^*( [u_n]_\al, [u]_\al) = 0$.
\end{lm}

\begin{proof} \ We only prove (\romannumeral1). (\romannumeral2) can be proved similarly.

  Let $\varepsilon>0$. Since $\lim_{\gamma \to \alpha+} H( [u]_{\gamma}, \overline{\{u>\al\}}) = 0$,
then there is a $\gamma( \alpha) \in (\al, 1]$
such that
$H(\overline{\{u>\al\}}, [u]_\gamma)<\varepsilon/2$.
By Theorem \ref{lres} (\romannumeral1),  $\lim_{n\to\infty} H^*({\rm end}\, u, {\rm end}\, u_n) = 0$ implies that
$\lim_{n\to \infty} H^*([u]_{\gamma},    [u_n]_\al) = 0$.
Then there is an $N\in \mathbb{N}$ such that
for all $n\geq N$,
$H^*([u]_{\gamma},    [u_n]_\al) < \varepsilon/2$.
Hence
for all $n\geq N$,
$$H^*(\overline{\{u>\al\}},    [u_n]_\al) \leq H(\overline{\{u>\al\}}, [u]_\gamma)  + H^*([u]_{\gamma},    [u_n]_\al) < \varepsilon.$$
From the arbitrariness of $\varepsilon>0$,
we thus have
$\lim_{n\to\infty} H^*(\overline{\{u>\al\}},    [u_n]_\al) = 0$.

\end{proof}

The assumption that
$\lim_{\gamma \to \alpha+} H( [u]_{\gamma}, \overline{\{u>\al\}}) = 0$
in
(\romannumeral1) of
Lemma \ref{acm} can not be omitted.
The assumption that
$\lim_{\beta \to \alpha-} H( [u]_\al, [u]_{\beta}) = 0$
in
(\romannumeral2) of
Lemma \ref{acm} also can not be omitted.
The following Examples \ref{snc} and \ref{fnc}
are counterexamples.

\begin{eap} \label{snc}
  {\rm

Let $u$ be a fuzzy set in $F^1_{USC}(\mathbb{\mathbb{R}})$ defined by putting
\[
[u]_\al
=\left\{
   \begin{array}{ll}
(-\infty, \frac{\frac{2}{3}}{\alpha-\frac{1}{3}}], & \al\in (\frac{1}{3}, 1], \\
    (-\infty, +\infty), & \al\in [0,\frac{1}{3}].
   \end{array}
 \right.
\]
For $n=1,2,\ldots$, let $u_n$ be a fuzzy set in $F^1_{USC}(\mathbb{\mathbb{R}})$ given by putting
\[
[u_n]_\al
=\left\{
   \begin{array}{ll}
(-\infty, \frac{1-\frac{1}{3}\frac{n-1}{n}}{\alpha-\frac{1}{3}\frac{n-1}{n}}], & \al\in (\frac{1}{3}\frac{n-1}{n}, 1], \\
    (-\infty, +\infty), & \al\in [0,   \frac{1}{3}\frac{n-1}{n}].
   \end{array}
 \right.
\]
Clearly
 $\lim_{\gamma \to \frac{1}{3}+} H( [u]_{\gamma}, \overline{\{u>\frac{1}{3}\}})
= \lim_{\gamma \to \frac{1}{3}+} H( [u]_{\gamma}, (-\infty, +\infty))
= +\infty \not= 0$.

It can be seen that
$H_{\rm end} (u, u_n) \to 0$.
However for each $n\in \mathbb{N}$,
$$H^*(\overline{\{u > \frac{1}{3}\}},   [u_n]_\frac{1}{3} )
=H^*([u]_\frac{1}{3},   [u_n]_\frac{1}{3} )
=
H^*((-\infty, +\infty),\ (-\infty, \frac{3-\frac{n-1}{n}}{1-\frac{n-1}{n}}]) =  +\infty,$$
and so $H^*(\overline{\{u > \frac{1}{3}\}},   [u_n]_\frac{1}{3} ) \not\to 0.$

}
\end{eap}

\begin{eap} \label{fnc} {\rm
Let $u$ be a fuzzy set in $F^1_{USC}(\mathbb{\mathbb{R}})$ defined by putting
\[
[u]_\al
=
\left\{
  \begin{array}{ll}
    \{1\}, & \al=1, \\
   \{1\} \cup (-\infty, -\frac{1}{1-\al}], & \al\in [0,1),
  \end{array}
\right.
\]
For $n=1,2,\ldots$,
 let $u_n$ be a fuzzy set in $F^1_{USC}(\mathbb{\mathbb{R}})$ given by putting
\[
[u_n]_\al
 =
\left\{
  \begin{array}{ll}
    [u]_\al, & \al\in [0, 1-\frac{1}{n}], \\
   \mbox{} [u]_{1-\frac{1}{n}}, &\al\in [1-\frac{1}{n},1],
  \end{array}
\right.
n=1,2,\ldots.
\]
Clearly $\lim_{\beta \to 1-} H( [u]_1, [u]_{\beta}) = +\infty \not= 0$.

We can see that
$H_{\rm end} (u, u_n) \to 0$.
However for each $n\in \mathbb{N}$,
$$H^*([u_n]_1, [u]_1) = H^*( \{1\} \cup (-\infty, -n],\  \{1\}) =  +\infty,$$
and so
$H^*([u_n]_1, [u]_1) \not\to 0$.

}
\end{eap}

\begin{tm} \label{uscg}
  Let $u$ be a fuzzy set in $F^1_{USCG} (X)$ and let $u_n$, $n=1,2,\ldots$, be fuzzy sets in $F^1_{USC} (X)$.

(\romannumeral1) \ The following are equivalent:
\\
(\romannumeral1-1) \
$\lim_{n\to\infty} H^*({\rm end}\, u, {\rm end}\, u_n) = 0$;
\\
(\romannumeral1-2) \
For each $\al\in (0,1)$,
$\lim_{n\to\infty} H^*(\overline{\{u>\al\}},    [u_n]_\al) = 0$;
\\
(\romannumeral1-3) \
There is a dense subset $P$ of $(0,1)$ such that
for each $\al\in P$,
$\lim_{n\to\infty} H^*(\overline{\{u>\al\}},    [u_n]_\al) = 0$.

(\romannumeral2) \ The following are equivalent:
\\
(\romannumeral2-1) \
$\lim_{n\to\infty} H^*{(\rm end}\, u_n, {\rm end}\, u) = 0$;
\\
(\romannumeral2-2) \
For each $\al\in (0,1]$,
$\lim_{n\to\infty} H^*([u_n]_\al, [u]_{\al}) = 0$;
\\
(\romannumeral2-3) \
There is a dense subset $P$ of $(0, 1]$
such that for each $\al\in P$,
$\lim_{n\to\infty} H^*([u_n]_\al, [u]_{\al}) = 0$.
\end{tm}

\begin{proof} \  The desired results follow from Lemmas \ref{gnc}, \ref{emr} and \ref{acm} (or Lemmas \ref{gnc} and \ref{acm} and Theorem \ref{lres}). The proof is routine.
\end{proof}

\begin{tm} \label{uscgre}
  Let $u$ be a fuzzy set in $F^1_{USCG} (X)$ and let $u_n$, $n=1,2,\ldots$, be fuzzy sets in $F^1_{USC} (X)$.
Then the following are equivalent:
\\
(\romannumeral1) \
 $\lim_{n\to\infty} H_{\rm end}(u_n, u) = 0$;
\\
(\romannumeral2) \
For each $\al\in (0,1)$,
$\lim_{n\to\infty} H^*(\overline{\{u>\al\}},    [u_n]_\al) = 0$
and
$\lim_{n\to\infty} H^*([u_n]_\al, [u]_{\al}) = 0$;
\\
(\romannumeral3) \
There is a dense subset $P$ of
$(0,1)$ such that for each $\al\in P$,
$\lim_{n\to\infty} H^*(\overline{\{u>\al\}},    [u_n]_\al) = 0$
and
$\lim_{n\to\infty} H^*([u_n]_\al, [u]_{\al}) = 0$;
\\
(\romannumeral4) \
For each $\al\in (0,1)\setminus P_0(u)$,
$\lim_{n\to\infty} H( [u]_{\al},    [u_n]_\al) = 0$;
\\
(\romannumeral5) \
There is a dense subset $P$ of
$(0,1)\setminus P_0(u)$ such that
for each $\al\in P$,
$\lim_{n\to\infty} H( [u]_{\al},    [u_n]_\al) = 0$;
\\
(\romannumeral6) \
There is a countable dense subset $P$ of
$(0,1)\setminus P_0(u)$ such that
for each $\al\in P$,
$\lim_{n\to\infty} H( [u]_{\al},    [u_n]_\al) = 0$;
\\
(\romannumeral7) \
 $H([u_n]_\al, [u]_\al) \rightarrow 0$ holds a.e. on $\al\in (0,1)$.

\end{tm}

\begin{proof}
 By Theorem \ref{uscg},
(\romannumeral1)$\Leftrightarrow$(\romannumeral2)$\Leftrightarrow$(\romannumeral3).
Since for each $\al\in (0,1)\setminus P_0(u) $,
$\overline{\{u>\al\}}  = [u]_\al$,
then (\romannumeral2)$\Rightarrow$(\romannumeral4).
Clearly
 (\romannumeral4)$\Rightarrow$(\romannumeral5)$\Leftrightarrow$(\romannumeral6).
Since $\overline{\{u>\al\}}\subseteq [u]_\al$,
(\romannumeral6)$\Rightarrow$(\romannumeral3).
By Lemma \ref{gnc}, $P_0 (u)$ is at most countable,
 and therefore (\romannumeral4)$\Rightarrow$(\romannumeral7).
Since $\overline{\{u>\al\}}\subseteq [u]_\al$,
(\romannumeral7)$\Rightarrow$(\romannumeral3).
So the proof is completed.

The desired conclusion can also be deduced from Theorems \ref{uscg} and \ref{aec}
as follows.
By Theorem \ref{uscg}, (\romannumeral1)$\Leftrightarrow$(\romannumeral2)$\Leftrightarrow$(\romannumeral3).
From
Theorem \ref{aec}, (\romannumeral1)$\Leftrightarrow$(\romannumeral4)$\Leftrightarrow$(\romannumeral5)$\Leftrightarrow$(\romannumeral6)$\Leftrightarrow$(\romannumeral7).
So
 Theorem \ref{uscgre} is proved.
\end{proof}

\begin{tm} \label{aedr}
Let $u$, $u_n$, $n=1,2,\ldots$, be fuzzy sets in $F^1_{USC} (X)$.
If there is a dense subset $P$ of $(0,1)$ such that
  $H([u_n]_\al, [u]_\al) \to   0$ for each $\al \in P$,
  then
  $H_{\rm end}(u_n, u) \to 0$.
\end{tm}

\begin{proof} We proceed by contradiction. If $H_{\rm end}(u_n, u) \not\to 0$,
then
 $H^*({\rm end}\, u_{n},   {\rm end}\, u)  \not\to 0$ or
 $H^*({\rm end}\, u,   {\rm end}\, u_{n})  \not\to 0$.

Suppose that $H^*({\rm end}\, u_{n},   {\rm end}\, u)  \not\to 0$.
Then there is an $\varepsilon>0$ and a subsequence $\{ u_{n_k} \}$
of
$\{u_n\}$
such that
$H^*({\rm end}\, u_{n_k},   {\rm end}\, u) > \varepsilon$
for each $k\in \mathbb{N}$.
So for each $k\in \mathbb{N}$,
 there
exists
an $(x_{n_k}, \al_{n_k}) \in {\rm end}\, u_{n_k}$
such that
\begin{equation}\label{contpe}
\overline{d}(    (x_{n_k}, \al_{n_k}) ,     {\rm end}\, u) > \varepsilon.
  \end{equation}
  With no loss of generality we can assume that $\{\al_{n_k}\}$ converges to
  an element $\al$ in $[0,1]$. By \eqref{contpe}, $\al_{n_k} > \varepsilon$ for each $k\in \mathbb{N}$, and so $\al \geq   \varepsilon$.
Pick $\beta\in P$ satisfying $\alpha \in (\beta, \beta + \varepsilon/2)$.
Then there exists $K$ such that
$   \al_{n_k}  \in  (\beta,   \beta + \varepsilon/2) $    for all $k \geq  K$.
Thus
for each $k \geq K$,
\begin{align}
\overline{d}  &  (    (x_{n_k}, \al_{n_k}) ,     {\rm end}\, u)     \nonumber
\\
& \leq   \overline{d}    (    (x_{n_k}, \beta) ,     {\rm end}\, u)   +
\overline{d}  (    (x_{n_k}, \al_{n_k}),  (x_{n_k}, \beta))
  <  \overline{d}    (    (x_{n_k}, \beta) ,     [u]_\beta \times \{\beta\})   +   \varepsilon/2\nonumber
\\
& = \overline{d}    (    x_{n_k} ,     [u]_\beta)   +   \varepsilon/2
\leq H([u_{n_k}]_\beta,  [u]_\beta)   +   \varepsilon/2.   \label{xne}
\end{align}
Note that $ H([u_{n_k}]_\beta,  [u]_\beta)  \to 0$, thus \eqref{contpe} contradicts \eqref{xne}.
So
the supposition is false.

For $H^*({\rm end}\, u,   {\rm end}\, u_n)  \not\to 0$,
we can similarly derive a contradiction.
\end{proof}

\begin{re}{\rm
It can be seen
that
Theorem \ref{aedr} can also be deduced from Lemma \ref{emr} (\romannumeral3).
 Fan (Lemma 1 in \cite{fan3}) proved a result of Theorem \ref{aedr} type.

}
\end{re}

\begin{tm} \label{hepc}
   Let $u$, $u_n$, $n=1,2,\ldots$, be fuzzy sets in $F^1_{USC} (X)$. If
  $H_{\rm end}(u_n, u) \to 0$,
  then
    $H([u_n]_\al, [u]_\al)  \to  0$ for each $\al \in (0,1) \setminus   P_0 (u)$.
\end{tm}

\begin{proof}
 Let $\al \in (0,1) \setminus  P_0(u)$.
Given $\varepsilon>0$.        Then there exists a $\delta (\al, \varepsilon)  \in (0, \varepsilon/2) $
such that
$[\al-\delta, \al+\delta] \subset [0,1]$
and
$
 H([u]_\beta,  [u]_\al) <   \varepsilon / 2
$ for all $\beta \in [\al-\delta, \al+\delta]$.

From    $H_{\rm end}(u_n, u) \to 0$,
there
exists an $N(\delta)$ such that
\begin{equation}\label{unds}
H_{\rm end} (u_n, u) < \delta
\end{equation} for all $n\geq N$.
Thus by Proposition \ref{aen},
$H^*([u_n]_\al, [u]_{\al-\delta}) <  \delta  <   \varepsilon/2.$
So,
 for each $n\geq N$,
\begin{align}\label{len}
 & H^*([u_n]_\al, [u]_{\al})  \nonumber  \\
& \leq  H^*([u_n]_\al, [u]_{\al-\delta})   +   H([u]_\al, [u]_{\al-\delta})  \nonumber    \\
& <  \varepsilon/2 + \varepsilon/2 = \varepsilon.
  \end{align}

Similarly, it follows from \eqref{unds} and Proposition \ref{aen} that
$H^*( [u]_{\al+\delta},  [u_n]_\al) < \delta < \varepsilon/2,$
and
then, for each $n\geq N$,
\begin{align}\label{rehn}
 & H^*([u]_\al, [u_n]_{\al})  \nonumber  \\
& \leq   H([u]_\al, [u]_{\al+\delta})   +    H^*( [u]_{\al+\delta},    [u_n]_\al )     \nonumber    \\
& <  \varepsilon/2 + \varepsilon/2=\varepsilon.
  \end{align}

Combined with \eqref{len} and \eqref{rehn}, we obtain that $H([u]_\al, [u_n]_{\al})<\varepsilon$ for each $n\geq N$. Since $\varepsilon>0$ is arbitrary, it follows that
$H([u]_\al,  [u_n]_\al) \to 0.$
\end{proof}

\begin{re}{\rm
Note that
for each $\al\in (0,1) \setminus P_0(u)$,
$\lim_{\lambda \to \alpha}H( [u]_\al, [u]_{\lambda}) = 0$
and
$[u]_\al = \overline{\{u>\al \}}$.
Thus
Theorem \ref{hepc} can also be deduced from Lemma \ref{acm}.

}
\end{re}

The following theorem gives some conditions under
which
 the $H_{\rm end}$ convergence of fuzzy sets
can be decomposed to
the Hausdorff metric convergence of certain $\al$-cuts.

\begin{tm} \label{aec}
Let $u$ be a fuzzy set in $F^1_{USCG} (X)$ and let $u_n$, $n=1,2,\ldots$, be fuzzy sets in $F^1_{USC} (X)$.
Then
the following are equivalent:
\\
(\romannumeral1)
  $H_{\rm end}(u_n, u) \to 0$;
\\
(\romannumeral2)
 $H([u_n]_\al, [u]_\al) \rightarrow 0$ holds a.e. on $\al\in (0,1)$;
\\
(\romannumeral3)    $H ([u_n]_\al, [u]_\al) \to 0 $ for all $\al\in (0,1) \setminus P_0(u)$;
\\
(\romannumeral4)
   There is a dense subset $P$ of $(0,1) \backslash P_0(u)$ such that $H ([u_n]_\al, [u]_\al) \to 0 $ for
$\al\in P$;
\\
(\romannumeral5)     There is a countable dense subset $P$ of $(0,1) \backslash P_0(u)$
such that
 $H ([u_n]_\al, [u]_\al) \to 0 $ for
$\al\in P$.
\end{tm}

\begin{proof}
Suppose that $u$, $u_n$, $n=1,2,\ldots$, are fuzzy sets
  in $F^1_{USC} ( X )$. We claim that
\\
(a) (\romannumeral2)$\Rightarrow$(\romannumeral1)$\Rightarrow$(\romannumeral3)$\Rightarrow$(\romannumeral4)$\Leftrightarrow$(\romannumeral5).
\\
(b) If $(0,1) \backslash P_0(u)$ is dense in $(0,1)$, then
(\romannumeral1)$\Leftrightarrow$(\romannumeral3)$\Leftrightarrow$(\romannumeral4)$\Leftrightarrow$(\romannumeral5).
\\
(c) If $P_0(u)$ is a set of measure zero, then
(\romannumeral1)$\Leftrightarrow$(\romannumeral2)$\Leftrightarrow$(\romannumeral3)$\Leftrightarrow$(\romannumeral4)$\Leftrightarrow$(\romannumeral5).

(\romannumeral2)$\Rightarrow$(\romannumeral1) follows from Theorem \ref{aedr}. (\romannumeral1)$\Rightarrow$(\romannumeral3) is Theorem \ref{hepc}.
Clearly
 (\romannumeral3)$\Rightarrow$(\romannumeral4)$\Leftrightarrow$(\romannumeral5).
So (a) is true.
If $(0,1) \backslash P_0(u)$ is dense in $(0,1)$, then a
dense subset
 $P$ of $(0,1) \backslash P_0(u)$ is also a dense subset of $(0,1)$.
 So
by
Theorem \ref{aedr},
 (\romannumeral4)$\Rightarrow$(\romannumeral1). Combined with (a), we have that (b) is true.

 Assume that $P_0(u)$ is a set of measure zero.
Then (\romannumeral3)$\Rightarrow$(\romannumeral2).
At this time, $(0,1) \backslash P_0(u)$ is dense in $(0,1)$.
So from (a) and (b), we obtain that (c) is true.

By Lemma \ref{gnc}, for each $u \in F^1_{USCG} (X)$,
 $P_0(u)$ is countable, and then
$P_0(u)$ is a set of measure zero.
So by (c), we obtain the desired result.
\end{proof}

\begin{re}{\rm
  Clearly, Theorem \ref{uscgre} implies Theorem \ref{aec}.
}
\end{re}

\section{Relations among metrics on $F^1_{USC}(X)$} \label{rem}

In this section, we discuss the relation
among the $H_{\rm end}$ metric,
the $H_{\rm send}$ metric, the $d_\infty$ metric, and the $d_p^*$ metric on $F^1_{USC}(X)$.

For $u,v\in F^1_{USC}(X)$, the $d_p$ distance given by
$$d_p(u,v)= \left(\int_0^1 H([u]_\al, [v]_\al)^p  \,   d\al   \right)^{1/p}$$
 is well-defined if and only if
$H([u]_\al, [v]_\al)$ is a measurable function of $\al$ on $[0, 1]$ (here we see $H([u]_\al, [v]_\al)$ as a function of $\al$ from $[0, 1]$ to $\mathbb{R} \cup \{+\infty\}$).
In the sequel, we suppose that the $d_p$ distance satisfying
 $p \geq 1$.

Since $H([u]_\al, [v]_\al)$ could be a non-measurable function
of $\al$ on
$[0,1]$ (see Example 2.13 in \cite{huang17c}),
we introduce the
 $d_p^*$ distance on $F^1_{USC} (X)$, $p\geq 1$, in
\cite{huang17},
which is defined by
\begin{gather*}
\begin{split}
d_p^*(u,v) := \inf \{  \  \left(\int_0^1 f(\al)^p  \,   d\al   \right)^{1/p}
 :
& \ f \mbox{ is a measurable function from } [0,1] \mbox{ to } \mathbb{R} \cup \{+\infty\}    ;
\\
&   \  f(\al)  \geq H([u]_\al, [v]_\al) \mbox{ for } \al\in [0,1]  \  \}
\end{split}
\end{gather*}
for
$u,v \in F^1_{USC} (X)$.

For each $u,v \in F^1_{USC} (X)$,
if $d_p(u,v)$ is well-defined then clearly 
$d_p^*(u,v) = d_p(u,v)$.
In \cite{huang17, huang17c} and references therein, we have given several conclusions on the well-definedness
of the $d_p$ distance including: the $d_p$ distance is well-defined on $F^1_{USC}(\mathbb{R}^m)$, and the $d_p$ distance is well-defined on $F^1_{USCG}(X)$.

Let $u,v \in  F^1_{USCG}(X)$.
By Proposition \ref{ecn}(\rmn1), the function $H([u]_\al, [v]_\al) $ of $\al$ is
 left-continuous on $(0,1]$. So from Proposition \ref{lem} below, $H([u]_\al, [v]_\al) $ is a measurable function of $\al$ on $[0,1]$.
 Proposition \ref{lem} may be known. Here we give a proof for the 
completeness of this paper. A conclusion stronger than Proposition \ref{lem} was given in
 \cite{huang17c} and the references therein.

\begin{pp}\label{lem}
   If a function $f: [0,1] \to \mathbb{R}\cup \{+\infty\}$
 is left-continuous on $(0,1]$, then $f$ is a measurable function on $[0,1]$.
\end{pp}

\begin{proof}
Let $r\in \mathbb{R}$. Denote the set $\{x\in [0,1]: f(x)>r \}$ by $\{f>r\}$.
If $\{f>r\}\setminus \{0\} = \emptyset$, then
$\{f>r\}\setminus \{0\}$ is a measurable set.
Suppose that $\{f>r\}\setminus \{0\}  \not= \emptyset$.
 For each $x \in \{f>r\}\setminus \{0\}   $,
 let $\overbrace{x} = \bigcup\{ [a,b] : x \in [a,b] \subseteq \{f>r\}\setminus \{0\}   \}$. Clearly $\overbrace{x}$ is an interval. Note that
 for each interval $A$ with $x\in A \subseteq \{f>r\}\setminus \{0\}$, $A=\cup\{[a,b]: x\in [a,b]\subseteq A\}\subseteq \overbrace{x}$.
So $\overbrace{x}$ is the largest interval in $\{f>r\}\setminus \{0\}$ which contains $x$.
For each $x\in X$,
  if $\al\in \overbrace{x}$,
 then there exists a $\delta (\al) > 0$
with $[\al-\delta(\al), \al]  \subseteq \{f>r\}\setminus \{0\}$ since $f$ is left-continuous at $\al$.
So
 $\overbrace{x}$ is a positive length interval.
For each
 $x,y \in \{f>r\}\setminus \{0\} $,
 if $\overbrace{x} \cap \overbrace{y} \not= \emptyset$,
 then $\overbrace{x} \cup \overbrace{y} $ is an interval
 with $\{x,y\}\subset \overbrace{x} \cup \overbrace{y} \subseteq \{f>r\}\setminus \{0\}$,
 and hence
 $\overbrace{x} = \overbrace{x} \cup \overbrace{y} = \overbrace{y}$.
  Thus
  $\{f > r\}\setminus \{ 0 \}   $ is a union of disjoint positive length intervals.
Clearly the set $S$ of all
these disjoint positive length intervals is countable (For each $B\in S$,
choose a $q_B$ in $B\cap \mathbb{Q}$. Then $q_{B_1} \not= q_{B_2}$ for each $B_1$ and $B_2$ in $S$ with $B_1\not= B_2$. Thus $\overline{\overline{S}}= \overline{\overline{\{q_B: B\in S\}}}\leq \overline{\overline{\mathbb{Q}}}$ and so $S$ is countable, where $\overline{\overline{C}}$ denotes the cardinality of a set $C$.).
So $\{f>r\} \setminus \{0\}$ is a countable union of intervals and hence is a measurable set.
Thus
$\{f>r\}$ is a measurable set as $\{f>r\}=\{f>r\} \setminus \{0\}$ or $\{f>r\}=(\{f>r\} \setminus \{0\})\cup \{0\}$. Since $r\in \mathbb{R}$
is arbitrary, it follows that
$f$ is a measurable function on $[0,1]$.
\end{proof}

$d_p^*$ is an extended metric but may not be a metric on $F^1_{USC}(X)$.
 See also
Remark 3.3 in \cite{huang17}.
We can see
that
the $d_p$ distance is a metric on $F^1_{USCB}(X)$.
The $d_p$ distance on $F^1_{USCG}(\mathbb{R}^m)$ is an extended metric but not a metric,
and the $d_p$ distance on $F^1_{USC}(\mathbb{R}^m)$ is an extended metric but not a metric. The $d_p$ distance on $F^1_{USCG}(\mathbb{R}^m)$ could take
the value $+\infty$.
Let $u\in F^1_{USCG}(\mathbb{R}^m)$ be defined by putting
\[
[u]_\al=
\{x\in \mathbb{R}^m: \|x\|\leq n\} \mbox{ for each } n\in \mathbb{N} \mbox{ and } \al\in (1/(n+1), 1/n].
\]
Denote the origin of $\mathbb{R}^m$ by $o$.
Then $d_p(u, \widehat{o}_{F(\mathbb{R}^m)}) = (\sum_{n=1}^{+\infty} n^p\cdot (1/n - 1/(n+1)) )^{1/p}= + \infty$.

For simplicity,
in this paper, we call the $d_p^*$ distance on $F^1_{USC}(X)$ the $d_p^*$ metric,
and call the $d_p$ distance on $F^1_{USC}(\mathbb{R}^m)$ or $F^1_{USCG}(X)$ the $d_p$ metric.

Clearly for $u,v\in F^1_{USC}(X)$,
\begin{equation}\label{spr}
  d_\infty(u,v) \geq d_p^*(u,v).
\end{equation}
The proof of \eqref{spr} is routine. Set $d_\infty(u,v) =\xi \in \mathbb{R} \cup \{+\infty\}$.
Define $f: [0,1] \to \mathbb{R} \cup \{+\infty\} $ by $f(\al)= \xi$ for each $\al\in [0,1]$.
Hence $f$ is a measurable function from $[0,1]$ to $\mathbb{R} \cup \{+\infty\} $ and
$f(\al)  \geq H([u]_\al, [v]_\al)$ for $ \al\in [0,1] $.
So
$d_p^*(u,v) \leq \left(\int_0^1 f(\al)^p  \,   d\al   \right)^{1/p} = \xi $.
Thus \eqref{spr} is true.

\begin{tm} \label{gdpn}
  Let $u\in F^1_{USC} (X)$ and for each positive integer $n$, let $u_n\in F^1_{USC} (X)$.
  If $d_p^*(u_n, u) \to 0$, then
  $H_{\rm end} (u_n, u) \to 0$.
\end{tm}

\begin{proof} We prove by contradiction.
    If
  $H_{\rm end} (u_n, u) \not\to 0$, then there is an $\varepsilon>0$ and a subsequence $\{   v_n   \}$
of $\{u_n\}$
such that for each $n\in \mathbb{N}$,
\begin{equation}\label{gec}
 H_{\rm end} (v_n,  u ) \geq \varepsilon.
\end{equation}

By \eqref{gec},
for each $n\in \mathbb{N}$, $v_n\not=u$ and hence $d^*_p(v_n,  u)\not=0$.
From the definition of $d_p^*$, there exists a sequence $\{f_n\}$ of
measurable function
from $[0,1]$ to $\mathbb{R} \cup \{+\infty\}$
such that for each $n\in \mathbb{N}$,
\begin{gather}
 H([v_n]_\al,  [u]_\al)     \leq    f_n (\al)     \mbox{ for all }    \al\in [0,1],\label{smn}
 \\
 \left(\int_0^1   f_n (\al)^p   d\, \al     \right)^{1/p} \leq \frac{n+1}{n} d_p^*(v_n,  u). \label{tce}
\end{gather}
Since $d_p^*(v_n,  u) \to 0$, by \eqref{tce}, we have
$\left(\int_0^1   f_n (\al)^p   d\, \al     \right)^{1/p} \to 0$.
Thus there is a subsequence $\{f_{n_k}\}$ of $\{f_n\}$
such that
$\{f_{n_k}(\al) \}$ converges to $0$ a.e. on $\al\in [0,1]$.
Hence by \eqref{smn},
$H([v_{n_k}]_\al, [u]_\al) \rightarrow 0$ holds a.e. on $\al\in (0,1)$.
By
Theorem \ref{aedr}, this
  implies
   $H_{\rm end} (v_{n_k}, u) \to 0$, which
contradicts \eqref{gec}.
\end{proof}

 Theorem 4.1 in \cite{huang17c} says that
for $u\in F^1_{USCG} (X)$ and $v\in F^1_{USC} (X)$,
$H([u]_\al, [v]_\al)$ is a measurable function of $\al$ on $[0, 1]$.
So $d_p^*(u,v)
=
d_p(u,v)$ for $u\in F^1_{USCG} (X)$ and $v\in F^1_{USC} (X)$.

\begin{tm} \label{ecm} Suppose that $u\in  F^1_{USCG}(X)$ and $u_n \in F^1_{USC} (X)$, $n=1,2,\ldots$,
and
that there is a measurable function $F$ on $[0,1]$
such that $\int_0^1 F^p(\al) \,d\al < +\infty$
and
$H([u_n]_\al,  [u]_\al) \leq  F(\al)  $ for $n=1,2,\ldots$.
If
 $H_{\rm end} (u_n, u) \to 0$, then $d_p(u_n, u) \to 0$.
  \end{tm}

\begin{proof} \ By Theorem \ref{aec}, $H_{\rm end} (u_n, u) \to 0$
if and only if
 $H([u_n]_\al, [u]_\al) \rightarrow 0$ holds a.e. on $\al\in (0,1)$.
So the desired result follows from the Lebesgue's Dominated Convergence
Theorem.
\end{proof}

\begin{pp} \label{spu} Let
  $u\in F^1_{USCB} (X)$ and for each positive integer $n$, let $u_n\in F^1_{USC} (X)$.
Then
$H_{\rm send} (u_n, u) \to 0$
if and only if
$H([u_n]_0, [u]_0) \to 0$ and $d_p(u_n, u) \to 0$.
\end{pp}

\begin{proof}From
Proposition \ref{sge}, $H_{\rm send} (u_n, u) \to 0$
if and only if
$H([u_n]_0, [u]_0) \to 0$ and $H_{\rm end} (u_n, u) \to 0$.
To prove the desired result, we only need to show
that
$$H([u_n]_0, [u]_0) \to 0 \mbox{ and } d_p(u_n, u) \to 0
\Leftrightarrow
H([u_n]_0, [u]_0) \to 0 \mbox{ and } H_{\rm end} (u_n, u) \to 0.$$

From Theorem \ref{gdpn}, ``$\Rightarrow$'' is true.
To show ``$\Leftarrow$'',
suppose that
 $H([u_n]_0, [u]_0) \to 0$ and $H_{\rm end} (u_n, u) \to 0$.
Then
there exists an $N \in \mathbb{N}$ such that $\bigcup_{n>N} [u_n]_0$ is bounded.
Hence there is an $M>0$ such that for $n \geq N$ and $\al\in [0,1]$,
 $H([u_n]_\al, [u]_\al)\leq d_\infty(u_n,u) < M$.
Thus
by Theorem \ref{ecm}, $d_p(u_n, u) \to 0$. So``$\Leftarrow$'' is true.
\end{proof}

However, for a sequence $\{u_n\}$ and an element $u$
in $F_{USCG}(X)$,
$H_{\rm send}(u_n, u) \to 0 $ does not necessarily imply $d_p(u_n, u) \to 0$.
See the following Example \ref{snp}.

\begin{eap} \label{snp}
  {\rm
 Define $u\in F^1_{USCG}(\mathbb{R})$ by putting
$$[u]_\al= [0,1/\al] \mbox{ for } \al\in(0,1].$$
For
each $n=1,2,\ldots$, define $u_n \in  F^1_{USCG}(\mathbb{R})$ by putting
\[[u_n]_\al
=\left\{
   \begin{array}{ll}
     [0,1/\al], & \al\in (1/n,1], \\
    \mbox{} [0, \frac{1}{\al} + n^2 ], & \al\in (0, 1/n].
   \end{array}
 \right.
\]
Then $H_{\rm send} (u_n,u) \to 0$ and $d_p(u_n,u)=n^{2-1/p} \not\to 0 $.

}
\end{eap}

In \cite{huang17}, we obtain that the Skorokhod metric convergence imply the sendograph metric convergence
on $F^1_{USC} (X)$
 (see Theorem 8.1 in \cite{huang17}),
and that
Skorokhod metric convergence need not imply the $d_p$ convergence on a subset of $F^1_{USCG}(X)$ (see the end of Section 5 in \cite{huang17}). From the above conclusions in \cite{huang17},
we can also deduce that
the sendograph metric convergence need not imply the $d_p$ convergence on $F^1_{USCG}(X)$.

\section{Characterizations of compactness in $(F^1_{USCG} (X), H_{\rm end})$ and $(F^1_{USCB} (X), H_{\rm send})$}
\label{cmfuzzy}

Based on the conclusions in previous sections, we give characterizations of
total boundedness, relative compactness and compactness
in
 $(F^1_{USCG} (X), H_{\rm end})$ and $(F^1_{USCB} (X), H_{\rm send})$.

\begin{itemize}
 \item A subset $Y$ of a topological space $Z$ is said to be \emph{compact} if for every set $I$
and every family of open sets, $O_i$, $i\in I$, such that $Y\subset \bigcup_{i\in I} O_i$ there exists
a finite family $O_{i_1}$, $O_{i_2}$ \ldots, $O_{i_n}$ such that
$Y\subseteq O_{i_1}\cup O_{i_2}\cup\ldots \cup O_{i_n}$.
In
the case of a metric topology, the criterion for compactness becomes that any sequence in $Y$ has a subsequence convergent in $Y$.

 \item
A \emph{relatively compact} subset $Y$ of a topological space $Z$ is a subset with compact closure. In the case of a metric topology, the criterion for relative compactness becomes that any sequence in $Y$ has a subsequence convergent in $X$.

 \item Let $(X, d)$ be a metric space. A set $U$ in $X$ is \emph{totally bounded} if and only if, for each $\varepsilon>0$, it contains a finite $\varepsilon$ approximation, where an $\varepsilon$ approximation to $U$ is a subset $S$ of $U$ such that $d(x,S)<\varepsilon$ for each $x\in U$.
     It is known that $U$ in $X$ is totally bounded if and only if, for each $\varepsilon>0$, there is a finite \emph{weak}
$\varepsilon$\emph{-net} of $U$, where
a \emph{weak}
$\varepsilon$\emph{-net} of $U$ is a subset $S$ of $X$
satisfying $d(x,S)<\varepsilon$ for each $x\in U$.

\end{itemize}

Let $(X, d)$ be a metric space. A set $U$ is compact in $(X,d)$ implies that $U$ is relatively compact
in $(X,d)$, which in turn
implies that $U$ is totally bounded in $(X,d)$.

We use $(\widetilde{X}, \widetilde{d})$ to denote the completion of $(X, d)$.
We see $(X, d)$ as a subspace of $(\widetilde{X}, \widetilde{d})$.
Let
$S \subseteq \widetilde{X}$.
The symbol $\widetilde{\overline{S}}$ is used to denote
the closure of $S$
in
 $(\widetilde{X}, \widetilde{d})$.

As defined in Section \ref{bas}, we have
$K(\widetilde{X})$,
 $C(\widetilde{X})$, $F^1_{USC}(\widetilde{X})$,
$F^1_{USCG} (\widetilde{X})$, etc. according to $(\widetilde{X}, \widetilde{d})$.
If there is no confusion,
 we also
use $H$ to denote the Hausdorff metric
on
 $C(\widetilde{X})$ induced by $\widetilde{d}$.
  We also
use $H$ to denote the Hausdorff metric
 on $C(\widetilde{X}\times [0,1])$ induced by $\overline{\widetilde{d}}$.
 We
 also use $H_{\rm end}$ to denote the endograph metric on $F^1_{USC}(\widetilde{X})$
defined by using $H$ on
$C(\widetilde{X} \times [0,1])$.

Let $U\subseteq X$. If $U$ is compact in $(X,d)$, then $U$ is compact in $(\widetilde{X}, \widetilde{d})$.
We see $(K(X), H)$ as a subspace of $(K(\widetilde{X}), H)$.

Define $j$ from $F(X)$ to $F(\widetilde{X})$ as follows:
for $u\in F(X)$, $j(u)\in F(\widetilde{X})$ is given by
\[
j(u) (t)
=
\left\{
  \begin{array}{ll}
    u(t), & t\in X,\\
0, & t\in \widetilde{X} \setminus X.
  \end{array}
\right.
\]

Let $u\in F^1_{USCG}(X)$. Then $j(u)\in F^1_{USCG} (\widetilde{X})$
because $[j(u)]_\al = [u]_\al \subseteq K(\widetilde{X})$ for each $\al\in (0,1]$.
Clearly for each $u,v\in F^1_{USCG}(X)$, $H_{\rm end}(u,v) = H_{\rm end}(j(u), j(v))$.
In the sequel, we treat
$(F^1_{USCG}(X), H_{\rm end})$ as a subspace of $(F^1_{USCG}(\widetilde{X}), H_{\rm end})$ by
identifying $u$ in $F^1_{USCG}(X)$ with $j(u)$ in $F^1_{USCG}(\widetilde{X})$.

\subsection{Characterizations of compactness in $(K(X), H)$} \label{cmpu}

In this subsection, we give characterizations of total boundedness, relative compactness and compactness in $(K(X), H)$.
The
results in this subsection are basis for contents in the sequel.

\begin{tm} \label{cpm} Let $(X, d)$ be a complete metric space and let $\{C_n\}$ be a Cauchy sequence in $(K(X), H)$.
Put
 $D= \overline{\bigcup_{n=1}^{+\infty} C_n}$. For $n=1,2,\ldots$, put $D_n= \bigcup_{l=1}^{n} C_l$.
 Then
 $D \in K(X)$
and
  $H(D_n, D) \to 0$ as $n\to\infty$.
\end{tm}

\begin{proof}
 Let $k,j\in \mathbb{N}$ with $k>j$. Then $H^*(D_j, D_k)=0$, and for each $i\in \mathbb{N}$ with $1\leq i \leq j$,
$H^*(C_i, D_j)=0$.
 Thus
 $H(D_k, D_j) = H^*(D_k, D_j)
 =\max\{H^*(C_i, D_j) : i=1,\ldots,k\} =\max\{H^*(C_i, D_j) : i=j+1,\ldots,k \}\leq \max\{H^*(C_i, C_j) : i=j+1,\ldots,k \}$.
 Hence
  $$H(D_k, D_j)\leq \max\{H(C_i, C_j) : i=j+1,\ldots,k \}.$$
 So
  $\{D_n\}$ is a Cauchy sequence in $(K(X), H)$.
   From Theorem \ref{bfc}, $(K(X), H)$ is complete,
  and thus there is an
$E\in K(X)$ such that $H(D_n, E) \to 0$ as $n\to\infty$.
By Theorem \ref{hkg}, $E=\lim^{(K)}_{n\to \infty} D_n = D$.
  \end{proof}

\begin{tm} \label{tbe}
  Let $(X,d)$ be a metric space and $\mathcal{D}\subseteq K(X)$.
   Then $\mathcal{D}$ is totally bounded in $(K(X), H)$
if and only if
$\mathbf{ D} =   \bigcup \{C:  C \in  \mathcal{D} \} $ is totally bounded in $(X,d)$.
\end{tm}

\begin{proof} If $\mathcal{D} = \emptyset$, then the desired result follows immediately.
Suppose that
$\mathcal{D} \not= \emptyset$.

\textbf{\emph{Necessity}}. \
 To prove that $\mathbf{ D}$ is totally bounded it suffices to show that
each sequence in $\mathbf{ D}$ has a Cauchy subsequence.

 Given a sequence $\{x_n\}$
in $\mathbf{D}$.
 Suppose that
$x_n \in C_n \in \mathcal{D}$ for $n=1,2,\ldots$.
Since $\mathcal{D}$ is totally bounded in $(K(X), H)$, then $\{C_n\}$ has a Cauchy subsequence $\{C_{n_k}\}$ in $(K(X), H)$. As $\{C_{n_k}\}$ is also a Cauchy sequence in $(K(\widetilde{X}), H)$,
 by Theorem \ref{cpm},
$\widetilde{\overline{\bigcup_{k=1}^{+\infty} C_{n_k} }}$ is in $K(\widetilde{X})$.
Thus
$\{x_{n_k}\} $
 has a Cauchy subsequence, and hence so does $\{x_n\}$.

\textbf{\emph{Sufficiency}}. \   If   $\mathbf{ D} $ is totally bounded in $X$, then $\widetilde{\overline{\mathbf{D}}}$ is in $K(\widetilde{X})$.
So, by Theorem \ref{bfc}, $(K(\widetilde{\overline{\mathbf{D}}}), H)$ is compact, and
thus
 $\mathcal{D} \subseteq K(\widetilde{\overline{\mathbf{D}}})$ is totally bounded in $(K(\widetilde{\overline{\mathbf{D}}}), H)$.
So $\mathcal{D}$ is obviously totally bounded in $(K(X), H)$.
\end{proof}

\begin{tm}\cite{greco} \label{rce}
  Let $(X,d)$ be a metric space and $\mathcal{D}\subseteq K(X)$. Then $\mathcal{D} $ is relatively compact in $(K(X), H)$
if and only if
$\mathbf{ D} =   \bigcup \{C:  C \in  \mathcal{D} \} $ is relatively compact in $(X,d)$.
\end{tm}

\begin{proof} \ If $\mathcal{D} = \emptyset$, then the desired result follows immediately.
Suppose that
$\mathcal{D} \not= \emptyset$.

\textbf{\emph{Necessity}}. \   To prove that $\mathbf{ D}$ is relatively compact it suffices to show that
each sequence in $\mathbf{ D}$ has a convergent subsequence in $(X,d)$.

Given a sequence $\{x_n\}$
in   $\mathbf{D}$.
 Suppose that
$x_n \in C_n \in \mathcal{D}$ for $n=1,2,\ldots$.
Since $\mathcal{D}$ is relatively compact in $(K(X), H)$, then $\{C_n\}$ has a subsequence $\{C_{n_k}\}$ which converges to an element $C$ in $(K(X), H)$.
Clearly $\{C_{n_k}\}$ converges to $C$ in $(K(\widetilde{X}), H)$.
Hence,
 by Theorem \ref{cpm},
$\widetilde{\overline{\bigcup_{k=1}^{+\infty} C_{n_k} }}$ is in $K(\widetilde{X})$
(Indeed, $\widetilde{\overline{\bigcup_{k=1}^{+\infty} C_{n_k} }}$ is in $K(X)$).
So $\{x_{n_k}\}$ has a subsequence which converges to $x$
in $\widetilde{\overline{\bigcup_{k=1}^{+\infty} C_{n_k} }}$,
and thus
$x\in C = \lim_{k\to \infty}^{(K)} C_{n_k} \subseteq X$.

\textbf{\emph{Sufficiency}}. \ If $\mathbf{ D} $ is relatively compact in $X$, then  $\overline{\mathbf{ D}} $ is in $K(X)$,
and
therefore $(K(\overline{\mathbf{ D}}), H)$ is compact.
Thus $\mathcal{D} \subseteq  K(\overline{\mathbf{ D}}) $ is relatively compact in $(K(\overline{\mathbf{ D}}), H)$. So clearly
$\mathcal{D}$ is relatively compact in $(K(X), H)$.
\end{proof}

\begin{lm} \label{coms}
   Let $(X,d)$ be a metric space and $\mathcal{D}\subseteq K(X)$. If $\mathcal{D} $ is compact in $(K(X), H)$,
then
$\mathbf{ D} =   \bigcup \{C:  C \in  \mathcal{D} \} $ is compact in $(X, d)$.
\end{lm}

\begin{proof} \ If $\mathcal{D} = \emptyset$, then the desired result follows immediately.
Suppose that
$\mathcal{D} \not= \emptyset$.
  To show that $\mathbf{ D}$ is compact, we only need to show that
each sequence in $\mathbf{ D}$ has a subsequence which converges to a point in $\mathbf{ D}$.

Given a sequence $\{x_n\}$
in   $\mathbf{D}$.
 Suppose that
$x_n \in C_n \in \mathcal{D}$ for $n=1,2,\ldots$.
Since $\mathcal{D}$ is compact, then $\{C_n\}$ has a subsequence $\{C_{n_k}\}$ converges to $C \in \mathcal{D}$.
Clearly $\{C_{n_k}\}$ converges to $C$ in $(K(\widetilde{X}), H)$.
Hence, by Theorem \ref{cpm},
$\widetilde{\overline{\bigcup_{k=1}^{+\infty} C_{n_k} }}$ is in $K(\widetilde{X})$
(Indeed, $\widetilde{\overline{\bigcup_{k=1}^{+\infty} C_{n_k} }}$ is in $K(\mathbf{D})$).
So $\{ x_{n_k} \}$ has a subsequence which converges to $x$
in $\widetilde{\overline{\bigcup_{k=1}^{+\infty} C_{n_k} }}$.
Thus $x \in   C = \lim_{k\to \infty}^{(K)} C_{n_k}   \subseteq \mathbf{D}$.
\end{proof}

\begin{re}
{\rm
The converse of the implication in Lemma \ref{coms} does not hold.
Let $(X,d) = \mathbb{R}$ and
$\mathcal{D} = \{ [0,x]:  x\in (0.3,1] \}   \subset   K(\mathbb{R})$.
Then $\mathbf{D} = [0,1] \in K(\mathbb{R})$.
But
$\mathcal{D}$ is not compact in $(  K(\mathbb{R}),    H   )$.
}
\end{re}

\begin{tm} \label{come}
 Let $(X,d)$ be a metric space and $\mathcal{D}\subseteq K(X)$. Then
    the following are equivalent:
    \\
    (\romannumeral1) \ $\mathcal{D} $ is compact in $(K(X), H)$;
        \\
 (\romannumeral2) \
$\mathbf{ D} =   \bigcup \{C:  C \in  \mathcal{D} \} $ is relatively compact in $(X, d)$
and
$\mathcal{D} $ is closed in $(K(X), H)$;
    \\
 (\romannumeral3) \
$\mathbf{ D} =   \bigcup \{C:  C \in  \mathcal{D} \} $ is compact in $(X, d)$
and
$\mathcal{D} $ is closed in $(K(X), H)$.
\end{tm}

\begin{proof}
Note that $\mathcal{D} $ is compact in $(K(X), H)$ if and only if $\mathcal{D} $ is relatively compact and closed in $(K(X), H)$.
 Then from Theorem \ref{rce} we have (\romannumeral1)$\Leftrightarrow$(\romannumeral2). Clearly (\romannumeral3)$\Rightarrow$(\romannumeral2). We
shall complete the proof by showing that (\romannumeral1)$\Rightarrow$(\romannumeral3),
which can be deduced by Lemma \ref{coms}.
\end{proof}

\begin{re}
{\rm
    Theorem \ref{rce} is Proposition 5 in \cite{greco}.
   We cannot find the proof of Proposition 5 in \cite{greco}.
Some of the results in this subsection may be known. We cannot find the results 
of this subsection other than Theorem \ref{rce} in the references that we can obtain. So we give our proofs here.
}
\end{re}

\subsection{Characterizations of compactness in $(F^1_{USCG} (X), H_{\rm end})$}

In this subsection, we give characterizations of total boundedness, relative compactness and compactness in $(F^1_{USCG} (X), H_{\rm end})$.

Suppose that
$U$ is a subset of $F^1_{USC} (X)$ and $\al\in [0,1]$.
For
writing convenience,
we denote
\begin{itemize}
  \item  $U(\al):= \bigcup_{u\in U} [u]_\al$, and

\item  $U_\al : =  \{[u]_\al: u \in U\}$.
\end{itemize}

\begin{tm} \label{tbfegn}
  Let $U$ be a subset of $F^1_{USCG} (X)$. Then $U$ is totally bounded in $(F^1_{USCG} (X), H_{\rm end})$
if and only if
$U(\al)$
is totally bounded in $(X,d)$ for each $\al \in (0,1]$.
\end{tm}

\begin{proof}
   \textbf{\emph{Necessity}}.  Suppose that $U$ is totally bounded in $(F^1_{USCG} (X), H_{\rm end})$.
Let $\al \in (0,1]$.
To show
that
$U(\al)$
is totally bounded in $X$, we only need to show that
each sequence in $U(\al)$
has a Cauchy subsequence.

Given a sequence $\{x_n\} \subset U(\al)$. Suppose that $x_n \in  [u_n]_\al$,  $u_n  \in  U$, $n=1,2,\ldots$.
Then
$\{u_n\}$ has
 a Cauchy subsequence $\{u_{n_l}\}$. So given $\varepsilon\in (0, \al)$, there is a $K( \varepsilon) \in \mathbb{N}$
such that
$
H_{\rm end} (u_{n_l},  u_{n_K})  <   \varepsilon
$
for all $l \geq K$.
Thus by Proposition \ref{aen},
\begin{equation}\label{cuts}
H^*([u_{n_l}]_\al,       [u_{n_K}]_{\al-\varepsilon}) <  \varepsilon
\end{equation}
for all $l \geq K$.
From \eqref{cuts} and the arbitrariness of $\varepsilon$,  $\bigcup_{l=1} ^{+\infty} [u_{n_l}]_\al $ is totally bounded in $(X,d)$.
Thus
$\{x_{n_l}\}$, which is a subsequence of $\{x_n\}$,
has a Cauchy subsequence, and so does $\{x_n\}$.

In the following, we give
a detailed proof for the above conclusion that
        $\bigcup_{l=1} ^{+\infty} [u_{n_l}]_\al $ is totally bounded in $(X,d)$.
To show that
        $\bigcup_{l=1} ^{+\infty} [u_{n_l}]_\al $ is totally bounded in $(X,d)$, it suffices to show that
for each $\lambda>0$, there exists a finite weak $\lambda$-net of  $ \bigcup_{l=1} ^{+\infty} [u_{n_l}]_\al$.

Let $\lambda>0$. Set $\varepsilon= \min\{\lambda/2, \al/2\} $.
Then $\varepsilon \in (0, \al)$. Hence there is a $K(\varepsilon)$ such that \eqref{cuts} holds for all $l\geq K$.
Since $\bigcup_{l=1} ^{K}[u_{n_l}]_{\al-\varepsilon} $ is compact,
there is a finite $\varepsilon$ approximation $\{z_j\}_{j=1}^{m}$ to $\bigcup_{l=1} ^{K}[u_{n_l}]_{\al-\varepsilon} $. We claim that $\{z_j\}_{j=1}^{m}$
is a finite weak $\lambda$-net of $\bigcup_{l=1} ^{+\infty} [u_{n_l}]_\al$.

Let $z \in \bigcup_{l=1} ^{+\infty} [u_{n_l}]_\al$.
If $z \in \bigcup_{l=1} ^{K} [u_{n_l}]_\al$, then clearly $d(z, \{z_j\}_{j=1}^{m}) < \varepsilon<\lambda$.
If $z \in \bigcup_{l=K+1} ^{+\infty} [u_{n_l}]_\al$,
then
by \eqref{cuts},
  there exists a $y_z \in [u_{n_K}]_{\al-\varepsilon}$
such that
$d(z, y_z) <\varepsilon $,
and hence
$d(z, \{z_j\}_{j=1}^{m}) \leq  d(z, y_z) + d(y_z, \{z_j\}_{j=1}^{m} ) < 2\varepsilon \leq \lambda$.
Thus
$d(z, \{z_j\}_{j=1}^{m}) < \lambda$ for each $z \in \bigcup_{l=1} ^{+\infty} [u_{n_l}]_\al$.
So
$\{z_j\}_{j=1}^{m}$
is a finite weak $\lambda$-net of $\bigcup_{l=1} ^{+\infty} [u_{n_l}]_\al$.

\textbf{\emph{Sufficiency}}.
Suppose that $U(\al)$
is totally bounded in $(X,d)$ for each $\al \in (0,1]$.
Then
 $U(\al)$
is relatively compact in $(\widetilde{X},d)$ for each $\al \in (0,1]$.
So from
the sufficiency part of Theorem \ref{rcfegn},
we obtain
that
 $U$ is relatively compact in $(F^1_{USCG} (\widetilde{X}), H_{\rm end})$.
 Then
 $U$
  is totally bounded in $(F^1_{USCG} (\widetilde{X}), H_{\rm end})$,
 and so $U$
 is totally bounded in $(F^1_{USCG} (X), H_{\rm end})$.
\end{proof}

\begin{tm} \label{rcfegn}
  Let $U$ be a subset of $F^1_{USCG} (X)$. Then $U$ is relatively compact in $(F^1_{USCG} (X), H_{\rm end})$
if and only if
$U(\al)$
is relatively compact in $(X, d)$ for each $\al \in (0,1]$.
\end{tm}

\begin{proof}
    \textbf{\emph{Necessity}}. Suppose that $U$ is relatively compact. Given $\al \in (0,1]$.
To show that
$U(\al)$
is relatively compact in $X$, we only need to show that each sequence
in $U(\al)$ has a convergent subsequence in $(X,d)$.

 Let $\{x_n\}$ be a sequence in $U(\al)$.
Suppose
that $x_n\in [u_n]_\al$, $u_n \in U$, $n=1,2,\ldots$.
Then
 there is a subsequence $\{ u_{n_k} \}$ of $\{u_n\}$ and $u \in F^1_{USCG} (X)$
  such that
$H_{\rm end} (u_{n_k}, u) \to 0$.
So, by Theorem \ref{aec}, $H([u_{n_k}]_\al, [u]_\al) \rightarrow 0$ holds a.e. on $\al\in (0,1)$,
and therefore there is a $\beta \in (0, \al)$ such that
$H([  u_{n_k}  ]_\beta, [u]_\beta) \to 0$. Hence by Theorem \ref{rce}, $\bigcup_{k=1}^{+\infty}[  u_{n_k}  ]_\beta$ is relatively compact in $X$.
Thus
$\{x_{n_k}\}$ has a convergent subsequence in $X$, and so does $\{x_n\}$.

\emph{\textbf{Sufficiency}}.
Suppose that $U(\al)$ is relatively compact in $X$ for each $\al\in (0,1]$.
Note that $U(\al) =   \bigcup \{D:  D \in  U_{\al}\} $ for each $\al\in [0,1]$.
So, by
 Theorem \ref{rce},
we obtain that
$U_{\al}$ is relatively compact in $(K(X), H)$ for each $\al\in (0,1]$.
This means that the following
affirmation (a) is true.
\begin{description}
  \item[(a)] \ Given a sequence $\{w_n : n=1,2,\ldots\}$ in $U$ and $\al\in (0,1]$.
Then the corresponding sequence
$\{[w_n]_\al: n=1,2,\ldots\}$
has a convergent subsequence in $(K(X), H)$.
\end{description}

To show that $U$ is relatively compact
in $(F^1_{USCG} (X), H_{\rm end})$,
we only need to show that each sequence
in $U$ has a convergent subsequence in $(F^1_{USCG} (X), H_{\rm end})$.
Suppose that
$\{u_n\}$ is a sequence in $U$. Based on the above affirmation (a) and Theorem \ref{aec},
  and
 proceeding similarly to the proof of the ``Sufficiency part'' of Theorem 7.1 in \cite{huang}, it can be shown
that
$\{u_n\}$
has a subsequence $\{v_n\}$ which converges to an element $v$ in $(F^1_{USCG} (X), H_{\rm end})$.

A sketch of the
proof of the existence of $\{v_n\}$ and $v$ is given as follows.

First, we construct a subsequence $\{v_n\}$ of $\{u_n\}$
such that
$[v_n]_q$ converges to an element $u_q$ in $(K(X), H)$ for all $q\in \mathbb{Q}'$, where $\mathbb{Q}' = \mathbb{Q} \cap (0,1]$.
For $\al \in (0,1]$,
 we define
$v_\al = \bigcap_{q<\al, q\in \mathbb{Q}'} u_q$.
Then we define
 a fuzzy set
 $v$ in $X$ by using $\{v_\al, \al\in (0,1]\}$. At last, we show
that
 $v\in F^1_{USCG}(X)$, $[v]_\al= v_\al $ for each $\al\in (0,1]$,
and
$H_{\rm end} (v_n, v) \to 0$.
\end{proof}

\begin{tm}\label{cfeg}
  Let $U$ be a subset of $F^1_{USCG} (X)$. Then the following are equivalent:
\\
(\romannumeral1)
 $U$ is compact in $(F^1_{USCG} (X), H_{\rm end})$;
\\
(\romannumeral2)
 $U(\al)$
is relatively compact in $(X, d)$ for each $\al \in (0,1]$ and $U$ is closed in $(F^1_{USCG} (X), H_{\rm end})$;
\\
(\romannumeral3) $U(\al)$
is compact in $(X, d)$ for each $\al \in (0,1]$ and $U$ is closed in $(F^1_{USCG} (X), H_{\rm end})$.

\end{tm}

\begin{proof}
By Theorem \ref{rcfegn},
  (\romannumeral1) $\Leftrightarrow$ (\romannumeral2).
  Obviously (\romannumeral3) $\Rightarrow$ (\romannumeral2).
We
shall complete the proof by showing that
 (\romannumeral1) $\Rightarrow$
 (\romannumeral3). To do this,
suppose that (\romannumeral1) is true.
To verify (\romannumeral3), from the equivalence of (\romannumeral1) and (\romannumeral2),
we only need to
show that $U(\al)$ is closed in $(X,d)$ for each $\al\in (0,1]$.

Let $\al\in (0,1]$
and
let $\{x_n\}$ be a sequence in $U(\al)$ with $x_n \to x$.
Suppose that
$x_n \in [u_n]_\al$ and $u_n \in  U$ for $n=1,2,\ldots$.
Then
there exists a subsequence $\{u_{n_k}\}$ of $\{u_n\}$ and $u\in U$
such that $H_{\rm end}(u_{n_k}, u) \to 0$.
So
by Remark \ref{hkr} $\lim_{n\to \infty}^{(\Gamma)}  u_{n_k} = u $ and therefore by Theorem \ref{Gclnre},
$\limsup_{n\to \infty}[u_{n_k}]_\alpha
\subseteq
 [u]_\alpha$.
Hence
$x\in [u]_\al$, and thus $x\in U(\al)$.

We can also show $x\in  [u]_\al \subseteq U(\al)$ in the following way.
From Theorem \ref{hepc}, $H([u_{n_k}]_\al, [u]_\al) \rightarrow 0$ holds for $\al\in (0,1)\setminus P_0(u)$.
If $\al \in (0,1) \setminus P_0(u)$, then
$x\in [u]_\al$.
If $\al \in \{1\} \cup P_0(u)$, then for all $\beta\in  (0,\al) \setminus P_0(u)$, $x\in [u]_\beta$.
Thus
$x\in [u]_\al$.
\end{proof}

\subsection{Characterizations of compactness in $(P^1_{USCB} (X), H_{\rm send})$ and $(F^1_{USCB} (X), H_{\rm send})$}\label{pfc}

In this
subsection,
we give
 the characterizations of totally bounded sets, relatively compact sets and compact sets
in
$(P^1_{USCB} (X), H_{\rm send})$.
Then,
by treating
$(F^1_{USCB} (X), H_{\rm send})$ as a subspace of $(P^1_{USCB} (X), H_{\rm send})$,
 we give the characterizations of totally bounded sets and compact sets
in $(F^1_{USCB} (X), H_{\rm send})$.
The characterization of
relatively compact sets
in
$(F^1_{USCB} (X), H_{\rm send})$ has already been given in \cite{greco}.

Suppose that
$U$ is a subset of $P^1_{USC} (X)$ and $\al\in [0,1]$.
For
writing convenience,
we denote
\begin{itemize}
  \item  $U(\al):= \bigcup_{u\in U} \langle u \rangle_\al$, and

\item  $U_\al : =  \{\langle u \rangle_\al: u \in U\}$.
\end{itemize}

\begin{tm} \label{tbpu}
   Suppose that $U$ is a subset of $P^1_{USCB} (X)$. Then $U$ is totally bounded in $(P^1_{USCB} (X), H_{\rm send})$
if and only if
$U(0)$ is totally bounded in $(X,d)$.
\end{tm}

\begin{proof}
\textbf{\emph{ Necessity}}. \
 Suppose that $U$ is totally bounded. By clause (\romannumeral2) of Theorem \ref{pseu},
$U_0 $ is  totally bounded in $(K(X), H)$.
 From Theorem \ref{tbe}, this is equivalent to
$U(0)$ is totally bounded in $(X,d)$.

\textbf{\emph{Sufficiency}}. \
Suppose that $U(0)$ is totally bounded.
To show that $U$ is totally bounded in $(P^1_{USCB} (X), H_{\rm send})$, we
only need to prove
that each sequence in $U$ has a Cauchy subsequence with respect to $H_{\rm send}$.

Let
$\{u_n\}$ be a sequence in $U$. Note that $U(\al)$ is totally bounded for each $\al \in [0,1]$.
Then by Theorem \ref{tbfegn},
 $\{\overleftarrow{u_n}\}$ has a Cauchy subsequence $\{v_n\}$ in $(F^1_{USCB} (X), H_{\rm end})$.
From Theorem \ref{tbe}, $\{v_n\}$ has a subsequence $\{w_n\}$
such that
$\{ [w_n]_0 \}$
is a Cauchy sequence
in $(K(X), H)$.
 Thus by clauses (\romannumeral3) of Theorem \ref{pseu},
 $\{w_n\}$ is a Cauchy sequence in $(P^1_{USCB} (X), H_{\rm send})$.
\end{proof}

\begin{tm} \label{tbfe}
   Suppose that $U$ is a subset of $F^1_{USCB} (X)$. Then $U$ is totally bounded in $(F^1_{USCB} (X), H_{\rm send})$
if and only if
$U(0)$ is totally bounded in $(X,d)$.
\end{tm}

\begin{proof}
Note that $U$ is totally bounded in $(F^1_{USCB} (X), H_{\rm send})$
 if and only if
$\overrightarrow{U}$ is totally bounded in $(P^1_{USCB} (X), H_{\rm send})$,
and that
$U(0) = \overrightarrow{U}(0)$.
So the desired result follows from
Theorem \ref{tbpu}.
\end{proof}

\begin{tm}  \label{rcgu}
   Suppose that $U$ is a subset of $P^1_{USCB} (X)$. Then $U$ is relatively compact in $(P^1_{USCB} (X), H_{\rm send})$
if and only if
$U(0)$ is relatively compact in $X$.
\end{tm}

\begin{proof}
 \emph{\textbf{ Necessity}}. \  Suppose that $U$ is relatively compact. Then by clause (\romannumeral2) of Theorem \ref{pseu},
$U_0 $ is  relatively compact in $(K(X), H)$.
 By Theorem \ref{rce},
$U(0)$ is  relatively compact in $X$.

\textbf{\emph{Sufficiency}}. \ To prove that $U$ is relatively compact, it suffices
to
show that each sequence in $U$
has a convergent subsequence
in $(P^1_{USCB} (X), H_{\rm send})$.

Let $\{u_n\}$ be a sequence in $U$.
Since $U(0)$ is relatively compact in $X$, then $U(\al)$ is relatively compact in $X$ for each $\al \in [0,1]$.
By Theorems \ref{rcfegn} and \ref{rce},
there is a subsequence $\{u_{n_k} \}$ of $\{u_n\}$,
 a $u\in F^1_{USCG} (X)$ and a $u_0 \in K(X)$
such that
$H_{\rm end} (\overleftarrow{u_{n_k}}, u) \to 0$ and $H(\langle u_{n_k}\rangle_0, u_0) \to 0$.

By Theorem \ref{hkg} and Remark \ref{hkr}, $\lim_{k\to \infty}^{(\Gamma)}  \overleftarrow{u_{n_k}} = u $ and $\lim^{(K)}_{k\to\infty} \langle u_{n_k} \rangle_0 = u_0$.
By Remark \ref{sur}, $[u]_0 \subseteq \liminf_{k\to\infty}[\overleftarrow{u_{n_k}}]_0 \subseteq \liminf_{k\to\infty}\langle u_{n_k}\rangle_0 = \lim^{(K)}_{k\to\infty} \langle u_{n_k} \rangle_0 = u_0$.
So we can define
 $w \in P^1_{USCB}(X)$
 by putting
\[
\langle w \rangle_\al=\left\{
           \begin{array}{ll}
           [u]_\al, &  \al\in(0,1],
\\
             u_0,  & \al=0.
           \end{array}
         \right.
\]
Then $u=\overleftarrow{w}$, $H_{\rm end}(u_{n_k}, w) = H_{\rm end}(\overleftarrow{u_{n_k}}, u) \to 0$
and
$H(\langle u_{n_k}\rangle_0, \langle w\rangle_0) = H(\langle u_{n_k}\rangle_0, u_0) \to 0$ for $n=1,2,\ldots$
 Thus from (\romannumeral3) or (\romannumeral4) of Theorem \ref{pseu},
 $\{u_{n_k}\}$ converges to
 $w$ in $(P^1_{USCB} (X), H_{\rm send})$.
\end{proof}

\begin{itemize}
\item
$u \in F^1_{USC} (X)$ is said to be right-continuous at $0$ if $\lim_{\delta\to 0+} H([u]_\delta, [u]_0) =0$.

  \item
$U \subseteq F^1_{USC} (X)$ is said to be equi-right-continuous at $0$ if for each $\varepsilon>0$,
there is a $\delta\in(0,1]$
such that
$H([u]_\delta, [u]_0) < \varepsilon$ for all $u \in U$.

\end{itemize}

Let $\delta\in (0,1]$
and $\xi\in [0,\delta]$. Then
$H([u]_\xi, [u]_0) = H^*([u]_0, [u]_\xi)\leq  H^*([u]_0, [u]_\delta) = H([u]_\delta, [u]_0)$.
So
$U \subseteq F^1_{USC} (X)$ is equi-right-continuous at $0$ if and only if for each $\varepsilon>0$,
there is a $\delta\in (0,1]$
such that $H([u]_\xi, [u]_0) < \varepsilon$ for all $u \in U$
and $\xi\in [0,\delta]$.

Theorem \ref{rcfe} below is presented in \cite{greco}.

\begin{tm} \cite{greco}   \label{rcfe}
   Suppose that $U$ is a subset of $F^1_{USCB} (X)$. Then $U$ is relatively compact in $(F^1_{USCB} (X), H_{\rm send})$
if and only if
$U(0)$ is relatively compact in $X$ and $U$ is equi-right-continuous at $0$.
\end{tm}

$\overrightarrow{F^1_{USCB}(X)}$ need not be a closed set of $P^1_{USCB}(X)$.
For instance, $F^1_{USCB}(D)$ given in Example \ref{nce} is not a closed set of $P^1_{USCB}(D)$.
We can see that
$\overrightarrow{F^1_{USCB}(X)}$ is a closed set of $P^1_{USCB}(X)$
if and only if
$X$ has only one element.

For a set $U$ in $F^1_{USCB}(X)$, suppose that
\\
(a) \ $U$
is relatively compact in $(F^1_{USCB}(X), H_{\rm send})$;
\\
(b) \
 $\overrightarrow{U}$
is relatively compact in $(P^1_{USCB}(X), H_{\rm send})$;
\\
(c) \ The topological closure of $\overrightarrow{U}$ in $(P^1_{USCB}(X), H_{\rm send})$ is a subset of $\overrightarrow{F^1_{USCB}(X)}$.

Then
(a) holds if and only if (b) and (c) hold.

$\overrightarrow{F^1_{USCB}(X)}$ is closed in $P^1_{USCB}(X)$ if and only if for each set $U$ in $F^1_{USCB}(X)$,
(c) holds.

The following
Proposition \ref{psf} illustrates the role of the condition ``$U$ is equi-right-continuous at $0$''
in
the characterization of the relative compactness for a set $U$ in $(F^1_{USCB}(X), H_{\rm send})$ given in Theorem \ref{rcfe}.

\begin{pp} \label{psf}
  Let $U$ be a subset of $ F^1_{USCB}(X)$.
Suppose the following conditions (\romannumeral1), (\romannumeral2) and (\romannumeral3):
\\
(\romannumeral1) \
$U$ is relatively compact in $(F^1_{USCB}(X), H_{\rm send})$;
\\
(\romannumeral2) \
$U$ is equi-right-continuous at $0$;
\\
(\romannumeral3) \
The topological closure of $\overrightarrow{U}$ in $(P^1_{USCB}(X), H_{\rm send})$ is a subset of $\overrightarrow{F^1_{USCB}(X)}$.

Then the condition (\romannumeral1) implies the condition (\romannumeral2),
and the condition (\romannumeral2) implies the condition (\romannumeral3).
 If $\overrightarrow{U}$ is relatively compact in $(P^1_{USCB}(X), H_{\rm send})$,
then the conditions (\romannumeral1), (\romannumeral2) and (\romannumeral3) are equivalent to each other.

\end{pp}

\begin{proof}
By Theorem \ref{rcfe}, we have that
 (\romannumeral1)$\Rightarrow$(\romannumeral2).

Now we show that
(\romannumeral2)$\Rightarrow$(\romannumeral3).
Suppose that $\{\overrightarrow{u_n}\}$ is a sequence in $\overrightarrow{U}$ which converges to an element $u$ in $(P^1_{USCB}(X), H_{\rm send})$.
Then by Theorem \ref{pseu}(\romannumeral4),
$\lim_{n\to \infty} H(\langle \overrightarrow{u_n} \rangle_0, \langle u \rangle_0 )=0$;
that is,
$\lim_{n\to \infty} H(   [u_n]_0, \langle u \rangle_0 )=0$.

Let $\varepsilon>0$.
 Since $\overleftarrow{u}\in F^1_{USCB}(X)$, by Lemma \ref{gnc}(\romannumeral5),
$\lim_{\alpha\to 0+} H([\overleftarrow{u}]_\alpha, [\overleftarrow{u}]_0) =0$.
Note that
 $U$ is equi-right-continuous at $0$.
Then there is a $\delta>0$
such that for all $\al\in [0,\delta]$ and $n\in \mathbb{N}$,
$H([u_n]_0, [u_n]_{\al})\leq \varepsilon/3$ and
$H([\overleftarrow{u}]_{\al}, [\overleftarrow{u}]_0) \leq \varepsilon/3$.
By Theorem \ref{pseu}(\romannumeral4), $\{u_n\}$ converges to $\overleftarrow{u}$ in $(F^1_{USCB}(X), H_{\rm end})$.
Then by
Theorem \ref{aec}, $H([u_n]_\al, [\overleftarrow{u}]_\al) \rightarrow 0$ holds a.e. on $\al\in (0,1)$. So we can choose an $\al_0\in (0,\delta)$
with
$H([u_n]_{\al_0}, [\overleftarrow{u}]_{\al_0}) \rightarrow 0$.
Thus
 there is an $N$ such that
for all $n\geq N$, $H([u_n]_{\al_0}, [\overleftarrow{u}]_{\al_0}) \leq \varepsilon/3$, and
hence
$
H([u_n]_0, [\overleftarrow{u}]_0)  \leq H([u_n]_0, [u_n]_{\al_0})
+
H([u_n]_{\al_0}, [\overleftarrow{u}]_{\al_0}) + H([\overleftarrow{u}]_{\al_0}, [\overleftarrow{u}]_0 \leq \varepsilon$.
Since
 $\varepsilon>0$ is arbitrary, we have
$\lim_{n\to \infty} H([u_n]_0, [\overleftarrow{u}]_0 )=0$.

Thus $\langle u \rangle_0 = [\overleftarrow{u}]_0 $,
and hence by Proposition \ref{pfbe},
$u\in \overrightarrow{F^1_{USCB} (X)}$.
So (\romannumeral2)$\Rightarrow$(\romannumeral3).

If $\overrightarrow{U}$ is relatively compact in $(P^1_{USCB}(X), H_{\rm send})$,
then clearly
 (\romannumeral3)$\Rightarrow$(\romannumeral1),
and
thus
the conditions (\romannumeral1), (\romannumeral2) and (\romannumeral3)
are equivalent to each other.
\end{proof}

\begin{re}{\rm

For conditions (\romannumeral1), (\romannumeral2) and (\romannumeral3) in Proposition \ref{psf},
 (\romannumeral2)  does not imply (\romannumeral1);
 (\romannumeral3) does not imply (\romannumeral2).

Let $\{u_n\}$ be a sequence of fuzzy sets in $F^1_{USCB}(\mathbb{R})$ defined by
$u_n:=[0,n]_{F(\mathbb{R})}$,
$n=1,2,\ldots$.
Then
 $\{u_n\}$ is equi-right-continuous at $0$ but $\{u_n\}$ is not
relatively compact
in $(F^1_{USCB}(X), H_{\rm send})$. So (\romannumeral2)  does not imply (\romannumeral1).

Let $\{v_n\}$ be a sequence of fuzzy sets in $F^1_{USCB}(\mathbb{R})$ defined by
\[v_n(x)=
\left\{
  \begin{array}{ll}
1, & x\in [0,n], \\
1/n, & x\in [-n,0],\\
   0, & x\in \mathbb{R}\setminus  [-n, n],
  \end{array}
\right.
n=1,2,\ldots.
\]
Then $\{\overrightarrow{v_n}\}$ has no limit in $(P^1_{USCB}(X), H_{\rm send})$ and hence is closed
in $(P^1_{USCB}(X), H_{\rm send})$. However $\{v_n\}$ is not equi-right-continuous at $0$.
 So (\romannumeral3) does not imply (\romannumeral2).
}
\end{re}

\begin{re}
  {\rm
Theorem \ref{rcgu} and Proposition \ref{psf} imply Theorem \ref{rcfe}.
This is because
by Proposition \ref{psf} we can obtain that for a subset $U$ of $F^1_{USCB} (X)$,
$U$ is relatively compact in $(F^1_{USCB}(X), H_{\rm send})$
if and only if
$\overrightarrow{U}$ is relatively compact in $(P^1_{USCB}(X), H_{\rm send})$
and
$U$ is equi-right-continuous at $0$.

}
\end{re}

\begin{tm} \label{cgu}
   Suppose that $U$ is a subset of $P^1_{USCB} (X)$. Then the following
are
equivalent:
\\
(\romannumeral1)  $U$ is compact in $(P^1_{USCB} (X), H_{\rm send})$;
\\
(\romannumeral2)    $U$ is closed in $(P^1_{USCB} (X), H_{\rm send})$, and $U(0)$ is relatively compact in $(X,d)$;
\\
(\romannumeral3)
  $U$ is closed in $(P^1_{USCB} (X), H_{\rm send})$, and $U(0)$ is compact in $(X,d)$.
\end{tm}

\begin{proof}
By Theorem \ref{rcgu}, we obtain that (\romannumeral1)$\Leftrightarrow$(\romannumeral2).
Clearly (\romannumeral3)$\Rightarrow$(\romannumeral2).
We shall complete the proof by showing that
(\romannumeral1)$\Rightarrow$(\romannumeral3).
 To do this, suppose
 that $U$ is compact in $(P^1_{USCB} (X), H_{\rm send})$. Then $U$ is closed in $(P^1_{USCB} (X), H_{\rm send})$.
To verify (\romannumeral3), we only need to show that
$U(0)$ is compact in $(X, d)$.

By clause (\romannumeral2) of Theorem \ref{pseu},
$U_0$ is compact in $(K(X), H)$.
Thus by Lemma \ref{coms},
$U(0)=\bigcup\{D: D\in U_0\}$ is compact in $(X,d)$.
\end{proof}

\begin{tm} \label{cfe}
   Suppose that $U$ is a subset of $F^1_{USCB} (X)$. Then the following
are
equivalent:
\\
(\romannumeral1) $U$ is compact in $(F^1_{USCB} (X), H_{\rm send})$;
\\
(\romannumeral2)
  $U$ is closed in $(F^1_{USCB} (X), H_{\rm send})$, $U(0)$ is relatively compact in $X$
and
$U$ is equi-right-continuous at $0$;
\\
(\romannumeral3)
  $U$ is closed in $(F^1_{USCB} (X), H_{\rm send})$, $U(0)$ is compact in $X$
and
$U$ is equi-right-continuous at $0$.

\end{tm}

\begin{proof}
By Theorem \ref{rcfe}, we obtain that (\romannumeral1)$\Leftrightarrow$(\romannumeral2).
Clearly (\romannumeral3)$\Rightarrow$(\romannumeral2).
We shall complete the proof by showing that
(\romannumeral1)$\Rightarrow$(\romannumeral3). To do this, suppose
that $U$ is compact in $(F^1_{USCB} (X), H_{\rm send})$.
Since (\romannumeral1) implies (\romannumeral2),
to verify (\romannumeral3) we only need to show
that $U(0)$ is compact in $X$.

Note that
$U$ is compact in $(F^1_{USCB} (X), H_{\rm send})$ if and only if $\overrightarrow{U}$ is compact in $(P^1_{USCB} (X), H_{\rm send})$. Thus
by
Theorem \ref{cgu}, $U(0)=\overrightarrow{U}(0)$ is compact in $X$.
\end{proof}

\begin{re}
{\rm
For a subset $U$ of $F^1_{USCB} (X)$,
\\
(a) $U$ satisfies (\romannumeral1) of Theorem \ref{cfe}
if and only if
$\overrightarrow{U}$ satisfies (\romannumeral1) of Theorem \ref{cgu};
\\
(b)
$U$ satisfies (\romannumeral2) of Theorem \ref{cfe}
if and only if
$\overrightarrow{U}$ satisfies (\romannumeral2) of Theorem \ref{cgu};
\\
(c) $U$ satisfies (\romannumeral3) of Theorem \ref{cfe}
if and only if
$\overrightarrow{U}$ satisfies (\romannumeral3) of Theorem \ref{cgu}.

Clauses (b) and (c) can be obtained
by using
 Proposition \ref{psf} and Theorem \ref{rcgu}.
Theorem \ref{cgu} and clauses (a), (b) and (c) imply Theorem \ref{cfe}.

Since $U$ is compact in $(F^1_{USCB} (X), H_{\rm send})$ if and only if $\overrightarrow{U}$ is compact in $(P^1_{USCB} (X), H_{\rm send})$, we can use
Theorem \ref{cgu} to judge the compactness of a set $U$ in $(F^1_{USCB} (X), H_{\rm send})$.

}
\end{re}

\section{Completions
of $(F^1_{USCB} (X), H_{\rm send})$,
$(F^1_{USCG} (X), H_{\rm end})$ and $(F^1_{USCB} (X),  d_\infty)$
}

In this section,
we show that $(P^1_{USCB} (\widetilde{X}), H_{\rm send})$ is a completion of
$(F^1_{USCB} (X), H_{\rm send})$,
$(F^1_{USCG} (\widetilde{X}), H_{\rm end})$ is a completion of $(F^1_{USCB} (X), H_{\rm end})$,
and
  $(F^1_{USCB} (\widetilde{X}), d_\infty)$ is a completion of $(F^1_{USCB} (X),  d_\infty)$.

We can see that for $x,y\in (X,d)$,
\begin{gather}
  d(x,y)=d_\infty(\widehat{x}, \widehat{y})= d_p(\widehat{x}, \widehat{y})= H_{\rm send} (\widehat{x}, \widehat{y})=H_{\rm send} (\overrightarrow{\widehat{x}}, \overrightarrow{\widehat{y}}), \label{emc}
\\
   H_{\rm end} (\overrightarrow{\widehat{x}}, \overrightarrow{\widehat{y}}) =H_{\rm end} (\widehat{x}, \widehat{y}) = \min\{d(x, y), 1\}  \label{emcu}.
\end{gather}

\begin{tm} \label{comfegn} Let $(X,d)$ be a metric space. Then the following
are
equivalent:
\\
(\romannumeral1)   $(X,d)$ is complete;
\\
(\romannumeral2)
  $(F^1_{USCG} (X), H_{\rm end})$ is complete.
\end{tm}

\begin{proof}

(\romannumeral1) $\Rightarrow$ (\romannumeral2). \
Assume that (\romannumeral1) is true. To show that (\romannumeral2) is true,
we only need to show that each Cauchy sequence in $(F^1_{USCG} (X), H_{\rm end})$ is a convergent sequence in $(F^1_{USCG} (X), H_{\rm end})$.
Let $\{u_n\}$
 be a Cauchy sequence in $(F^1_{USCG} (X), H_{\rm end})$.
Then $U := \{u_n, n=1,2,\ldots\}$ is total bounded in $(F^1_{USCG} (X), H_{\rm end})$.
By
Theorem \ref{tbfegn}, for each $\al\in (0,1]$, $U(\al)$ is total bounded in $(X,d)$. Since $(X,d)$ is complete, it follows
that for each $\al\in (0,1]$, $U(\al)$ is relatively compact in $(X,d)$.
Thus by Theorem \ref{rcfegn}, $U$ is relatively compact in $(F^1_{USCG} (X), H_{\rm end})$.
So $\{u_n\}$ has a convergent subsequence in $(F^1_{USCG} (X), H_{\rm end})$,
and thus
 $\{u_n\}$ is convergent in
$(F^1_{USCG} (X), H_{\rm end})$.
So (\romannumeral2) is true.

(\romannumeral2) $\Rightarrow$ (\romannumeral1).
 \
Assume that (\romannumeral2) is true.
Let
 $\{x_n\}$ be a Cauchy sequence in $(X,d)$.
Then by \eqref{emcu},
$\{\widehat{x_n}\}$ is a Cauchy sequence in
 $(F^1_{USCG} (X), H_{\rm end})$.
So $\{\widehat{x_n}\}$ is a convergent sequence in
 $(F^1_{USCG} (X), H_{\rm end})$.
 Hence by Remark \ref{cue} (\romannumeral1),
  there exists an $x\in X$
such
that $H_{\rm end} (\widehat{x_n},\widehat{x})\to 0$.
Then by \eqref{emcu}, $d(x_n,x)\to 0$.
 Thus
 $\{x_n\}$ is a convergent sequence in $(X,d)$.
So $(X,d)$ is complete; that is, (\romannumeral1) is true.
\end{proof}

\begin{re}\label{cue}
{\rm
Define $\widehat{X} := \{\widehat{x}: x\in X\}$.
Then
$\overrightarrow{\widehat{X}} := \{\overrightarrow{\widehat{x}}: x\in X\}$.
\\
(\romannumeral1) If $S$ satisfies $\widehat{X}\subseteq S \subseteq F^1_{USC}(X)$, then $\widehat{X}$
is closed in $(S, H_{\rm end})$.
\\
 (\romannumeral2)\ (\romannumeral1) remains true if $H_{\rm end}$ is replaced by each one of $d_\infty$, $H_{\rm send}$ and $d^*_p$.
  \\
  (\romannumeral3) \  If $S$ satisfies $\overrightarrow{\widehat{X}}\subseteq S \subseteq P^1_{USC}(X)$, then $\overrightarrow{\widehat{X}}$
is closed in $(S, H_{\rm send})$.

We claim that
for $u\in F^1_{USC}(X)$, if for each $\al\in(0,1]$, $[u]_\al$ is a singleton,
then $u\in \widehat{X}$.
To show $u\in \widehat{X}$, we only need to show that $[u]_\al=[u]_\beta$ for all $\al,\beta \in (0,1]$. This is true since otherwise
there exists $\al,\beta$ such that $0<\al<\beta\leq 1$
and
$[u]_\al \supsetneqq [u]_\beta$, which
 contradicts that $[u]_\al$ and $[u]_\beta$ are singletons.

To show (\romannumeral1), we only need to show that $\widehat{X}$
is closed in $(F^1_{USC}(X), H_{\rm end})$.
 Let $x, p, q\in X$ and $\al\in [0,1]$.
Then
$\overline{d}((p, \al), {\rm end}\,\widehat{x})= \min\{\overline{d}((p, \al), (x,\al)), \al\} = \min\{d(p, x), \al\}$.
And $\max\{d(p,x), d(q,x)\} \geq \frac{1}{2} d(p,q)$ since $d(p,x) + d(q,x) \geq d(p,q)$.
Given $u \in F^1_{USC}(X)\setminus\widehat{X}$. Then there exist
$p_1,q_1\in X$ and $\al_1\in (0,1]$ such that $p_1 \not= q_1$ and $p_1, q_1 \in [u]_{\al_1}$.
Thus for each $\widehat{x}\in \widehat{X}$,
$
H_{\rm end}(u, \widehat{x}) \geq \max\{\overline{d}((p_1, \al_1), {\rm end}\,\widehat{x}),  \overline{d}((q_1, \al_1), {\rm end}\,\widehat{x})\}
 =\max\{\min\{d(p_1, x), \al_1\}, \min\{d(q_1, x), \al_1\}  \}
  \geq \min\{\frac{1}{2} d(p_1,q_1), {\al_1}\}.
$
This implies that $u$ is not in the closure of $\widehat{X}$
in $(F^1_{USC}(X), H_{\rm end})$.
Since $u \in F^1_{USC}(X)\setminus\widehat{X}$ is arbitrary,
 $\widehat{X}$
in closed in $(F^1_{USC}(X), H_{\rm end})$.
So (\romannumeral1) is true.

By \eqref{smr}, each one of the $d_\infty$ convergence and the $H_{\rm send}$ convergence implies the $H_{\rm end}$ convergence on $F^1_{USC} (X)$.
 Theorem \ref{gdpn} says that the $d_p^*$ convergence implies the $H_{\rm end}$ convergence on $F^1_{USC} (X)$.
So (\romannumeral2) follows from (\romannumeral1).

To show (\romannumeral3), we only need to show that
$\overrightarrow{\widehat{X}}$
is closed in $(P^1_{USC}(X), H_{\rm send})$.
Given $u \in P^1_{USC}(X)\setminus\overrightarrow{\widehat{X}}$.
If $u\in P^1_{USC}(X) \setminus \overrightarrow{F^1_{USC}(X)}$. Then $\langle u\rangle_0$ has at least two points. If $u\in \overrightarrow{F^1_{USC}(X)} \setminus \overrightarrow{\widehat{X}}$, then
there exists $\al\in (0,1]$ such that $\langle u\rangle_\al$ has at least two points.
So there exist
$p_1,q_1\in X$ and $\al_1\in [0,1]$ such that $p_1 \not= q_1$ and $p_1, q_1 \in \langle u\rangle_{\al_1}$.
Thus for each $\overrightarrow{\widehat{x}}\in \overrightarrow{\widehat{X}}$,
$
H_{\rm send}(u, \overrightarrow{\widehat{x}}) \geq \max\{ \overline{d}((p_1,\al_1), \overrightarrow{\widehat{x}}), \overline{d}((q_1,\al_1), \overrightarrow{\widehat{x}})\}=
 \max\{d(p_1, x), d(q_1, x)\}
  \geq \frac{1}{2} d(p_1,q_1).
$
This implies that $u$ is not in the closure of $\overrightarrow{\widehat{X}}$
in $(P^1_{USC}(X), H_{\rm send})$.
Since $u \in P^1_{USC}(X)\setminus\overrightarrow{\widehat{X}}$ is arbitrary,
 $\overrightarrow{\widehat{X}}$
in closed in $(P^1_{USC}(X), H_{\rm send})$. So (\romannumeral3) is true.

(\romannumeral2) $\Rightarrow$ (\romannumeral1) in Theorem \ref{comfegn} can also be shown as follows.
$(X, d)$ is complete if and only if $(X, d^*)$ is complete, where $d^*(x,y) = \min\{d(x,y),1\}$ for $x,y \in X$.
Note that $H_{\rm end} (\widehat{x},\,  \widehat{y}) = d^*(x,y)$.
So the desired result follows
from the fact that
$(X, d^*)$ is isometric to the closed subspace
$(\widehat{X}, H_{\rm end})$
of $(F^1_{USCG} (X), H_{\rm end})$.
}
\end{re}

\begin{pp} \label{fcm}
Let $(X,d)$ be a metric space. Then
the following
are
equivalent:
\\
(\romannumeral1) \ $(X,d)$ is complete;
\\
(\romannumeral2) \
  $(F^1_{USCB} (X), d_\infty)$ is complete.

\end{pp}

\begin{proof}
 By \eqref{emc},
 $(X,d)$ is isometric to $(\widehat{X}, d_\infty)$, which by Remark \ref{cue} (\romannumeral2) is a closed subspace
of   $(F^1_{USCB} (X), d_\infty)$. So (\romannumeral2)$\Rightarrow$(\romannumeral1) is proved.

To show (\romannumeral1)$\Rightarrow$(\romannumeral2),
suppose that $(X,d)$ is complete.
Let $\{u_n\}$ be a Cauchy sequence in  $(F^1_{USCB} (X), d_\infty)$.
Then for each $\al\in [0,1]$,
$\{ [u_n]_\al\}$ is a Cauchy sequence in $(K(X), H)$, and hence by Theorem \ref{bfc}, there is a $u(\al)\in K(X)$ such that
$H([u_n]_\al,  u(\al)) \to 0$.
 We
have that:
\\
(a-1) \ For $0\leq \beta \leq \alpha \leq 1$, $u(\al) \subseteq u(\beta)$ since by Theorem \ref{hkg}, $u(\al)=\lim^{(K)}_{n\to\infty}[u_n]_\al \subseteq \lim^{(K)}_{n\to\infty}[u_n]_\beta = u(\beta)$;
\\
(a-2) \  $\lim_{n\to\infty}\sup_{\al\in [0,1]}H([u_n]_\al, u(\al)) =0$; that is, $\{H([u_n]_\al, u(\al))\}$ converges uniformly to $0$ on $\al\in [0,1]$;
\\
(a-3) \
For each $\al\in (0,1]$, $\lim_{\beta\to \al-} H(u(\beta), u(\alpha)) = 0$;
\\
(a-4) \
$\lim_{\gamma\to 0+} H(u(\gamma), u(0)) = 0$.

Given $\varepsilon>0$. As $\{u_n\}$ is a Cauchy sequence in  $(F^1_{USCB} (X), d_\infty)$, there is an $N (\varepsilon) \in \mathbb{N}$ such that
for all $n,m\geq N$, $\sup_{\al\in [0,1]}  H([u_n]_\al, [u_m]_\al)\leq \varepsilon$.
Let $\al\in [0,1]$ and $n\geq N$. Then for each $l\in \mathbb{N}$,
 $H([u_n]_\al, u(\al))\leq  H([ u_n]_\al, [u_{n+l}]_\al) +  H([u_{n+l}]_\al, u(\al)) \leq \varepsilon + H([u_{n+l}]_\al, u(\al))$.
So $H([u_n]_\al, u(\al))\leq \varepsilon$ as
 $\lim_{l\to\infty} H([u_{n+l}]_\al,  u(\al)) = 0$.
As $\al\in [0,1]$ and $n\geq N$ are arbitrary, we obtain that $\sup_{\al\in [0,1]}  H([u_n]_\al, u(\al)) \leq \varepsilon$
for all $n\geq N$.
Since $\varepsilon>0$ is arbitrary, we have that
 (a-2) is true.

Let $\al\in (0,1]$. Given $\varepsilon>0$, by (a-2), there
is an $N$ such that for $n\geq N$,
$\sup_{\xi\in [0,1]}H([u_n]_\xi, u(\xi)) \leq \varepsilon/3$.
By Lemma \ref{gnc}(\rmn1),
$\lim_{\beta\to \al-} H([u_N]_\al, [u_N]_\beta) = 0$.
Then there is a $\delta>0$ such that
for all $\beta \in [\al-\delta, \al]$,
$H([u_N]_\al, [u_N]_\beta) \leq \varepsilon/3$.
Thus for all $\beta \in [\al-\delta, \al]$,
$H(u(\alpha), u(\beta)) \leq
H(u(\alpha), [u_N]_\alpha) + H([u_N]_\alpha, [u_N]_\beta) + H([u_N]_\beta, u(\beta)) \leq \varepsilon$.
Since
$\varepsilon>0$ is arbitrary, it follows that (a-3) is true.
Similarly, by using (a-2) and the fact that for $n\in \mathbb{N}$,
$\lim_{\gamma\to 0+} H([u_n]_\gamma, [u_n]_0) = 0$,
we can show that
 (a-4) is true.

If there is a $u\in F^1_{USCB} (X)$ such that $[u]_\al=u(\al)$ for all $\al\in [0,1]$, then by (a-2),
$d_\infty(u_n, u) \to 0$, and so the proof is complete.
As $\{u(\al):\al\in [0,1]\} \subseteq K(X)$,
to prove the existence of a such $u\in F^1_{USCB} (X)$,
by Proposition \ref{repm},
we only need to show that $\{u(\al), \al\in [0,1]\}$ has the following properties:
\\
(b) for each $\al\in (0,1]$, $u(\al) =  \bigcap_{\beta < \al} u(\beta)$, and
(c) $u(0) = \overline{\bigcup_{\gamma>0} u(\gamma)}$.

By Remark \ref{crg} and (a-1), (b) and (c) follow
from (a-3) and (a-4), respectively. This
completes the proof.

(b) and (c) can also be shown in the following way.
From (a-1), $\overline{\bigcup_{\gamma>0} u(\gamma)} \subseteq u(0)$ and then $\overline{\bigcup_{\gamma>0} u(\gamma)} \in K(X)$.
By Lemma \ref{c},
for each $\al\in (0,1]$,
$\lim_{\beta\to \al-} H(u(\beta), \bigcap_{\beta < \al} u(\beta)) = 0$,
and
$\lim_{\gamma\to 0+} H(u(\gamma), \overline{\bigcup_{\gamma>0} u(\gamma)}) = 0$ (see the proof of Lemma \ref{gnc}(\rmn1)(\rmn5)).
Combined with (a-3) and (a-4), we obtain that
 (b) and (c) are true.
\end{proof}

Even if $(X,d)$ is complete,
$(F^1_{USCB} (X), H_{\rm send})$ need not be complete.
See Example \ref{nce} below.

\begin{eap} \label{nce}
  {\rm
We use $\rho_1$ to denote the induced metric on \{0,1\} by the Euclidean metric on $\mathbb{R}$.
Let $D$ denote the metric space $(\{0,1\}, \rho_1)$.
Then $D$
is a subspace of $\mathbb{R}$ and $D$ is complete.
For $n=1,2,\ldots$, let $u_n\in F^1_{USCB} (D)$ be defined by
\[
u_n(x)=\left\{
         \begin{array}{ll}
           1, & x=0, \\
           1/n, & x=1.
         \end{array}
       \right.
\]
Define $u\in P^1_{USCB} (D) \setminus \overrightarrow{F^1_{USCB} (D)}$ by putting
\[
\langle u \rangle_\al =\left\{
         \begin{array}{ll}
           \{0\}, & \al\in (0,1], \\
           \{0,1\}, & \al=0.
         \end{array}
       \right.
\]
Then $\{\overrightarrow{u_n}\}$ converges
to $u$ in $(P^1_{USCB} (D), H_{\rm send})$.
Thus
$\{u_n\}$ is a Cauchy sequence in $(F^1_{USCB} (D), H_{\rm send})$ and
 has no limit in $(F^1_{USCB} (D), H_{\rm send})$.
So
$(F^1_{USCB} (D), H_{\rm send})$ is not complete.

}
\end{eap}

We can see that $(F^1_{USCB} (X), H_{\rm send})$ is complete if and only if $X$ has only one element.

Theorem \ref{sce} below discusses the completeness of $(P^1_{USCB} (X), H_{\rm send})$
and then
Theorem \ref{scom} below gives
the completion of $(F^1_{USCB} (X), H_{\rm send})$.

\begin{tm} \label{sce} \ Let $(X,d)$ be a metric space. Then the following
are
equivalent:
\\
(\romannumeral1)  $X$ is complete;
\\
(\romannumeral2)
  $(P^1_{USCB} (X), H_{\rm send})$ is complete.

\end{tm}

\begin{proof}

(\romannumeral1) $\Rightarrow$ (\romannumeral2). \
Let $\{u_n\}$
 be a Cauchy sequence in $(P^1_{USCB} (X), H_{\rm send})$.
Then $\{\overleftarrow{u_n}\} \subseteq F^1_{USCB} (X)$,
 and by (\romannumeral1) of Theorem \ref{pseu},
$H_{\rm end}(\overleftarrow{u_n}, \overleftarrow{u_m}) = H_{\rm end}(u_n, u_m) \leq  H_{\rm send}(u_n, u_m)$
for $n,m=1,2,\ldots$.
Hence
 $\{\overleftarrow{u_n}\}$
  is a Cauchy sequence in $(F^1_{USCG} (X), H_{\rm end})$.
From Theorem \ref{comfegn},
there is a $u \in F^1_{USCG} (X)$ such that
  $\{\overleftarrow{u_n}\}$ converges to $u$ in $(F^1_{USCG} (X), H_{\rm end}$).

 By (\romannumeral2) of Theorem \ref{pseu}, $ H(\langle u_n \rangle_0, \langle u_m \rangle_0) \leq H_{\rm send}(u_n, u_m)$ for $n,m=1,2,\ldots$.
So
 $\{ \langle u_n \rangle_0 \}$
is a Cauchy sequence in
 $(K(X), H)$.
From Theorem \ref{bfc},
there is a $u_0 \in K(X)$
such that $\{ \langle u_n \rangle_0 \}$ converges to $u_0$ in
 $(K(X), H)$.

By Theorem \ref{hkg} and Remark \ref{hkr}, $\lim_{n\to \infty}^{(\Gamma)}  \overleftarrow{u_n} = u $ and $\lim^{(K)}_{n\to\infty} \langle u_n \rangle_0 = u_0$.
By Remark \ref{sur}, $[u]_0 \subseteq \liminf_{n\to\infty}[\overleftarrow{u_n}]_0 \subseteq \liminf_{n\to\infty}\langle u_n\rangle_0 = \lim^{(K)}_{n\to\infty} \langle u_n \rangle_0 = u_0$.
So we can define $w \in P^1_{USCB}(X)$
by putting
\[
\langle w \rangle_\al=\left\{
           \begin{array}{ll}
           [u]_\al, &  \al>0,
\\
             u_0,  & \al=0.
           \end{array}
         \right.
\]
Then $u=\overleftarrow{w}$, $H_{\rm end}(u_n, w) = H_{\rm end}(\overleftarrow{u_n}, u)$
and
$H(\langle u_n\rangle_0, \langle w\rangle_0) = H(\langle u_n\rangle_0, u_0)$ for $n=1,2,\ldots$
 Thus from (\romannumeral3) or (\romannumeral4) of Theorem \ref{pseu},
 $\{u_n\}$ converges to
 $w$ in $(P^1_{USCB} (X), H_{\rm send})$.

(\romannumeral2) $\Rightarrow$ (\romannumeral1). Note that $d(x,y) = H_{\rm send} (\overrightarrow{\widehat{x}},\,  \overrightarrow{\widehat{y}})$.
So the desired result
follows from the fact
that $(X,d)$ is isometric to$(\overrightarrow{\widehat{X}}, H_{\rm send})$,
which by Remark \ref{cue} (\romannumeral3) is
a closed subspace of
 $(P^1_{USCB} (X), H_{\rm send})$.
\end{proof}

\begin{tm} \label{scom}
 $(P^1_{USCB} (\widetilde{X}), H_{\rm send})$ is a completion of $(F^1_{USCB} (X), H_{\rm send})$.
\end{tm}

\begin{proof}
From Theorem \ref{sce}, $(P^1_{USCB} (\widetilde{X}), H_{\rm send})$ is complete.
To show that
$(P^1_{USCB} (\widetilde{X}), H_{\rm send})$ is a completion of $(F^1_{USCB} (X), H_{\rm send})$,
we only need to verify
that $\overrightarrow{F^1_{USCB} (X)}$ is dense in $(P^1_{USCB} (\widetilde{X}), H_{\rm send})$.

We claim the following affirmations (a) and (b):
\begin{description}
  \item[(a)] For each $u \in P^1_{USCB} (\widetilde{X})$ and each $\varepsilon>0$,
there exists a $v \in F^1_{USCB} (\widetilde{X})$
such that
$H_{\rm send} (u, \overrightarrow{v}) \leq \varepsilon$.

  \item[(b)] For each $v \in F^1_{USCB} (\widetilde{X})$ and each $\varepsilon>0$,
there exists a $w  \in F^1_{USCB} (X) $
such that $d_\infty (v, w) \leq \varepsilon$ (so by \eqref{smr},
$H_{\rm end} (v, w) \leq H_{\rm send} (v, w) \leq \varepsilon$, and by \eqref {spr}, $d_p (v, w) = d^*_p (v, w)\leq \varepsilon$).
\end{description}

If affirmations (a) and (b) are true, then
for each $u \in P^1_{USCB} (\widetilde{X})$ and each $\varepsilon>0$,
there is a $v \in F^1_{USCB} (\widetilde{X})$ and
$w\in F^1_{USCB} (X)$
such that
$H_{\rm send}(u, \overrightarrow{v}) \leq \varepsilon/2$
and
 $ H_{\rm send} (\overrightarrow{v}, \overrightarrow{w}) = H_{\rm send} (v, w) \leq \varepsilon/2 $.
Thus
$H_{\rm send}(u, \overrightarrow{w}) \leq H_{\rm send}(u, \overrightarrow{v})
+
H_{\rm send} (\overrightarrow{v}, \overrightarrow{w})
\leq \varepsilon $.
So
for each $u \in P^1_{USCB} (\widetilde{X})$ and each $\varepsilon>0$,
there is a
$w\in F^1_{USCB} (X)$
such that
$H_{\rm send}(u, \overrightarrow{w}) \leq \varepsilon$.
This means
that
$\overrightarrow{F^1_{USCB} (X)}$ is dense in $(P^1_{USCB} (\widetilde{X}), H_{\rm send})$.
So to show the desired result,
it suffices to
prove affirmations (a) and (b).

Let $u \in P^1_{USCB} (\widetilde{X})$ and $\varepsilon>0$. Define $u_\varepsilon \in F^1_{USCB} (\widetilde{X})$ by putting
  \[
  [u_\varepsilon]_{\al} =\left\{
         \begin{array}{ll}
           \langle u \rangle_\al , & \al \in (\varepsilon, 1], \\
           \langle u \rangle_0, & \al \in  [0,\varepsilon ].
         \end{array}
       \right.
  \]
Then
$H_{\rm send} (u, \overrightarrow{u_\varepsilon}) \leq \varepsilon$.
So affirmation (a) is proved.

Let $v \in F^1_{USCB} (\widetilde{X})$ and $\varepsilon>0$.
 We can choose
 a finite subset
$C_0$ of $X$ such that $H(C_0, [v]_0) < \varepsilon$ as follows.
Since $[v]_0\in K(\widetilde{X})$, there is a finite subset $A$ of $[v]_0$
such
that $H(A, [v]_0) < \varepsilon/2$.
As $X$ is dense in $(\widetilde{X}, \widetilde{d})$, there is a subset
$C_0$ of $X$ such that
$H(C_0, A)< \varepsilon/2$. Thus $H(C_0, [v]_0)\leq H(C_0, A) + H(A, [v]_0)   < \varepsilon$.

Define
\begin{equation}\label{caf}
  C_\al: =   \{x\in C_0:  d (x, [v]_\al ) \leq \varepsilon    \}, \ \al\in (0,1].
\end{equation}
We affirm that $\{C_\al: \al \in [0,1]\}$ has the following properties
\\
(\romannumeral1) \ $C_\al \not= \emptyset$ for all $\al\in [0,1]$.
\\
(\romannumeral2) \ $H( C_\al, [v]_\al ) \leq \varepsilon$ for all $\al\in [0,1]$.
\\
(\romannumeral3) \  $C_\al = \bigcap_{\beta<\al} C_\beta$ for all $\al\in (0,1]$.
\\
(\romannumeral4) \ $C_0 = \bigcup_{\al>0} C_\al = \overline{\bigcup_{\al>0} C_\al }$.

For each $y\in [v]_\al $,
there
 exists $z_y\in C_0$ such that
$d(y,z_y) = d(y, C_0) < \varepsilon$. Hence $z_y\in C_\al$ and thus $C_\al \not= \emptyset$.
So
(\romannumeral1) is true.

To show (\romannumeral2),
let $\al\in [0,1]$. If $\al=0$, then (\romannumeral2) is true since $H(C_0, [v]_0) < \varepsilon$.
Assume that $0<\al \leq 1$.
From \eqref{caf},
 $ H^* ( C_\al, [v]_\al ) \leq \varepsilon  $.
So to show $ H( C_\al, [v]_\al ) \leq \varepsilon  $,
 it suffices to show that
 $ H^* ( [v]_\al, C_\al) < \varepsilon  $.

Let $y\in [v]_\al$.
From $ d(y,z_y)  \geq d(y, C_\al) \geq d(y, C_0)$
and $ d(y,z_y)  = d(y, C_0)$,
we have that
$ d(y,z_y) = d(y, C_\al) = d(y, C_0)$.
Thus
\begin{align*}
  H^* ([v]_\al, C_\al) &= \sup_{y\in [v]_\al} d(y, C_\al)
  = \sup_{y\in [v]_\al} d(y, C_0)
\leq \sup_{y\in [v]_0} d(y, C_0) =  H^*([v]_0, C_0) \leq  H( C_0, [v]_0)< \varepsilon.
\end{align*}
So
(\romannumeral2) is proved.

To show (\romannumeral3),
let $\al_0 \in (0,1]$.
From the definition of $\{C_\al: \al\in [0,1]\}$,
we can see that $C_\xi \subseteq C_\eta $ for $0\leq \eta\leq \xi \leq 1$.
Thus $C_{\al_0} \subseteq \bigcap_{\beta<\al_0} C_\beta$.

By Lemma \ref{gnc} (\romannumeral1),
 $\lim_{\beta \to \al_0-} H([v]_{\al_0}, [v]_\beta ) = 0$.
Note that for each $x\in X$ and $\xi, \eta\in [0,1]$,
$|d(x, [v]_\xi )  - d(x, [v]_\eta )|  \leq H([v]_\xi, [v]_\eta)$.
Thus for each $x\in X$,
$d(x, [v]_{\al_0}) = \lim_{\beta \to \al_0-} d(x, [v]_\beta ) $.
Let $x\in \bigcap_{\beta<\al_0} C_\beta$.
Then
$d(x, [v]_{\al_0}) = \lim_{\beta \to \al_0-} d(x, [v]_\beta ) \leq \varepsilon$, which means that $x\in C_{\al_0}$.
Since $x \in \bigcap_{\beta<\al_0} C_\beta$ is arbitrary,
we have
$\bigcap_{\beta<\al_0} C_\beta \subseteq C_{\al_0}  $.
Hence
$C_{\al_0} = \bigcap_{\beta<\al_0} C_\beta$.
So (\romannumeral3) is true.

Since $C_\al \subseteq C_0$ for all $\al\in [0,1]$ and $C_0$ is finite,
it follows that
$\bigcup_{\al>0} C_\al = \overline{\bigcup_{\al>0} C_\al } \subseteq C_0$.
So to show (\romannumeral4), we only need to show
that
$C_0 \subseteq \bigcup_{\al>0} C_\al$.

Since
$[v]_0 =  \overline{\cup_{\al>0} [v]_\al}$,
it follows that for each $x\in X$,
$d(x, [v]_0) = d(x, \cup_{\al>0} [v]_\al) = \inf_{\al>0}  d(x, [v]_\al)$.
Let $x \in C_0$. Then $d(x, [v]_0) < \varepsilon$.
So
there
exists $\al>0$ such that
$d(x, [v]_\al) < \varepsilon$,
which means that
$x\in C_\al$.
Thus $C_0 \subseteq \bigcup_{\al>0} C_\al$ and (\romannumeral4) is proved.

Now,
define $w \in F(X)$ by putting
$
  [w]_{\al} = C_\al \mbox{ for all } \al \in [0,1].
$
Then
by (\romannumeral1), (\romannumeral3) and (\romannumeral4),
$w \in F^1_{USCB} (X)$.
From
(\romannumeral2), we have $d_\infty (v, w) \leq \varepsilon$, and then
 affirmation (b) is proved.
This completes the proof.
\end{proof}

\begin{pp}\label{supc}
  $(F^1_{USCB} (\widetilde{X}), d_\infty)$ is a completion of $(F^1_{USCB} (X),  d_\infty)$.
\end{pp}

\begin{proof}
By Proposition \ref{fcm},
$(F^1_{USCB} (\widetilde{X}), d_\infty)$ is complete.
 So from affirmation (b)
 in the proof of Theorem \ref{scom},
we have the desired result.
\end{proof}

\begin{tl} \label{pcom}
 $(P^1_{USCB} (\widetilde{X}), H_{\rm send})$ is a completion of $(P^1_{USCB} (X), H_{\rm send})$.
\end{tl}

\begin{proof}
Note that
  $(F^1_{USCB} (X), H_{\rm send})$ can be seen as a subspace of $(P^1_{USCB} (X), H_{\rm send})$,
and
$(P^1_{USCB} (X), H_{\rm send})$ is a subspace of $(P^1_{USCB} (\widetilde{X}), H_{\rm send})$.
So the desired result follows from Theorem \ref{scom}.
\end{proof}

\begin{tm} \label{ecom}
 $(F^1_{USCG} (\widetilde{X}), H_{\rm end})$ is a completion of $(F^1_{USCB} (X), H_{\rm end})$.

\end{tm}

\begin{proof}
  From Theorem \ref{comfegn}, $(F^1_{USCG} (\widetilde{X}), H_{\rm end})$ is complete.
To
 prove the desired result, it suffices to
show
that $F^1_{USCB} (X)$ is dense in $(F^1_{USCG} (\widetilde{X}), H_{\rm end})$.
By
affirmation (b) in the proof of Theorem \ref{scom},
to verify this it is enough to show
that for each $u \in F^1_{USCG} (\widetilde{X})$ and each $\varepsilon>0$,
there is a $v \in  F^1_{USCB} (\widetilde{X})$ such that $H_{\rm end} (u,v) \leq \varepsilon$.

Let $u \in F^1_{USCG} (\widetilde{X})$ and $\varepsilon>0$.
Define $u^\varepsilon \in  F^1_{USCB} (\widetilde{X})$ by putting
  \[
  [u^\varepsilon]_{\al} =\left\{
         \begin{array}{ll}
           [u]_\al , & \al \in (\varepsilon, 1], \\
          \mbox{} [u]_\varepsilon, & \al \in  [0,\varepsilon ].
         \end{array}
       \right.
  \]
Then
$H_{\rm end}  (u, u^\varepsilon) \leq \varepsilon$.
\end{proof}

\begin{tl} \label{fgecom}
 $(F^1_{USCG} (\widetilde{X}), H_{\rm end})$ is a completion of $(F^1_{USCG} (X), H_{\rm end})$.

\end{tl}

\begin{proof}
  Since $F^1_{USCB} (X) \subseteq F^1_{USCG} (X) \subseteq F^1_{USCG} (\widetilde{X})$,
the desired result
follows from
Theorem \ref{ecom}.
\end{proof}

\begin{re}
  {\rm

If $X$ is complete, then by Theorem \ref{scom}, $(P^1_{USCB} (X), H_{\rm send})$ is a completion of $(F^1_{USCB} (X), H_{\rm send})$, and thus
$(P^1_{USCB} (X), H_{\rm send})$ is complete.
So
Theorem \ref{scom} implies that $(P^1_{USCB} (X), H_{\rm send})$
is complete when $X$ is complete.
Similarly,
Proposition \ref{supc} implies that $(F^1_{USCB} (X), d_\infty)$ is complete when $X$ is complete.
Theorem \ref{ecom} implies that
 $(F^1_{USCG} (X), H_{\rm end})$
is complete when $X$ is complete.

}
\end{re}

\section{Conclusions}

In this paper, we discuss the properties and relations of
$H_{\rm end}$ metric and $H_{\rm send}$ metric on fuzzy sets in a metric space $X$.

To aid discussion, we introduce the sets $P^1_{USC}(X)$ and $P^1_{USCB}(X)$. $P^1_{USCB}(X)$ is a subset of $P^1_{USC}(X)$.
The
$F^1_{USC}(X)$ and $F^1_{USCB}(X)$ can be viewed as the subsets of $P^1_{USC}(X)$ and $P^1_{USCB}(X)$, respectively.
We
define the
$H_{\rm send}$ distance and the $H_{\rm end}$ distance on $P^1_{USC}(X)$,
and give the relations among
the $H_{\rm send}$ distance, the $H_{\rm end}$ distance and the Kuratowski convergence on $P^1_{USC}(X)$.
Then, as corollaries, we obtain the
 relations among
the $H_{\rm send}$ metric, the $H_{\rm end}$ metric and the $\Gamma$-convergence on $F^1_{USC}(X)$.

We
give the
level characterizations of $H_{\rm end}$ convergence and $\Gamma$-convergence on $F^1_{USC}(X)$.
By using the above results including the level characterizations of the $H_{\rm end}$ convergence,
we give
 the relations among the $H_{\rm end}$ metric, the $H_{\rm send}$ metric and
the $d_p^*$ metric.

Based on above results,
we
give characterizations of compactness and completions of two kinds
of
fuzzy set spaces $(F^1_{USCG} (X), H_{\rm end})$
and
$(F^1_{USCB} (X), H_{\rm send})$, respectively.
We also investigate characterizations of compactness and completions
of
$(P^1_{USCB} (X), H_{\rm send})$.
$(F^1_{USCB} (X), H_{\rm send})$ can be treated as a subspace of $(P^1_{USCB} (X), H_{\rm send})$.

Note
 that $\mathbb{R}^m$ is complete and
 for a set $V$ in $\mathbb{R}^m$, the following are equivalent:
(\romannumeral1 ) $V$ is bounded;
(\romannumeral2) $V$ is totally bounded; (\romannumeral3) $V$
is relatively compact.
We can obtain the characterizations of compactness and completions
of
$(F^1_{USCB} (\mathbb{R}^m), H_{\rm send})$, $(F^1_{USCG} (\mathbb{R}^m), H_{\rm end})$ and $(P^1_{USCB} (\mathbb{R}^m), H_{\rm send})$
by using that of
$(F^1_{USCB} (X), H_{\rm send})$, $(F^1_{USCG} (X), H_{\rm end})$ and $(P^1_{USCB} (X), H_{\rm send})$ given in this paper.

The results in this paper
have potential applications in fuzzy set research involving
$H_{\rm end}$ metric and $H_{\rm send}$ metric.

\section*{Acknowledgement}

The author would like to thank the Area Editor and the three anonymous referees
for their invaluable comments and suggestions
which greatly improves the readability of this paper.

\end{document}